\documentclass[a4paper,11pt]{article}
\usepackage{amsmath,amsthm}
\usepackage{authblk}
\usepackage[style=numeric-comp,maxbibnames=99,bibencoding=utf8,firstinits=true]{biblatex}
\usepackage{booktabs}
\usepackage{cancel}
\usepackage[hmargin={26mm,26mm},vmargin={30mm,35mm}]{geometry}
\usepackage{multirow}
\usepackage{nicefrac}
\usepackage{paralist}
\usepackage[colorlinks, allcolors=blue]{hyperref}
\usepackage{subcaption}
\usepackage[compact,small]{titlesec}
\usepackage{tikz}
\usepackage{tikz-cd}
\usepackage{newtxmath}
\usepackage{newtxtext}

\bibliography{divdivsymm-ser}
\newcommand{\email}[1]{\href{mailto:#1}{#1}}

\allowdisplaybreaks

\numberwithin{equation}{section}

\newtheorem{theorem}{Theorem}
\newtheorem{proposition}[theorem]{Proposition}
\newtheorem{lemma}[theorem]{Lemma}

\theoremstyle{remark}
\newtheorem{remark}[theorem]{Remark}
\theoremstyle{definition}
\newtheorem{assumption}[theorem]{Assumption}

\newcommand{\st}{\,:\,}
\newcommand{\Real}{\mathbb{R}}

\newcommand{\Symm}{\mathbb{S}}

\DeclareRobustCommand{\bvec}[1]{\boldsymbol{#1}}
\let\ser\widehat

\newcommand{\svec}[1]{\ser{\bvec{#1}}}
\newcommand{\stens}[1]{\ser{\btens{#1}}}

\newcommand{\uvec}[1]{\underline{\bvec{#1}}}
\newcommand{\suvec}[1]{\underline{\ser{\bvec{#1}}}}
\newcommand{\cvec}[1]{\bvec{\mathcal{#1}}}

\DeclareRobustCommand{\btens}[1]{\boldsymbol{#1}}
\newcommand{\utens}[1]{\underline{\bvec{#1}}}
\newcommand{\sutens}[1]{\underline{\ser{\bvec{#1}}}}
\newcommand{\ctens}[1]{\bvec{\mathcal{#1}}}

\DeclareMathOperator{\card}{card}

\DeclareMathOperator{\tr}{tr}

\DeclareMathOperator{\GRAD}{\bf grad}
\DeclareMathOperator{\CURL}{\bf curl}
\DeclareMathOperator{\TCURL}{\bf Curl}
\DeclareMathOperator{\DIV}{div}
\DeclareMathOperator{\VDIV}{\bf div}
\DeclareMathOperator{\ROT}{rot}
\DeclareMathOperator{\VROT}{\bf rot}
\DeclareMathOperator{\SYM}{sym}
\DeclareMathOperator{\ANTI}{\bf Anti}
\DeclareMathOperator{\HESS}{\bf hess}
\DeclareMathOperator{\TR}{tr}

\newcommand{\sym}{{\rm sym}}

\newcommand{\RT}[1]{\boldsymbol{\mathcal{RT}}^{#1}}

\newcommand{\Hdivdiv}[2]{\bvec{H}(\DIV\VDIV,#1;#2)}

\newcommand{\compl}{{\rm c}}

\newcommand{\Poly}[2][]{\mathcal{P}_{#1}^{#2}}
\newcommand{\vPoly}[2][]{\cvec{P}_{#1}^{#2}}
\newcommand{\tPoly}[2][]{\ctens{P}_{#1}^{#2}}
\newcommand{\Holy}[1]{\cvec{H}^{#1}}
\newcommand{\cHoly}[1]{\cvec{H}^{\compl,#1}}
\newcommand{\Coly}[1]{\ctens{C}^{#1}}
\newcommand{\cColy}[1]{\ctens{C}^{\compl,#1}}

\newcommand{\lproj}[2]{\pi_{\mathcal{P},#2}^{#1}}
\newcommand{\vlproj}[2]{\bvec{\pi}_{\cvec{P},#2}^{#1}}
\newcommand{\tlproj}[2]{\btens{\pi}_{\ctens{P},#2}^{#1}}
\newcommand{\Hproj}[1]{\btens{\pi}_{\cvec{H},T}^{#1}}
\newcommand{\cHproj}[1]{\btens{\pi}_{\cvec{H},T}^{\compl,#1}}

\newcommand{\RV}[1]{\suvec{R}_{\bvec{V},#1}}
\newcommand{\EV}[1]{\uvec{E}_{\bvec{V},#1}}
\newcommand{\EPT}{\bvec{E}_{\cvec{P},T}^{k-2}}
\newcommand{\RS}[1]{\sutens{R}_{\btens{\Sigma},#1}}
\newcommand{\ES}[1]{\utens{E}_{\btens{\Sigma},#1}}

\newcommand{\SV}{\btens{S}_{\bvec{V},T}^{k-1}}
\newcommand{\SSigma}{\btens{S}_{\btens{\Sigma},T}^{k-1}}

\newcommand{\edges}[1]{\mathcal{E}_{#1}}
\newcommand{\vertices}[1]{\mathcal{V}_{#1}}

\newcommand{\ET}{\edges{T}}
\newcommand{\sET}{\widehat{\mathcal{E}}_T}
\newcommand{\VE}{\vertices{E}}
\newcommand{\VT}{\vertices{T}}

\newcommand{\normal}{\bvec{n}}
\newcommand{\tangent}{\bvec{t}}

\newcommand{\Mh}{\mathcal{M}_h}
\newcommand{\Th}{\mathcal{T}_h}
\newcommand{\Eh}{\mathcal{E}_h}
\newcommand{\Vh}{\mathcal{V}_h}

\DeclareMathOperator{\Ker}{Ker}
\DeclareMathOperator{\Image}{Im}

\newcommand{\norm}[2][]{\|#2\|_{#1}}
\newcommand{\seminorm}[2][]{|#2|_{#1}}
\newcommand{\vvvert}{\vert\kern-0.25ex\vert\kern-0.25ex\vert}
\newcommand{\tnorm}[2][]{\vvvert #2\vvvert_{#1}}

\newcommand{\term}{\mathfrak{T}}

\newcommand{\IV}[2][T]{\uvec{I}_{\bvec{V},#1}^{#2}}
\newcommand{\sIV}[2][T]{\suvec{I}_{\bvec{V},#1}^{#2}}
\newcommand{\ISigma}[2][T]{\utens{I}_{\btens{\Sigma},#1}^{#2}}
\newcommand{\sISigma}[2][T]{\sutens{I}_{\btens{\Sigma},#1}^{#2}}

\newcommand{\dE}[1][E]{\delta_{#1}}

\newcommand{\Csym}[1]{\boldsymbol{\mathsf{C}}_{\sym,T}^{#1}}
\newcommand{\DD}[2][T]{\mathsf{DD}_{#1}^{#2}}
\newcommand{\sDD}[2][T]{\ser{\mathsf{DD}}_{#1}^{#2}}
\newcommand{\uCsym}[2][T]{\utens{C}_{\sym,#1}^{#2}}
\newcommand{\suCsym}[2][T]{\sutens{C}_{\sym,#1}^{#2}}
\newcommand{\PV}[1]{\bvec{P}_{\bvec{V},T}^{#1}}
\newcommand{\PVh}[1]{\bvec{P}_{\bvec{V},h}^{#1}}
\newcommand{\PSigmaT}[1]{\btens{P}_{\btens{\Sigma},T}^{#1}}
\newcommand{\PSigmaE}[1]{P_{\btens{\Sigma},E}^{#1}}

\newcommand{\ErrV}{\mathfrak{E}_{\bvec{V},h}}
\newcommand{\ErrSigma}{\mathfrak{E}_{\btens{\Sigma},h}}
\newcommand{\Errsymcurl}{\mathfrak{E}_{\SYM\CURL,h}}
\newcommand{\Errdivdiv}{\mathfrak{E}_{\DIV\VDIV,h}}

\newcommand{\jump}[1]{[#1]_V}

\begin{document}

\title{A serendipity fully discrete div-div complex on polygonal meshes}
\author[1]{Michele Botti}
\author[2]{Daniele A. Di Pietro}
\author[2]{Marwa Salah}
\affil[1]{MOX, Politecnico di Milano, Italy, \email{michele.botti@polimi.it}}
\affil[2]{IMAG, Univ Montpellier, CNRS, Montpellier, France, \email{daniele.di-pietro@umontpellier.fr}, \email{marwa.salah@etu.umontpellier.fr}}

\maketitle

\begin{abstract}
  In this work we address the reduction of face degrees of freedom (DOFs) for discrete elasticity complexes.
  Specifically, using serendipity techniques, we develop a reduced version of a recently introduced two-dimensional complex arising from traces of the three-dimensional elasticity complex.
  The keystone of the reduction process is a new estimate of symmetric tensor-valued polynomial fields in terms of boundary values, completed with suitable projections of internal values for higher degrees.
  We prove an extensive set of new results for the original complex and show that the reduced complex has the same homological and analytical properties as the original one.
  This paper also contains an appendix with proofs of general Poincar\'e--Korn-type inequalities for hybrid fields.
\end{abstract}

%% \tableofcontents

\section{Introduction}

The development of computationally viable discrete elasticity complexes is a long-standing problem in numerical analysis.
Finite element versions of the elasticity complex typically require a large number of degrees of freedom (DOFs) to deal with the symmetry constraint on tensor-valued fields \cite{Arnold.Winther:02,Arnold.Awanou:05,Arnold.Falk.ea:06,Arnold.Awanou.ea:08,Chen.Huang:22, Sky.Muench.ea:22}.
Particularly critical are DOFs attached to mesh faces, that cannot be efficiently eliminated via static condensation.
In this work, we study DOFs reduction through \emph{serendipity}.
Serendipity techniques exploit the information on the boundary to fix the values of (a subset of) internal DOFs while preserving polynomial consistency.
When working with discrete complexes, this reduction must be carefully designed in order to preserve key properties of the original complex.
\smallskip

With face DOFs reduction in mind, we focus on the two-dimensional \emph{div-div complex} \cite{Chen.Hu.ea:18} that arises when considering traces for the three-dimensional elasticity complex on polyhedra (see \cite[Section 3.4]{Chen.Huang:22}).
Specifically, denoting by $\Omega\subset\Real^2$ a bounded connected polygonal set and by $\Symm$ the set of symmetric $2\times 2$ matrices, this complex reads:
\begin{equation}\label{eq:continuous.complex}
  \begin{tikzcd}
    \RT{1}(\Omega)
    \arrow[r,hook] & \bvec{H}^1(\Omega;\Real^2)
    \arrow{r}[above=2pt]{\SYM\CURL} & \Hdivdiv{\Omega}{\Symm}
    \arrow{r}[above=2pt]{\DIV\VDIV} & L^2(\Omega)
    \arrow{r}[above=2pt]{0} & 0,
  \end{tikzcd}
\end{equation}
where ``$\SYM$'' denotes the symmetric part of a space or an operator,
$\RT{1}(\Omega) \coloneq \vPoly{0}(\Omega) + \bvec{x}\Poly{0}(\Omega)$ is the lowest-order Raviart--Thomas space \cite{Raviart.Thomas:77},
and a definition of the $\SYM\CURL$ and $\DIV\VDIV$ operators in Cartesian coordinates is given in \eqref{eq:2d.differential.operators} below.
A discrete version of the complex \eqref{eq:continuous.complex} has been recently obtained in \cite{Di-Pietro.Droniou:22*1} following the \emph{discrete de Rham} (DDR) paradigm \cite{Di-Pietro.Droniou.ea:20,Di-Pietro.Droniou:21*1}.
A salient feature of DDR constructions is the native support of general polygonal/polyhedral meshes, which simplifies the discretisation of complex domain geometries and/or the capture of fine-scale features of the solution.
Alternative approaches to the use of polygonal/polyhedral meshes in the finite element framework include the fictitious domain method popularised by the work of Glowinski and coauthors; see, e.g., \cite{Glowinski.Pan.ea:94,Girault.Glowinski:95}.
In this work, following the abstract framework of \cite{Di-Pietro.Droniou:22} (closely inspired, through the bridges constructed in \cite{Beirao-da-Veiga.Dassi.ea:22}, by the ideas originally developed in \cite{Beirao-da-Veiga.Brezzi.ea:16*3,Beirao-da-Veiga.Brezzi.ea:17,Beirao-da-Veiga.Brezzi.ea:18}), we derive a reduced version of the DDR complex of \cite{Di-Pietro.Droniou:22*1} that preserves both its homological and analytical properties.
The keystone of this reduced version is the estimate of tensor-valued polynomials established in Lemma \ref{lem:estimate.symmetric.poly} below, which provides indications on which DOFs can be discarded while preserving polynomial consistency.
A comparison of the number of DOFs between the full and serendipity div-div complexes for various element shapes is provided in Table \ref{tab:dofs.reduction}, showing gains between 13\% and 27\% for the considered polynomial degrees $k$ and element shapes.

\begin{table}\centering
\begin{tabular}{c|cccc}
\toprule
Discrete space & $k=3$ & $k=4$ & $k=5$ & $k=6$ \\
\midrule
\multicolumn{ 5 }{c}{ Triangle, $\eta_T =  3 $ } \\
\midrule
$\bvec{H}^1(\Omega;\Real^2)$
& 24 \textbullet{} 20 (-17\%)& 36 \textbullet{} 30 (-17\%)& 50 \textbullet{} 42 (-16\%)& 66 \textbullet{} 56 (-15\%) \\
$\Hdivdiv{\Omega}{\Symm}$
& 24 \textbullet{} 20 (-17\%)& 39 \textbullet{} 33 (-15\%)& 57 \textbullet{} 49 (-14\%)& 78 \textbullet{} 68 (-13\%) \\
\midrule
\multicolumn{ 5 }{c}{ Quadrangle, $\eta_T =  4 $ } \\
\midrule
$\bvec{H}^1(\Omega;\Real^2)$
& 30 \textbullet{} 24 (-20\%)& 44 \textbullet{} 34 (-23\%)& 60 \textbullet{} 46 (-23\%)& 78 \textbullet{} 60 (-23\%) \\
$\Hdivdiv{\Omega}{\Symm}$
& 30 \textbullet{} 24 (-20\%)& 47 \textbullet{} 37 (-21\%)& 67 \textbullet{} 53 (-21\%)& 90 \textbullet{} 72 (-20\%) \\
\midrule
\multicolumn{ 5 }{c}{ Pentagon, $\eta_T =  5 $ } \\
\midrule
$\bvec{H}^1(\Omega;\Real^2)$
& 36 \textbullet{} 30 (-17\%)& 52 \textbullet{} 40 (-23\%)& 70 \textbullet{} 52 (-26\%)& 90 \textbullet{} 66 (-27\%) \\
$\Hdivdiv{\Omega}{\Symm}$
& 36 \textbullet{} 30 (-17\%)& 55 \textbullet{} 43 (-22\%)& 77 \textbullet{} 59 (-23\%)& 102 \textbullet{} 78 (-24\%) \\
\bottomrule
\end{tabular}
\caption{Number of DOFs for the full \textbullet{} serendipity discrete counterparts of the spaces $\bvec{H}^1(\Omega;\Real^2)$ and $\Hdivdiv{\Omega}{\Symm}$ on a triangle, quadrangle, and pentagon element $T$ for polynomial degrees $k$ ranging from 3 to 6. The relative DOFs reduction is in parenthesis. The parameter $\eta_T$ is defined in Assumption \ref{assum:choice.detaP} below. \label{tab:dofs.reduction}}
\end{table}

The rest of this work is organised as follows.
In Section \ref{sec:setting} we briefly recall the general setting.
The construction underlying the full DDR div-div complex is briefly recalled in Section \ref{sec:complex}, where we also prove a complete set of analytical results (Poincar\'e inequalities, consistency, and adjoint consistency) that complement the ones established in \cite{Di-Pietro.Droniou:22*1}.
The serendipity version of the DDR div-div complex is derived in Section \ref{sec:serendipity.complex}.
Through the sufficient conditions identified in \cite{Di-Pietro.Droniou:22*1}, we establish, in Theorems \ref{thm:homological.properties} and \ref{thm:analytical.properties} below, that the serendipity and full complexes have analogous homological and analytical properties.
Finally, Appendix \ref{appendix} focuses on Poincar\'e--Korn type inequalities for hybrid vector fields that are instrumental for the previous analysis.

%------------------------------------------------------------------------------%

\section{Setting}\label{sec:setting}

\subsection{Two-dimensional vector calculus operators}

Consider the real plane $\Real^2$ endowed with the Cartesian coordinate system $(x_1,x_2)$, and denote by $\partial_i$ the partial derivative with respect to the $i$th coordinate.
We need the following two-dimensional differential operators acting on smooth enough
scalar-valued fields $q$,
vector-valued fields $\bvec{v}=\begin{pmatrix}v_1\\v_2\end{pmatrix}$,
or matrix-valued fields $\btens{\tau}=\begin{pmatrix}\tau_{11} & \tau_{12}\\ \tau_{21} & \tau_{22}\end{pmatrix}$:
\begin{equation}\label{eq:2d.differential.operators}
  \begin{gathered}
    %% \GRAD q\coloneq\begin{pmatrix}\partial_1 q\\ \partial_2 q\end{pmatrix},\quad
    \CURL q\coloneq\begin{pmatrix}\partial_2 q\\ -\partial_1 q\end{pmatrix}, \quad
    \ROT \bvec{v} \coloneq \partial_2 v_1 - \partial_1 v_2,
    \\
    \DIV\bvec{v}\coloneq\partial_1 v_1 + \partial_2 v_2,\quad
    \GRAD\bvec{v}\coloneq\begin{pmatrix}
    \partial_1 v_1 & \partial_2 v_1 \\
    \partial_1 v_2 & \partial_2 v_2
    \end{pmatrix},\quad
    \SYM\CURL\bvec{v}\coloneq\begin{pmatrix}
    \partial_2 v_1 & \frac{-\partial_1 v_1+\partial_2 v_2}{2} \\
    \frac{-\partial_1 v_1+\partial_2 v_2}{2} & -\partial_1 v_2
    \end{pmatrix},
    \\
    \VDIV\btens{\tau}\coloneq\begin{pmatrix}
    \partial_1\tau_{11} + \partial_2\tau_{12}
    \\
    \partial_1\tau_{21} + \partial_2\tau_{22}
    \end{pmatrix},\quad
    \VROT\btens{\tau}\coloneq\begin{pmatrix}
    \partial_2\tau_{11} - \partial_1\tau_{12} \\
    \partial_2\tau_{21} - \partial_1\tau_{22}
    \end{pmatrix}.
  \end{gathered}
\end{equation}
Defining the fourth-order tensor $\mathbb{C}$ such that
\begin{equation}\label{eq:def.tensor.C}
  \mathbb{C}\btens{\tau}=\begin{pmatrix}\tau_{12} & \frac{-\tau_{11}+\tau_{22}}{2}\\ \frac{-\tau_{11}+\tau_{22}}{2} & -\tau_{21}\end{pmatrix}\qquad\forall \btens{\tau}=\begin{pmatrix}\tau_{11} & \tau_{12}\\ \tau_{21} & \tau_{22}\end{pmatrix}\in\Real^{2\times 2},
\end{equation}
we have $\SYM\CURL\bvec{v} = \mathbb{C}\GRAD\bvec{v}$.

\subsection{Mesh and notation for inequalities up to a constant}\label{sec:setting:mesh}

We denote by $\Mh = \Th\cup\Eh\cup\Vh$ a polygonal mesh of $\Omega$ in the usual sense of \cite{Di-Pietro.Droniou:20}, with $\Th$, $\Eh$, and $\Vh$ collecting, respectively, the elements, edges, and vertices and $h$ denoting the meshsize.
For all $Y\in\Mh$, we let $h_Y$ denote its diameter so that, in particular, $h=\max_{T\in\Th}h_T$.
$\Mh$ is assumed to belong to a refined mesh sequence with regularity parameter bounded away from zero.
We additionally assume that each element $T\in\Th$ is contractible and denote by $\bvec{x}_T$ a point inside $T$ such that there exists a disk contained in $T$ centered in $\bvec{x}_T$ and of diameter comparable to $h_T$ uniformly in $h$.
The sets of edges and vertices of $T$ are denoted by $\ET$ and $\VT$, respectively.
By mesh regularity, the number of edges (and vertices) of mesh elements are bounded uniformly in $h$.
For each edge $E\in\Eh$, we denote by $\VE$ the set of vertices corresponding to its endpoints and fix an orientation by prescribing a unit tangent vector $\tangent_E$.
This orientation determines two numbers $(\omega_{EV})_{V\in\VE}$ in $\{-1,+1\}$ such that $\omega_{EV}=+1$ whenever $\tangent_E$ points towards $V$.
The corresponding unit normal vector $\normal_E$ is selected so that $(\tangent_E,\normal_E)$ forms a right-handed system of coordinates, and, for each $T\in\Th$ such that $E\in\ET$, we denote by $\omega_{TE}\in\{-1,+1\}$ the orientation of $E$ relative to $T$, defined so that $\omega_{TE}\normal_E$ points out of $T$.

From this point on, $a \lesssim b$ means $a\le Cb$ with $C$ only depending on $\Omega$, the mesh regularity parameter, and the polynomial degree $k$ of the complex (see \eqref{eq:k} below).
We also write $a\simeq b$ as a shorthand for ``$a\lesssim b$ and $b\lesssim a$''.

\subsection{Polynomial spaces}

Given $Y\in\Mh$ and an integer $m\ge 0$, we denote by $\Poly{m}(Y)$ the space spanned by the restriction to $Y$ of two-variate polynomials of total degree $\le m$,  with the additional convention that $\Poly{-1}(Y)=\{0\}$.
The symbols $\vPoly{m}(Y;\Real^2)$ and $\tPoly{m}(Y;\Symm)$ denote, respectively, vector-valued and symmetric tensor-valued functions over $Y$ whose components are in $\Poly{m}(Y)$.
Finally, for each $T\in\Th$, we denote by $\Poly{m}(\ET)$ the space of broken polynomials of total degree $\le m$ on $\ET$.
Vector and tensor versions of this space are denoted in boldface and the codomain is specified.

Denoting by $\SYM\btens{\tau}=\frac{\btens{\tau}+\btens{\tau}^\top}{2}$ the symmetrisation operator, the following decompositions hold:
\begin{equation*}%\label{eq:PolySym=Holy.oplus.cHoly}
  \begin{gathered}
    \tPoly{m}(T;\Symm)
    = \Holy{m}(T) \oplus\cHoly{m}(T)
    \\
    \text{
      with $\Holy{m}(T)\coloneq\HESS\Poly{m+2}(T)$ and
      $\cHoly{m}(T)\coloneq\SYM\big(
      (\bvec{x} - \bvec{x}_T)^\bot\otimes\vPoly{m-1}(T;\Real^2)
      \big)$,
    }
  \end{gathered}
\end{equation*}
and
\begin{equation}\label{eq:vPoly=Coly.oplys.cColy}
  \begin{gathered}
    \tPoly{m}(T;\Symm)
    = \Coly{m}(T) \oplus \cColy{m}(T)
    \\
    \text{
      with $\Coly{m}(T)\coloneq\SYM\CURL\vPoly{m+1}(T;\Real^2)$ and
      $\cColy{m}(T)\coloneq(\bvec{x} - \bvec{x}_T)(\bvec{x} - \bvec{x}_T)^\top\Poly{m-2}(T)$.
    }
  \end{gathered}
\end{equation}

The following result will be needed in the analysis.
\begin{proposition}[Continuity of the inverses of local isomorphisms]\label{prop:continuity.inverses}
  Let $m\ge 1$ and denote by $\Poly{m\perp 1}(T)$ the subspace of $\Poly{m}(T)$ spanned by polynomials that are $L^2$-orthogonal to $\Poly{1}(T)$.
  Then,
  $\VROT:\cHoly{m}(T)\to\vPoly{m-1}(T;\Real^2)$, and, if $m\ge 2$,
  $\DIV\VDIV:\cColy{m}(T)\to\Poly{m-2}(T)$
  and $\HESS:\Poly{m\perp 1}(T)\mapsto\Holy{m-2}(T)$
  are isomorphisms with continuous inverse, i.e.,
  \begin{alignat}{4}\label{eq:est.vrot.-1}
    \norm[\btens{L}^2(T;\Real^{2\times 2})]{\btens{\upsilon}}
    &\lesssim
    h_T \norm[\bvec{L}^2(T;\Real^2)]{\VROT\btens{\upsilon}}
    &\qquad&\forall\btens{\upsilon}\in\cHoly{m}(T),
    \\
    \label{eq:est.div.div.-1}
    \norm[\btens{L}^2(T;\Real^{2\times 2})]{\btens{\upsilon}}
    &\lesssim
    h_T^2 \norm[L^2(T)]{\DIV\VDIV\btens{\upsilon}}
    &\qquad&\forall\btens{\upsilon}\in\cColy{m}(T),
    \\
    \label{eq:est.hess.-1}
    \norm[L^2(T)]{q}
    &\lesssim
    h_T^2\norm[\btens{L}^2(T;\Real^{2\times 2})]{\HESS q}
    &\qquad&\forall q\in\Poly{m\perp 1}(T).
  \end{alignat}
\end{proposition}

\begin{proof}
  The proof hinges on a scaling argument analogous to the one used in \cite[Lemma 9]{Di-Pietro.Droniou:21*1}, not repeated here for the sake of brevity.
\end{proof}

Given a polynomial (sub)space $\mathcal{X}^m(Y)$ on $Y\in\Mh$, the corresponding $L^2$-orthogonal projector is denoted by $\pi_{\mathcal{X},Y}^m$.
Boldface fonts will be used when the elements of $\mathcal{X}^m(Y)$ are tensor- or vector-valued.

%------------------------------------------------------------------------------%
%------------------------------------------------------------------------------%

\section{Full spaces and reconstructions}\label{sec:complex}

In this section we briefly recall the discrete div-div complex of \cite{Di-Pietro.Droniou:22*1}, for which we prove a complete panel of properties including Poincaré inequalities, consistency, and adjoint consistency results that complement the ones established in the previous reference.

\subsection{Spaces}

Throughout the rest of the paper, the integer
\begin{equation}\label{eq:k}
  k\ge 3
\end{equation}
will denote the polynomial degree of the discrete complex.
The discrete counterparts of the spaces $\bvec{H}^1(\Omega;\Real^2)$ and $\Hdivdiv{\Omega}{\Symm}$ are, respectively,
\begin{align}\label{eq:VT}
  \uvec{V}_h^k&\coloneq\Big\{
  \begin{aligned}[t]
    &\uvec{v}_h=
    \big(
    (\bvec{v}_T)_{T\in\Th},
    (\bvec{v}_E)_{E\in\Eh},
    (\bvec{v}_V, \btens{G}_{\bvec{v},V})_{V\in\Vh}
    \big)\st
    \\
    &\qquad\text{
      $\bvec{v}_T\in\vPoly{k-2}(T;\Real^2)$
      for all $T\in\Th$,
    }
    \\
    &\qquad\text{
      $\bvec{v}_E\in\vPoly{k-4}(E;\Real^2)$ for all $E\in\Eh$,
    }
    \\
    &\qquad\text{      
      $\bvec{v}_V\in\Real^2$
      and $\btens{G}_{\bvec{v},V}\in\Real^{2\times 2}$
      for all $V\in\Vh$
    }
    \Big\},
  \end{aligned}
  \\\label{eq:SigmaT}
  \utens{\Sigma}_h^{k-1}&\coloneq\Big\{
  \begin{aligned}[t]
    &\utens{\tau}_h
    =\big(
    (\btens{\tau}_{\cvec{H},T}, \btens{\tau}_{\cvec{H},T}^\compl)_{T\in\Th},
    (\tau_E,D_{\btens{\tau},E})_{E\in\Eh},
    (\btens{\tau}_V)_{V\in\Vh}
    \big)\st
    \\
    &\qquad\text{
      $\btens{\tau}_{\cvec{H},T}\in\Holy{k-4}(T)$
      and $\btens{\tau}_{\cvec{H},T}^\compl\in\cHoly{k-1}(T)$
      for all $T\in\Th$,        
    }
    \\
    &\qquad\text{
      $\tau_E\in\Poly{k-3}(E)$
      and $D_{\btens{\tau},E}\in\Poly{k-2}(E)$ for all $E\in\Eh$,
    }
    \\
    &\qquad\text{
      $\btens{\tau}_V\in\Symm$ for all $V\in\Vh$
    }
    \Big\}.
  \end{aligned}
\end{align}
The interpolators 
$\IV[h]{k}:\bvec{C}^1(\overline{\Omega};\Real^2)\to\uvec{V}_h^k$ and
$\ISigma[h]{k-1}:\btens{H}^2(\Omega;\Symm)\to\utens{\Sigma}_h^{k-1}$ are
such that, for all $\bvec{v}\in\bvec{C}^1(\overline{\Omega};\Real^2)$
and all $\btens{\tau}\in\btens{H}^2(\Omega;\Symm)$,
\begin{align*}%% \label{eq:IVT}
  \IV[h]{k}\bvec{v}&\coloneq
  \Big(
  (\bvec{\pi}_{\cvec{P},T}^{k-2}\bvec{v}_{|T})_{T\in\Th},
  (\bvec{\pi}_{\cvec{P},E}^{k-4}\bvec{v}_{|E})_{E\in\ET},
  \big(\bvec{v}(\bvec{x}_V),\GRAD\bvec{v}(\bvec{x}_V)\big)_{V\in\VT}
  \Big),
  \\ \label{eq:ISigmaT}
  \ISigma[h]{k-1}\btens{\tau}&\coloneq
  \Big(
  \big(\btens{\pi}_{\cvec{H},T}^{k-4}\btens{\tau}_{|T},
  \btens{\pi}_{\cvec{H},T}^{\compl,k-1}\btens{\tau}_{|T}
  \big)_{T\in\Th},
  \big(
  \pi_{\cvec{P},E}^{k-3}(\btens{\tau}_{|E}\normal_E\cdot\normal_E),
    \pi_{\cvec{P},E}^{k-2}\dE\btens{\tau}
    \big)_{E\in\ET},
    \big(\btens{\tau}(\bvec{x}_V)\big)_{V\in\VT}
    \Big),
\end{align*}
where $\bvec{x}_V$ denotes the coordinate vector of the vertex $V\in\VT$
while, for all $E\in\ET$, $\partial_{\tangent_E}$ denotes the derivative along the edge $E$ in the direction of $\tangent_E$ and we have set, for the sake of conciseness,
\[
\dE\btens{\tau}
\coloneq
\partial_{\tangent_E}(\btens{\tau}_{|E}\normal_E\cdot\tangent_E)
+ (\VDIV\btens{\tau})_{|E}\cdot\normal_E.
\]
As customary for DDR methods, we denote the restrictions of spaces and operators to a mesh element or edge $Y\in\Th\cup\Eh$ by replacing the subscript ``$h$'' with ``$Y$''.
Such restrictions are obtained collecting the polynomial components on $Y$ and its boundary.
Given $T\in\Th$, for $\uvec{V}_h^k$ we will also need its restriction $\uvec{V}_{\partial T}^k$ to the boundary of $T$, obtained collecting all the polynomial components that lie thereon.

\subsection{Reconstructions}\label{sec:reconstructions}

Let a mesh element $T\in\Th$ be fixed.
The DDR method hinges on the reconstructions of differential operators and of the corresponding potentials described below.

\subsubsection{Symmetric curl and vector potential}

The key integration by parts formula to reconstruct discrete counterparts of the symmetric curl and of the corresponding vector potential is the following:
For any $\bvec{v}:T\to\Real^2$ and any $\btens{\tau}:T\to\Symm$ smooth enough,
\begin{equation}\label{eq:ibp.VT}
  \int_T\bvec{v}\cdot\VROT\btens{\tau}
  = -\int_T\SYM\CURL\bvec{v}:\btens{\tau}
  + \sum_{E\in\ET}\omega_{TE}\int_{E}\bvec{v}\cdot(\btens{\tau}\,\bvec{t}_E).
\end{equation}
The \emph{full symmetric curl} $\Csym{k-1}:\uvec{V}_T^k\to\tPoly{k-1}(T;\Symm)$ is such that, for all $\uvec{v}_T\in\uvec{V}_T^k$,
\begin{equation}\label{eq:CsymT}
  \int_T\Csym{k-1}\uvec{v}_T:\btens{\tau}_T
  = -\int_T\bvec{v}_T\cdot\VROT\btens{\tau}_T
  + \sum_{E\in\ET}\omega_{TE}\int_E\bvec{v}_{\ET}\cdot(\btens{\tau}_T\,\bvec{t}_E)\qquad
  \forall\btens{\tau}_T\in\tPoly{k-1}(T;\Symm),
\end{equation}
where $\bvec{v}_{\ET}\in\vPoly{k}(\ET;\Real^2)\cap\bvec{C}^0(\partial T;\Real^2)$ is uniquely defined by the following conditions:
\begin{equation}\label{eq:v.ET}
  \begin{gathered}
    \text{
      $\bvec{\pi}_{\cvec{P},E}^{k-4}(\bvec{v}_{\ET})_{|E} = \bvec{v}_E$ for all $E\in\ET$,
      $\quad\partial_{\tangent_E}(\bvec{v}_{\ET})_{|E}(\bvec{x}_V)=\btens{G}_{\bvec{v},V}\tangent_E$
      for all $E\in\ET$ and $V\in\VE$,
    }
    \\
    \text{
      and $\bvec{v}_{\ET}(\bvec{x}_V) = \bvec{v}_V$ for all $V\in\VT$.
    }
  \end{gathered}
\end{equation}
The \emph{discrete symmetric curl} $\uCsym{k-1}:\uvec{V}_T^k\to\utens{\Sigma}_T^{k-1}$, acting between the discrete spaces in the com\-plex, is obtained setting, for all $\uvec{v}_T\in\uvec{V}_T^k$,
\begin{equation}\label{eq:uCsymT}
  \uCsym{k-1}\uvec{v}_T\coloneq\Big(
  \begin{aligned}[t]
    &\btens{\pi}_{\ctens{H},T}^{k-4}\big(\Csym{k-1}\uvec{v}_T\big),
    \btens{\pi}_{\ctens{H},T}^{\compl,k-1}\big(\Csym{k-1}\uvec{v}_T\big),
    \\
    &\big(
    \pi_{\mathcal{P},E}^{k-3}(\partial_{\tangent_E}\bvec{v}_{\ET}\cdot\normal_E),
    \partial_{\tangent_E}^2\bvec{v}_{\ET}\cdot\tangent_E
    \big)_{E\in\ET},
    \\
    &\big(
    \mathbb{C}\btens{G}_{\bvec{v},V}
    \big)_{V\in\VT}
    \Big),
  \end{aligned}
\end{equation}
with $\mathbb{C}$ as in \eqref{eq:def.tensor.C}. The \emph{global symmetric curl operator} $\uCsym[h]{k-1}:\uvec{V}_h^k\to\utens{\Sigma}_h^{k-1}$ is such that, for all $\uvec{v}_h\in\uvec{V}_h^k$,
\[
(\uCsym[h]{k-1}\uvec{v}_h)_{|T} =
\uCsym{k-1}\uvec{v}_T\qquad
\forall T\in\Th.
\]
Notice that this definition makes sense since the discrete curl components at vertices and edges are single-valued.
The \emph{vector potential} $\PV{k}:\uvec{V}_T^k\to\vPoly{k}(T;\Real^2)$ is such that, for all $\uvec{v}_T\in\uvec{V}_T^k$,
\begin{equation}\label{eq:PVT}
  \int_T\PV{k}\uvec{v}_T\cdot\VROT\btens{\tau}_T
  = -\int_T\Csym{k-1}\uvec{v}_T:\btens{\tau}_T
  + \sum_{E\in\ET}\omega_{TE}\int_E\bvec{v}_{\ET}\cdot(\btens{\tau}_T\,\tangent_E)\qquad
  \forall\btens{\tau}_T\in\cHoly{k+1}(T).
\end{equation}
We recall the following polynomial consistency property from \cite{Di-Pietro.Droniou:22*1}:
\begin{equation}\label{eq:PVT.consistency}
  \PV{T}\IV{k}\bvec{v} = \bvec{v}
  \qquad
  \forall \bvec{v}\in\vPoly{k}(T;\Real^2).
\end{equation}

\begin{remark}[Validity of \eqref{eq:PVT}]\label{rem:validity.PVT}
  Relation \eqref{eq:PVT} remains valid for all $\btens{\tau}_T\in\Holy{k-1}(T)\oplus\cHoly{k+1}(T)$, as can be checked taking $\btens{\tau}_T = \HESS q_T$ with $q_T\in\Poly{k+1}(T)$ and noticing that both sides vanish (use $\VROT\HESS = \bvec{0}$ for the left-hand side and the definition \eqref{eq:CsymT} of $\Csym{k-1}$ with $\btens{\tau}_T = \HESS q_T$ along with $\VROT\HESS = \bvec{0}$ for the right-hand side).
  This implies, in particular, that \eqref{eq:PVT} holds for all $\btens{\tau}_T\in\tPoly{k-1}(T;\Symm)\subset\Holy{k-1}(T)\oplus\cHoly{k+1}(T)$.
\end{remark}

\subsubsection{Div-div and tensor potential}

The starting point for reconstructions in $\utens{\Sigma}_T^{k-1}$ is the following integration by parts formula, corresponding to \cite[Eq.~(2.4)]{Comodi:89} (see also \cite[Eq.~(2)]{Chen.Huang:20}) and valid for all tensor-valued functions $\bvec{\tau}:T\to\Symm$ and all scalar-valued functions $q:T\to\Real$ smooth enough:
\begin{equation}\label{eq:ibp.SigmaT}
  \begin{aligned}
    \int_T\DIV\VDIV\btens{\tau}~q
    &= \int_T\btens{\tau}:\HESS q
    -\sum_{E\in\ET}\omega_{TE}\left[
      \int_E(\btens{\tau}\normal_E\cdot\normal_E)~\partial_{\normal_E}q
      - \int_E\dE\btens{\tau}~q
      \right]
    \\
    &\quad
    - \sum_{E\in\ET}\omega_{TE}\sum_{V\in\VE}\omega_{EV}(\btens{\tau}\normal_E\cdot\tangent_E)(\bvec{x}_V)~q(\bvec{x}_V).
  \end{aligned}
\end{equation}
For all $\utens{\tau}_T\in\utens{\Sigma}_T^{k-1}$, the \emph{discrete div-div operator} $\DD{k-2}:\uvec{\Sigma}_T^{k-1}\to\Poly{k-2}(T)$ is such that
\begin{multline}\label{eq:DDT}
  \int_T \DD{k-2}\utens{\tau}_T~q_T
  =
  \int_T\btens{\tau}_{\cvec{H},T}:\HESS q_T
  -\sum_{E\in\ET}\omega_{TE}\left(
  \int_E\tau_E\,\partial_{\normal_E}q_T
  -\int_E D_{\tau,E}\,q_T
  \right)
  \\
  -\sum_{E\in\ET}\omega_{TE}\sum_{V\in\VE}\omega_{EV}\,(\btens{\tau}_V\normal_E\cdot\tangent_E)\,q_T(\bvec{x}_V)
  \qquad\forall q_T\in\Poly{k-2}(T),
\end{multline}
while the \emph{tensor potential} $\PSigmaT{k-1}:\utens{\Sigma}_T^{k-1}\to\Poly{k-1}(T;\Symm)$ satisfies, for all $(q_T,\btens{\upsilon}_T)\in\Poly{k+1}(T)\times\cHoly{k-1}(T)$,
\begin{multline} \label{eq:P.Sigma.T}
  \int_T\PSigmaT{k-1}\utens{\tau}_T:\left(\HESS q_T + \btens{\upsilon}_T\right)
  =
  \int_T\DD{k-2}\utens{\tau}_T\,q_T
  + \sum_{E\in\ET}\omega_{TE}\left(
  \int_E \PSigmaE{k-1}\utens{\tau}_E\,\partial_{\normal_E} q_T
  - \int_E D_{\btens{\tau},E}\,q_T
  \right)
  \\    
  + \sum_{E\in\ET}\omega_{TE}\sum_{V\in\VE}\omega_{EV}(\btens{\tau}_V\normal_E\cdot\tangent_E)\,q_T(\bvec{x}_V)    
  + \int_T\btens{\tau}_{\cvec{H},T}^\compl:\btens{\upsilon}_T.
\end{multline}
Above, for all $E\in\ET$, denoting by $\utens{\tau}_E\coloneq\big(\tau_E, D_{\btens{\tau},E}, (\btens{\tau}_V)_{V\in\VE}\big)$ the restriction of $\utens{\tau}_T$ to $E$, $\PSigmaE{k-1}\utens{\tau}_E\in\Poly{k-1}(E)$ uniquely defined by the following conditions:
\[%% \begin{equation}\label{eq:P.Sigma.E}
  \text{
    $\PSigmaE{k-1}\utens{\tau}_E(\bvec{x}_V) = \btens{\tau}_V\normal_E\cdot\normal_E$ for all $V\in\VE$
    and $\pi_{\mathcal{P},E}^{k-3}\big(\PSigmaE{k-1}\utens{\tau}_E\big) = \tau_E$.
  }
\]%% \end{equation}
We recall for future use the following result proved in \cite[Lemma 4]{Di-Pietro.Droniou:22*1}:
\begin{equation}\label{eq:PSigmaT.circ.uCsymT=CsymT}
  \PSigmaT{k-1}\circ\uCsym{k-1} = \Csym{k-1},
\end{equation}
expressing the commutativity of the following diagram: 
\[
  \begin{tikzcd}
    \uvec{V}_T^k \arrow[r, "\Csym{k-1}"] \arrow[rd, swap, "\uCsym{k-1}"] & \tPoly{k-1}(T;\Symm) \\ 
    {} & \utens{\Sigma}_T^{k-1} \arrow[u, swap, "\PSigmaT{k-1}"] 
  \end{tikzcd}
\]
The global div-div operator $\DD[h]{k-2}:\utens{\Sigma}_h^{k-1}\to\Poly{k-2}(\Th)$ acting between spaces in the discrete complex is such that, for all $\utens{\tau}_h\in\utens{\Sigma}_h^{k-1}$,
\[
(\DD[h]{k-2}\utens{\tau}_h)_{|T}
\coloneq\DD{k-2}\utens{\tau}_T\qquad
\forall T\in\Th.
\]

\subsection{$L^2$-products and norms}\label{sec:l2-product.norms}

The discrete $L^2$-products in $\uvec{V}_h^k$ and $\utens{\Sigma}_h^{k-1}$ are defined setting: For all $\uvec{w}_h,\,\uvec{v}_h\in\uvec{V}_h^k$ and all $\utens{\upsilon}_h,\,\utens{\tau}_h\in\utens{\Sigma}_h^{k-1}$,
\[
(\uvec{w}_h,\uvec{v}_h)_{\bvec{V},h}
\coloneq\sum_{T\in\Th}(\uvec{w}_T,\uvec{v}_T)_{\bvec{V},T},\qquad
(\utens{\upsilon}_h,\utens{\tau}_h)_{\btens{\Sigma},h}
\coloneq\sum_{T\in\Th}(\utens{\upsilon}_T,\utens{\tau}_T)_{\btens{\Sigma},T},
\]
where, for all $T\in\Th$,
\begin{alignat}{2} \label{eq:L2 product in VT}
  (\uvec{w}_T,\uvec{v}_T)_{\bvec{V},T}
  &\coloneq\int_T\PV{k}\uvec{w}_T\cdot\PV{k}\uvec{v}_T
  + s_{\bvec{V},T}(\uvec{w}_T,\uvec{v}_T),
  \\ \label{eq:L2 product in ET}
  (\utens{\upsilon}_T,\utens{\tau}_T)_{\btens{\Sigma},T}
  &\coloneq\int_T\PSigmaT{k-1}\utens{\upsilon}_T:\PSigmaT{k-1}\utens{\tau}_T
  + s_{\btens{\Sigma},T}(\utens{\upsilon}_T,\utens{\tau}_T).
\end{alignat}
Above, $s_{\bvec{V},T}:\uvec{V}_T^k\times\uvec{V}_T^k\to\Real$
and $s_{\btens{\Sigma},T}:\utens{\Sigma}_T^{k-1}\times\utens{\Sigma}_T^{k-1}\to\Real$
are local stabilisation bilinear forms.
We refer to \cite[Section 4.2]{Di-Pietro.Droniou:22*1} for the precise expression of $s_{\btens{\Sigma},T}$ and we set
\begin{equation}\label{eq:SVT}
  s_{\bvec{V},T}(\uvec{w}_T,\uvec{v}_T)
  \coloneq h_T \sum_{E\in\ET}\int_E(\PV{k}\uvec{w}_T-\bvec{w}_{\ET})\cdot(\PV{k}\uvec{v}_T-\bvec{v}_{\ET}).
\end{equation}
By \eqref{eq:PVT.consistency}, this stabilisation bilinear form satisfies the following polynomial consistency property:
\begin{equation*}%% \label{eq:sVT:polynomial.consistency}
  s_{\bvec{V},T}(\IV{k}\bvec{w},\uvec{v}_T) = 0
  \qquad\forall(\bvec{w},\uvec{v}_T)\in\vPoly{k}(T;\Real^2)\times\uvec{V}_T^k,
\end{equation*}
so that
\begin{equation}\label{eq:L2prod.VT:polynomial.consistency}
  (\IV{k}\bvec{w},\uvec{v}_T)_{\bvec{V},T}
  = \int_T\bvec{w}\cdot\PV{k}\uvec{v}_T
  \qquad\forall(\bvec{w},\uvec{v}_T)\in\vPoly{k}(T;\Real^2)\times\uvec{V}_T^k.
\end{equation}

We define the following \emph{$L^2$-product norms}:
For $\bullet\in\Th\cup\{h\}$ and all $(\uvec{v}_\bullet,\utens{\tau}_\bullet)\in\uvec{V}_\bullet^k\times\utens{\Sigma}_\bullet^{k-1}$,
\begin{equation}\label{eq:norms.V.Sigma}
  \norm[\bvec{V},\bullet]{\uvec{v}_\bullet}\coloneq
  (\uvec{v}_\bullet, \uvec{v}_\bullet)_{\bvec{V},\bullet}^{\nicefrac12},\qquad
  \norm[\btens{\Sigma},\bullet]{\utens{\tau}_\bullet}\coloneq
  (\utens{\tau}_\bullet, \utens{\tau}_\bullet)_{\btens{\Sigma},\bullet}^{\nicefrac12}.
\end{equation}
Given $T\in\Th$, we also define the local \emph{component norms} $\tnorm[\bvec{V},T]{{\cdot}}$ on $\uvec{V}_T^k$ and
$\tnorm[\btens{\Sigma},T]{{\cdot}}$ on $\utens{\Sigma}_T^{k-1}$ such that,
for all $(\uvec{v}_T,\utens{\tau}_T)\in\uvec{V}_T^k\times\utens{\Sigma}_T^{k-1}$,
\begin{align}\label{eq:tnorm.V}
  \tnorm[\bvec{V},T]{\uvec{v}_T}^2
  &\coloneq
  \norm[\bvec{L}^2(T;\Real^2)]{\bvec{v}_T}^2
  + \sum_{E\in\ET} h_T\norm[\bvec{L}^2(E;\Real^2)]{\bvec{v}_E}^2
  + \sum_{V\in\VT} \left(
  h_T^2 |\bvec{v}_V|^2 + h_T^4 |\btens{G}_{\bvec{v},V}|^2
  \right),
  \\ \label{eq:tnorm.Sigma}
  \tnorm[\btens{\Sigma},T]{\utens{\tau}_T}^2
  &\coloneq
  \norm[\btens{L}^2(T;\Real^{2\times 2})]{\btens{\tau}_{\ctens{H},T}}^2
  + \norm[\btens{L}^2(T;\Real^{2\times 2})]{\btens{\tau}_{\ctens{H},T}^\compl}^2
  + \sum_{E\in\ET}\left(
  h_T\norm[L^2(E)]{\tau_E}^2 + h_T^3\norm[L^2(E)]{D_{\btens{\tau},E}}^2
  \right)
  \\ \nonumber
  &\qquad
  + \sum_{V\in\VT}h_T^2|\btens{\tau}_V|^2.
\end{align}
The corresponding global component norms, respectively denoted by $\tnorm[\bvec{V},h]{{\cdot}}$ and $\tnorm[\btens{\Sigma},h]{{\cdot}}$, are obtained summing the squares of the local norms on every $T\in\Th$ and taking the square root of the result.
The following equivalences hold uniformly in $h$:
For all $\bullet\in\Th\cup\{h\}$ and all $(\uvec{v}_\bullet,\utens{\tau}_\bullet)\in\uvec{V}_\bullet^k\times\utens{\Sigma}_\bullet^{k-1}$,
\begin{equation}\label{eq:norm.equivalence}
  \norm[\bvec{V},\bullet]{\uvec{v}_\bullet}\simeq
  \tnorm[\bvec{V},\bullet]{\uvec{v}_\bullet},\qquad
  \norm[\btens{\Sigma},\bullet]{\utens{\tau}_\bullet}\simeq
  \tnorm[\btens{\Sigma},\bullet]{\utens{\tau}_\bullet}.
\end{equation}
The second equivalence has been proved in \cite[Lemma 9]{Di-Pietro.Droniou:22*1}.
The first one follows from similar arguments, not detailed here for the sake of conciseness.

For future use, we note the following boundedness properties of the local interpolators, that can be proved using trace inequalities:
\begin{subequations}\label{eq:I:boundedness}
  \begin{alignat}{4}\label{eq:IV:boundedness}
    \tnorm[\bvec{V},T]{\IV{k}\bvec{v}}
    &\lesssim
    \norm[\bvec{L}^2(T;\Real^2)]{\bvec{v}}
    + h_T\seminorm[\bvec{H}^1(T;\Real^2)]{\bvec{v}}
    + h_T^2\seminorm[\bvec{H}^2(T;\Real^2)]{\bvec{v}}
    + h_T^3\seminorm[\bvec{H}^3(T;\Real^2)]{\bvec{v}}
    &\quad&
    \forall\bvec{v}\in\bvec{H}^3(T;\Real^2),
    \\ \label{eq:ISigma:boundedness}
    \tnorm[\btens{\Sigma},T]{\ISigma{k-1}\btens{\tau}}
    &\lesssim
    \norm[\btens{L}^2(T;\Real^{2\times 2})]{\btens{\tau}}
    + h_T\seminorm[\btens{H}^1(T;\Real^{2\times 2})]{\btens{\tau}}
    + h_T^2\seminorm[\btens{H}^2(T;\Real^{2\times 2})]{\btens{\tau}}
    &\quad&
    \forall\btens{\tau}\in\btens{H}^2(T;\Symm).
  \end{alignat}
\end{subequations}

\subsection{Poincaré inequalities}

The goal of this section is to prove the following result.

\begin{lemma}[Poincaré inequalities]\label{lem:poincare}
  The following properties hold:
  \begin{enumerate}
  \item For all $\uvec{v}_h\in\uvec{V}_h^k$ such that
  \begin{equation}\label{eq:orthogonality:sym.curl}
    \sum_{T\in\Th}\int_T\PV{k}\uvec{v}_T\cdot\bvec{w} = 0
    \qquad\forall\bvec{w}\in\RT{1}(\Omega),
  \end{equation}
  it holds, with hidden constant independent of $\uvec{v}_h$,
  \begin{equation}\label{eq:poincare.V}
    \tnorm[\bvec{V},h]{\uvec{v}_h}\lesssim\tnorm[\btens{\Sigma},h]{\uCsym[h]{k-1}\uvec{v}_h};
  \end{equation}
\item Denote by $[\cdot,\cdot]_{\btens{\Sigma},h}$ an inner product in $\utens{\Sigma}_h^{k-1}$ with induced norm equivalent to $\tnorm[\btens{\Sigma},h]{{\cdot}}$ uniformly in $h$.
  Then, for all $\utens{\tau}_h\in\utens{\Sigma}_h^{k-1}$ such that
  \[
  [\utens{\tau}_h,\utens{\eta}_h]_{\btens{\Sigma},h} = 0
  \qquad\forall \utens{\eta}_h \in \ker \DD[h]{k-2},
  \]
  it holds, with hidden constant independent of $\utens{\tau}_h$,
  \begin{equation}\label{eq:poincare.Sigma}
    \tnorm[\btens{\Sigma},h]{\utens{\tau}_h}
    \lesssim\norm[L^2(\Omega)]{\DD[h]{k-2}\utens{\tau}_h}.
  \end{equation}
  \end{enumerate}
\end{lemma}

\begin{remark}[Poincaré inequality for the symmetric curl]
    The standard $L^2$-product in \eqref{eq:orthogonality:sym.curl} could be replaced by a weighted version and \eqref{eq:poincare.V} would still hold.
\end{remark}

\subsubsection{Preliminary results}\label{sec:poincare:preliminary.results}

This section contains preliminary results required in the proof of Lemma \ref{lem:poincare}.
\begin{proposition}[Estimate of the $L^2$-norm of one-variate functions]
  Let $E\in\Eh$ and, for a given polynomial degree $m\ge 0$, $\varphi\in\Poly{m}(E)$.
  Then,
  \begin{equation}\label{eq:est.normL2.1d}
    \norm[L^2(E)]{\varphi}
    \lesssim
    \norm[L^2(E)]{\lproj{m-2}{E}\varphi}
    + h_E^{\nicefrac12} \sum_{V\in\VE}|\varphi(\bvec{x}_V)|.
  \end{equation}
\end{proposition}

\begin{proof}
  Let $\overline{\varphi}_E\coloneq \frac{1}{h_E}\int_E\varphi$ denote the average value of $\varphi$ over $E$.
  Inserting $\overline{\varphi}_E$ and using a triangle inequality, we can write
  \begin{equation}\label{eq:est.normL2.1d:basic}
    \norm[L^2(E)]{\varphi}
    \le\norm[L^2(E)]{\varphi - \overline{\varphi}_E} + \norm[L^2(E)]{\overline{\varphi}_E}
    \eqcolon \term_1 + \term_2.
  \end{equation}
  To estimate the first term we start with a Poincaré--Wirtinger inequality to write $\term_1\lesssim h_E\norm[L^2(E)]{\varphi'}$, where $\varphi'\coloneq\partial_{\tangent_E}\varphi$.
  We then notice that, for all $\psi\in\Poly{m-1}(E)$,
  $
  \int_E\varphi'~\psi = -\int_E\lproj{m-2}{E}\varphi~\psi' + \sum_{V\in\VE}\omega_{EV}\varphi(\bvec{x}_V)~\psi(\bvec{x}_V),
  $
  where the introduction of $\lproj{m-2}{E}$ inside the integral is justified by its definition after observing that $\psi'\in\Poly{m-2}(E)$ and $\omega_{EV}$ is the orientation of $V$ relative to $E$.
  Taking $\psi = \varphi'$, using Cauchy--Schwarz and discrete trace inequalities in the right-hand side, simplifying and multiplying by $h_E$, we obtain
  $h_E\norm[L^2(E)]{\varphi'}\lesssim\norm[L^2(E)]{\lproj{m-2}{E}\varphi} + h_E^{\nicefrac12}\sum_{V\in\VE}|\varphi(\bvec{x}_V)|$, hence
  \begin{equation}\label{eq:est.normL2.1d:term1}
    \term_1\lesssim\norm[L^2(E)]{\lproj{m-2}{E}\varphi} + h_E^{\nicefrac12}\sum_{V\in\VE}|\varphi(\bvec{x}_V)|.
  \end{equation}
  To estimate $\term_2$ we distinguish two cases. If $m\ge 2$, then $\term_2 = \norm[L^2(E)]{\lproj{0}{E}\varphi} = \norm[L^2(E)]{\lproj{0}{E}(\lproj{m-2}{E}\varphi)}\le\norm[L^2(E)]{\lproj{m-2}{E}\varphi}$ by $L^2$-boundedness of $\lproj{0}{E}$.
  If, on the other hand, $m\in\{0,1\}$, $\overline{\varphi}_E = \frac12\sum_{V\in\VE}\varphi(\bvec{x}_V)$, so that $\term_2\le h_E^{\nicefrac12}\sum_{V\in\VE}|\varphi(\bvec{x}_V)|$.
  In all the cases, we therefore have
  $
  \term_2\lesssim\norm[L^2(E)]{\lproj{m-2}{E}\varphi} + h_E^{\nicefrac12}\sum_{V\in\VE}|\varphi(\bvec{x}_V)|,
  $
  which, used together with \eqref{eq:est.normL2.1d:term1} in \eqref{eq:est.normL2.1d:basic}, yields \eqref{eq:est.normL2.1d}.
\end{proof}

Let $\uvec{v}_{\partial T}\in\uvec{V}_{\partial T}^k$ and let $\bvec{v}_{\ET}$ be given by \eqref{eq:v.ET}.
For all $E\in\ET$, we decompose it into its tangential and normal components as $\bvec{v}_{\ET|E} = v_{\normal,E} \normal_E + v_{\tangent,E} \tangent_E$ and, for $\bullet\in\{\normal,\tangent\}$, we let $v_{\bullet,\partial T}\in\Poly{k}(\ET)$ be such that $(v_{\bullet,\partial T})_{|E} = \omega_{TE}v_{\bullet,E}$ for all $E\in\ET$.
We additionally denote by $\partial_{\tangent_{\partial T}}$ the piecewise tangential derivative on $\partial T$ such that $(\partial_{\tangent_{\partial T}})_{|E}\coloneq\omega_{TE}\partial_{\tangent_E}$ for all $E\in\ET$.

\begin{proposition}[Estimate of the tangential derivative of the boundary reconstruction]\label{prop:est.boundary.fct}
  Let $T\in\Th$, $\uvec{v}_{\partial T}\in\uvec{V}_{\partial T}^k$, and $\bvec{v}_{\ET}$ given by \eqref{eq:v.ET} be such that $\int_{\partial T}\partial_{\tangent_{\partial T}}v_{\tangent,\partial T} = 0$.
  Then,
  \begin{equation}\label{eq:est.boundary.fct}
    \norm[\bvec{L}^2(\partial T;\Real^2)]{\partial_{\tangent_{\partial T}}\bvec{v}_{\ET}}
    \lesssim
    \norm[L^2(\partial T)]{\lproj{k-3}{\ET}(\partial_{\partial T}v_{\normal,\partial T})}
    + h_T\norm[L^2(\partial T)]{\partial_{\tangent_{\partial T}}^2v_{\tangent,\partial T}}
    + h_T^{\nicefrac12}\sum_{V\in\VT}|\mathbb{C}\btens{G}_{\bvec{v},V}|,
  \end{equation}
  where $\lproj{k-3}{\ET}$ denotes the $L^2$-orthogonal projector on $\Poly{k-3}(\ET)$.
\end{proposition}

\begin{proof}
  Denote, for the sake of brevity, by $\mathcal{N}_{\partial T}(\bvec{v}_{\ET})$ the quantity in the right-hand side of \eqref{eq:est.boundary.fct}.
  We start using a triangle inequality along with H\"older inequalities and the fact that $\tangent_E$ and $\normal_E$ are unit vectors to write
  \begin{equation}\label{eq:est.H1norm.PVT:term21:basic}
    \norm[\bvec{L}^2(\partial T;\Real^2)]{\partial_{\tangent_{\partial T}}\bvec{v}_{\ET}}      
    \lesssim
    \norm[L^2(\partial T)]{\partial_{\tangent_{\partial T}}v_{\normal,\partial T}}
    + \norm[L^2(\partial T)]{\partial_{\tangent_{\partial T}}v_{\tangent,\partial T}}
    \eqcolon\term_1 + \term_2.
  \end{equation}
  
  By \eqref{eq:est.normL2.1d} applied to each $E\in\ET$ with $\varphi = \partial_{\tangent_E}v_{\normal,E}$ and $m = k-1$, it is readily inferred that    
  \begin{equation}\label{eq:est.boundary.fct:term1}
    \term_1
    \lesssim
    \norm[L^2(\partial T)]{\lproj{k-3}{\ET}(\partial_{\tangent_{\partial T}} v_{\normal,\partial T})}
    + h_T^{\nicefrac12}\sum_{E\in\ET}\sum_{V\in\VE}|\partial_{\tangent_E} v_{\normal,E}(\bvec{x}_V)|
    \lesssim\mathcal{N}_{\partial T}(\bvec{v}_{\ET}),
  \end{equation}
  where the conclusion follows noticing that, for all $E\in\ET$ and all $V\in\VE$,
  ${|\partial_{\tangent_E} v_{\normal,E}(\bvec{x}_V)|}
  = {|\btens{G}_{\bvec{v},V}\tangent_E\cdot\normal_E|}
  = {|\mathbb{C}\btens{G}_{\bvec{v},V}\normal_E\cdot\normal_E|}
  \lesssim |\mathbb{C}\btens{G}_{\bvec{v},V}|$ and using $\card(\ET)\lesssim 1$.

  Let us now turn to $\term_2$. 
  Let $\varphi\in\Poly{k}(\ET)$ be such that $\int_{\partial T}\varphi = 0$.
  For all $V\in\VT$ shared by the edges $E_1,\,E_2\in\ET$ numbered so that $E_2$ follows $E_1$ travelling along $\partial T$ according to its orientation, define the jump $\jump{\varphi}\coloneq\varphi_{|E_2} - \varphi_{E_1}$.
  Then, it holds
  \begin{equation}\label{eq:poincare.dg}
    \norm[L^2(\partial T)]{\varphi} 
    \lesssim h_T\norm[L^2(\partial T)]{\partial_{\tangent_{\partial T}}\varphi}
    + h_T^{\nicefrac12}\sum_{V\in\VT}|\jump{\varphi}|.
  \end{equation}
  Apply this inequality to $\varphi = \partial_{\tangent_{\partial T}} v_{\tangent,\partial T}$ and denote by $\term_{2,1}$ and $\term_{2,2}$ the terms in the right-hand side.
  Clearly, $\term_{2,1} = \norm[L^2(\partial T)]{\partial_{\tangent_{\partial T}}^2v_{\tangent,\partial T}}\le\mathcal{N}_{\partial T}(\bvec{v}_{\ET})$.
  For the second contribution, we start by noticing that, for all $E\in\ET$ and all $V\in\VE$, $\partial_{\tangent_{\partial T}}v_{\tangent,\partial T}(\bvec{x}_V) = -\mathbb{C}\btens{G}_{\bvec{v},V}\normal_E\cdot\tangent_E + \frac12\tr\btens{G}_{\bvec{v},V}$ so that, in particular, for all $V\in\VT$,
  $\jump{\partial_{\tangent_{\partial T}} v_{\tangent,\partial T}}
  = \mathbb{C}\btens{G}_{\bvec{v},V}\normal_{E_1}\cdot\tangent_{E_1}
  - \mathbb{C}\btens{G}_{\bvec{v},V}\normal_{E_2}\cdot\tangent_{E_2}$.
  Using this fact along with $\card(\ET)\lesssim 1$, we conclude that $|\term_{2,2}|\lesssim h_T^{\nicefrac12}\sum_{V\in\VT}|\mathbb{C}\btens{G}_{\bvec{v},V}|\le\mathcal{N}_{\partial T}(\bvec{v}_{\ET})$.
  Gathering the above estimates on $\term_{2,1}$ and $\term_{2,2}$ finally gives $\term_2\lesssim\mathcal{N}_{\partial T}(\bvec{v}_{\ET})$ which, combined with \eqref{eq:est.boundary.fct:term1}, yields \eqref{eq:est.boundary.fct}.
\end{proof}

\begin{proposition}[Estimate of the discrete sym-curl norm of the vector potential]\label{lem:est.H1norm.PVT}
  For all $T\in\Th$ and all $\uvec{v}_T\in\uvec{V}_T^k$, it holds
  \begin{equation}\label{eq:est.H1norm.PVT}
    \norm[\btens{L}^2(T;\Real^{2\times 2})]{\SYM\CURL\PV{k}\uvec{v}_T}
    + \sum_{E\in\ET}h_T^{-\nicefrac12}\norm[\bvec{L}^2(E;\Real^2)]{\PV{k}\uvec{v}_T - \bvec{v}_{\ET}}
    \lesssim\tnorm[\btens{\Sigma},T]{\uCsym{k-1}\uvec{v}_T}.
  \end{equation}  
\end{proposition}

\begin{proof}
  Recalling Remark \ref{rem:validity.PVT} to write \eqref{eq:PVT} for $\btens{\tau}\in\tPoly{k-1}(T;\Symm)$ and using the integration by parts formula \eqref{eq:ibp.VT} for the left-hand side of the resulting expression, we have
  \[
  \begin{aligned}
    \int_T\SYM\CURL\PV{k}\uvec{v}_T:\btens{\tau}
    = \int_T\Csym{k-1}\uvec{v}_T:\btens{\tau}
    + \sum_{E\in\ET}\omega_{TE}\int_E(\PV{k}\uvec{v}_T - \bvec{v}_{\ET})\cdot(\btens{\tau}~\tangent_E).
  \end{aligned}
  \]
  Taking $\btens{\tau} = \SYM\CURL\PV{k}\uvec{v}_T$, using Cauchy--Schwarz and discrete trace inequalities in the right-hand side, and simplifying, we infer that
  \begin{equation}\label{eq:est.H1norm.PVT:basic}
    \begin{aligned}
      \norm[\btens{L}^2(T;\Real^{2\times 2})]{\SYM\CURL\PV{k}\uvec{v}_T}
      &\lesssim\norm[\btens{L}^2(T;\Real^{2\times 2})]{\Csym{k-1}\uvec{v}_T}
      + \sum_{E\in\ET} h_T^{-\nicefrac12} \norm[\bvec{L}^2(E;\Real^2)]{\PV{k}\uvec{v}_T - \bvec{v}_{\ET}}
      \\
      &\eqcolon\term_1 + \term_2.
    \end{aligned}
  \end{equation}
  We proceed to estimate the terms in the right-hand side.
  \smallskip\\
  \underline{(i) \emph{Estimate of $\term_1$.}}
  Using, in this order, \eqref{eq:PSigmaT.circ.uCsymT=CsymT}, the definitions \eqref{eq:norms.V.Sigma} of $\norm[\btens{\Sigma},T]{{\cdot}}$ and \eqref{eq:L2 product in ET} of the discrete $L^2$-product in $\utens{\Sigma}_T^{k-1}$,
  and the norm equivalence \eqref{eq:norm.equivalence}, we can write:
  For all $\uvec{v}_T\in\uvec{V}_T^k$,
  \begin{equation}\label{eq:est.H1norm.PVT:term1}
    \term_1
    = \norm[\btens{L}^2(T;\Real^{2\times 2})]{\PSigmaT{k-1}\uCsym{k-1}\uvec{v}_T}
    \le \norm[\btens{\Sigma},T]{\uCsym{k-1}\uvec{v}_T}
    \lesssim\tnorm[\btens{\Sigma},T]{\uCsym{k-1}\uvec{v}_T}.
  \end{equation}
  \\
  \underline{(ii) \emph{Estimate of $\term_2$.}}
  Let $\bvec{w}\in\RT{1}(T)$ be such that
  \begin{equation}\label{eq:w}
    \text{
      $\int_{\partial T}\bvec{w} = \int_{\partial T}\bvec{v}_{\ET}$\quad
      and\quad $\int_{\partial T}\partial_{\tangent_{\partial T}} w_{\tangent,\partial T}
      = \int_{\partial T}\partial_{\tangent_{\partial T}} v_{\tangent_{\partial T}}$.
    }
  \end{equation}
  To check that it is possible to match these conditions, write $\bvec{w}(\bvec{x}) = \bvec{z} + (\bvec{x} - \bvec{x}_{\partial T})q$ with $\bvec{z}\in\Real^2$, $q\in\Real$, and $\bvec{x}_{\partial T}\coloneq\frac{1}{|\partial T|}\int_{\partial T}\bvec{x}$, and notice that the first condition in \eqref{eq:w} yields $\bvec{z} = \frac{1}{|\partial T|}\int_{\partial T}\bvec{v}_{\ET}$, while the second one is fulfilled taking 
  $q = \frac{1}{|\partial T|}\int_{\partial T}\partial_{\tangent_{\partial T}} v_{\tangent_{\partial T}}$. %\big/ \left(\sum_{E\in\ET}\omega_{TE}\int_E\partial_{\tangent_E}(\bvec{x}\cdot\tangent_E)\right)$.
  
  Using a triangle inequality, we have
  \[
  \term_2
  \lesssim
  \sum_{E\in\ET} h_T^{-\nicefrac12}\norm[\bvec{L}^2(E;\Real^2)]{\bvec{w} - \bvec{v}_{\ET}}
  + \sum_{E\in\ET} h_T^{-\nicefrac12}\norm[\bvec{L}^2(E;\Real^2)]{\PV{k}\uvec{v}_T - \bvec{w}}
  \eqcolon\term_{2,1} + \term_{2,2}.
  \]
  Starting with $\term_{2,1}$, noticing that $\bvec{w} - \bvec{v}_{\ET}\in\bvec{C}^0(\partial T;\Real^2)$ has zero average on $\partial T$, applying a Poincar\'e--Wirtinger inequality on $\partial T$ as in \cite[Point 1. of Lemma 7]{Di-Pietro.Droniou:22}, and concluding with Proposition \ref{prop:est.boundary.fct} gives
  \begin{equation}\label{eq:est.H1norm.PVT:term21}
    \term_{2,1}\lesssim
    h_T^{\nicefrac12}
    \norm[\bvec{L}^2(\partial_T;\Real^2)]{\partial_{\tangent_{\partial T}}(\bvec{v}_{\ET}-\bvec{w})}
    \lesssim\tnorm[\btens{\Sigma},T]{\uCsym{k-1}(\uvec{v}_T-\IV{k}\bvec{w})}
    = \tnorm[\btens{\Sigma},T]{\uCsym{k-1}\uvec{v}_T},
  \end{equation}
  where, in the second step, we have additionally used the consistency of the boundary reconstruction \eqref{eq:v.ET} applied to $\IV{k}\bvec{w}$, while the conclusion follows recalling that $\uCsym{k-1}\IV{k}\bvec{w} = \uvec{0}$ by the local complex property for the DDR sequence.

  Let us now consider $\term_{2,2}$.
  By polynomial consistency \eqref{eq:PVT.consistency} of $\PV{k}$, it holds $\PV{k}\IV{k}\bvec{w} = \bvec{w}$, hence
  \begin{equation}\label{eq:est.H1norm.PVT:term22:basic}
    \term_{2,2}
    = \sum_{E\in\ET}h_T^{-\nicefrac12}\norm[\bvec{L}^2(E;\Real^2)]{\PV{k}(\uvec{v}_T - \IV{k}\bvec{w})}
    \lesssim h_T^{-1}\norm[\bvec{L}^2(T;\Real^2)]{\PV{k}(\uvec{v}_T - \IV{k}\bvec{w})},
  \end{equation}
  where the conclusion follows from discrete trace inequalities along with $\card(\ET)\lesssim 1$.
  Taking, in the definition \eqref{eq:PVT} of $\PV{k}$, $\btens{\tau}\in\cHoly{k+1}(T)$ such that $\VROT\btens{\tau} = \PV{k}(\uvec{v}_T - \IV{k}\bvec{w})$ (this is possible since $\VROT:\cHoly{k+1}(T)\to\vPoly{k}(T)$ is surjective by Proposition \ref{prop:continuity.inverses}) and using Cauchy--Schwarz and discrete trace inequalities in the right-hand side along with \eqref{eq:est.vrot.-1} to write $\norm[\btens{L}^2(T;\Real^{2\times 2})]{\btens{\tau}}\lesssim h_T\norm[\bvec{L}^2(T;\Real^2)]{\VROT\btens{\tau}} = h_T\norm[\bvec{L}^2(T;\Real^2)]{\PV{k}(\uvec{v}_T - \IV{k}\bvec{w})}$, we obtain, after simplification,
  \[
  \begin{aligned}
    \norm[\bvec{L}^2(T;\Real^2)]{\PV{k}(\uvec{v}_T - \IV{k}\bvec{w})}
    &\lesssim h_T\left(
    \norm[\btens{L}^2(T;\Real^{2\times 2})]{\Csym{k-1}\uvec{v}_T}
    + h_T^{-\nicefrac12}\norm[\bvec{L}^2(\partial T;\Real^2)]{\bvec{v}_{\ET} - \bvec{w}}
    \right)
    \\
    &\lesssim h_T\tnorm[\btens{\Sigma},T]{\uCsym{k-1}\uvec{v}_T},
  \end{aligned}
  \]
  where the conclusion follows using, respectively, \eqref{eq:est.H1norm.PVT:term1} and \eqref{eq:est.H1norm.PVT:term21} to estimate the terms in parentheses.
  Plugging this estimate into \eqref{eq:est.H1norm.PVT:term22:basic}, we finally get
  $\term_{2,2}\lesssim\tnorm[\btens{\Sigma},T]{\uCsym{k-1}\uvec{v}_T}$,
  which, combined with \eqref{eq:est.H1norm.PVT:term21}, gives
  \begin{equation}\label{eq:est.H1norm.PVT:term2}
    \term_2\lesssim\tnorm[\btens{\Sigma},T]{\uCsym{k-1}\uvec{v}_T}.
  \end{equation}
  \\
  \underline{(iii) \emph{Conclusion.}}
  Plug \eqref{eq:est.H1norm.PVT:term1} and \eqref{eq:est.H1norm.PVT:term2} into \eqref{eq:est.H1norm.PVT:basic} to estimate the first term in the left-hand side of \eqref{eq:est.H1norm.PVT} and notice that the estimate of the second term in the left-hand side of \eqref{eq:est.H1norm.PVT} is precisely \eqref{eq:est.H1norm.PVT:term2}.
\end{proof}

\subsubsection{Proof of the discrete Poincar\'e inequalities}\label{sec:poincare:proof}

\begin{proof}[Proof of Lemma \ref{lem:poincare}]
\underline{(i) \emph{Poincar\'e inequality \eqref{eq:poincare.V}} for $\uCsym[h]{k-1}$.} %First, we use the result of the previous Section to establish \eqref{eq:poincare.V}. 
Let $\uvec{v}_h\in\uvec{V}_h^k$ be such that $\int_\Omega{\PVh{k}\uvec{v}_h\cdot\bvec{w}} = 0$ for all $\bvec{w}\in\RT{1}(\Omega)$, 
with the global reconstruction operator $\PVh{k}$ defined such that $(\PVh{k}\uvec{v}_h)_{|T} = \PV{k}\uvec{v}_T$ for all $T\in\Th$.
Owing to the uniform norm equivalence \eqref{eq:norm.equivalence}, the definitions \eqref{eq:norms.V.Sigma} of the $\norm[\bvec{V},T]{{\cdot}}$-norm and \eqref{eq:SVT} of the stabilisation bilinear form, and the fact that $h_T\lesssim 1$ for all $T\in\Th$, we infer
\begin{equation}\label{eq:poinV_start}
\begin{aligned}
  \tnorm[\bvec{V},h]{\uvec{v}_h}^2
  \lesssim \norm[\bvec{V},h]{\uvec{v}_h}^2
  &= \norm[\bvec{L}^2(\Omega;\Real^2)]{\PVh{k}\uvec{v}_h}^2 + \sum_{T\in\Th} s_{\bvec{V},T}(\uvec{v}_T,\uvec{v}_T) \\
  &\lesssim  \norm[\bvec{L}^2(\Omega;\Real^2)]{\PVh{k}\uvec{v}_h}^2 + \sum_{T\in\Th}\sum_{E\in\ET}h_T^{-1} \norm[\bvec{L}^2(E;\Real^2)]{\PV{k}\uvec{v}_T - \bvec{v}_{\ET}}^2.
\end{aligned}
\end{equation}
We notice that, for all neighboring elements $T_1, T_2\in\Th$ sharing the internal edge $E$, we have $(\bvec{v}_{\mathcal{E}_{T_1}})_{|E} = (\bvec{v}_{\mathcal{E}_{T_2}})_{|E}\eqcolon\hat{\bvec{v}}_E$.
Letting, for any boundary edge $E\subset\partial T$, $\hat{\bvec{v}}_E\coloneq(\bvec{v}_{\mathcal{E}_T})_{|E}$ and applying the second inequality of Proposition \ref{prop:Poin-Korn_hybrid} below to the hybrid vector field $\uvec{u}_h = \big((\PV{k}\uvec{v}_T)_{T\in\Th}, (\hat{\bvec{v}}_E)_{E\in\Eh}\big)$, we obtain
$$
 \norm[L^2(\Omega;\Real^2)]{\PVh{k}\uvec{v}_h}^2 \lesssim
\sum_{T\in\Th}\left(\norm[\btens{L}^2(\Omega;\Real^{2\times2})]{\SYM\CURL\PV{k}\uvec{v}_T}^2
 + \sum_{E\in\ET}h_T^{-1}\norm[\bvec{L}^2(E;\Real^2)]{\PV{k}\uvec{v}_T - \bvec{v}_{\ET}}^2\right).
$$
 Plugging the previous bound into \eqref{eq:poinV_start} and using Proposition \ref{lem:est.H1norm.PVT}, \eqref{eq:poincare.V} follows.
%%  , it is inferred that
%% $$
%% \begin{aligned}
%%  \tnorm[\bvec{V},h]{\uvec{v}_h}^2 &\lesssim
%%  \sum_{T\in\Th}\left(\norm[\btens{L}^2(\Omega;\Real^{2\times2})]{\SYM\CURL\PV{k}\uvec{v}_T}^2
%%  + \sum_{E\in\ET}h_T^{-1}\norm[\bvec{L}^2(E;\Real^2)]{\PV{k}\uvec{v}_T - \bvec{v}_{\ET}}^2\right) \\
%%  &\lesssim\sum_{T\in\Th}\left(\norm[\btens{L}^2(\Omega;\Real^{2\times2})]{\SYM\CURL\PV{k}\uvec{v}_T}
%%  + \sum_{E\in\ET}h_T^{-\nicefrac12}\norm[\bvec{L}^2(E;\Real^2)]{\PV{k}\uvec{v}_T - \bvec{v}_{\ET}}\right)^2 \\
%%   &\lesssim \tnorm[\btens{\Sigma},h]{\uCsym[h]{k-1}\uvec{v}_h}^2,
%% \end{aligned}
%% $$
%% where the conclusion follows by Proposition \ref{lem:est.H1norm.PVT}.
\medskip\\
\underline{(ii) \emph{Poincar\'e inequality \eqref{eq:poincare.Sigma}} for $\DD[h]{k-2}$.}  
Let $\utens{\tau}_h\in \big(\ker \DD[h]{k-2} \big)^{\perp}$, where the space $\big(\ker \DD[h]{k-2}\big)^{\perp} \subset \utens{\Sigma}_h^{k-1}$ denotes the orthogonal of $\ker \DD[h]{k-2}$ with respect to the inner product $[\cdot,\cdot]_{\btens{\Sigma},h}$. %Then $\DD[h]{k-2}: \left(\ker \DD[h]{k-2}\right)^{\perp} \to  \Poly{k-2}(\Th)$ is an isomorphism as a consequence of \cite[Eq. 32]{Di-Pietro.Droniou:22*1}
Owing to the surjectivity of the operator $\DIV\VDIV: \btens{H}^2(\Omega;\Symm)\to L^2(\Omega)$ (cf. \cite[Theorem 3.25]{Pauly.Zulehner:20}), the commutation property stated in \cite[Eq. (19)]{Di-Pietro.Droniou:22*1}, and the boundedness of the global interpolator $\ISigma[h]{k-1}$ resulting from \eqref{eq:ISigma:boundedness}, we infer the existence of a tensor field $\btens{\tau}\in \btens{H}^2(\Omega;\Symm)$ such that
\begin{equation}\label{eq:tnorm.eps}
\DD[h]{k-2} \utens{\tau}_h = \DIV\VDIV \btens{\tau} = \DD[h]{k-2}(\ISigma[h]{k-1}\btens{\tau}) \;\text{ and }\;
\tnorm[\btens{\Sigma},h]{\ISigma[h]{k-1}\btens{\tau}} \lesssim \norm[\btens{H}^2(\Omega;\Symm)]{\btens{\tau}} \lesssim \norm[\btens{L}^2(\Omega)]{\DD[h]{k-2} \utens{\tau}_h}.
\end{equation}
Therefore, we have that $\utens{\tau}_h-\ISigma[h]{k-1}\btens{\tau} \in \ker \DD[h]{k-2}$, i.e.,
\[
[\utens{\tau}_h-\ISigma[h]{k-1}\btens{\tau},\utens{\upsilon}_h]_{\btens{\Sigma},h}=0 \qquad\forall \utens{\upsilon}_h\in \big(\ker \DD[h]{k-2} \big)^{\perp}, 
\]
namely, $\utens{\tau}_h$ can be seen as the $[\cdot,\cdot]_{\btens{\Sigma},h}$-orthogonal projection of $\ISigma[h]{k-1}\btens{\tau}$ on the space $\big(\ker \DD[h]{k-2}\big)^{\perp}$.
Thus, the norm induced by $[\cdot,\cdot]_{\btens{\Sigma},h}$ of $\utens {\tau}_h$ is bounded by that of $\ISigma[h]{k-1}\btens{\tau}$, and the assumed uniform equivalence between the induced norm and $\tnorm[\btens{\Sigma},h]{{\cdot}}$ along with the inequality in \eqref{eq:tnorm.eps} yields the result.
\end{proof}

\subsection{Consistency of the discrete $L^2$-products}

\begin{lemma}[Consistency of the discrete $L^2$-products]\label{lem:consistency}
  The discrete $L^2$-products satisfy the following consistency properties:
  \begin{enumerate}
  \item For all $\bvec{w}\in\bvec{H}^3(\Omega;\Real^2)$, define the linear form $\ErrV(\bvec{w};\cdot):\uvec{V}_h^k\to\Real$ such that
    \[
    \ErrV(\bvec{w};\uvec{v}_h)
    \coloneq\sum_{T\in\Th}\int_T\bvec{w}\cdot\PV{k}\uvec{v}_T
    - (\IV[h]{k}\bvec{w}, \uvec{v}_h)_{\bvec{V},h}
    \qquad\forall\uvec{v}_h\in\uvec{V}_h^k.
    \]
    Then, under the additional regularity $\bvec{w}\in\bvec{H}^{k+1}(\Th;\Real^2)$, it holds
    \begin{equation}\label{eq:ErrV:est}
      \sup_{\uvec{v}_h\in\uvec{V}_h^k\setminus\{\uvec{0}\}}\frac{\ErrV(\bvec{w};\uvec{v}_h)}{\norm[\bvec{V},h]{\uvec{v}_h}}
      \lesssim h^{k+1}\seminorm[\bvec{H}^{k+1}(\Th;\Real^2)]{\bvec{v}}.
    \end{equation}
  \item For all $\btens{\upsilon}\in\btens{H}^2(\Omega;\Symm)$, define the linear form $\ErrSigma(\btens{\upsilon};\cdot):\utens{\Sigma}_h^{k-1}\to\Real$ such that
    \[
    \ErrSigma(\btens{\upsilon};\utens{\tau}_h)
    \coloneq\sum_{T\in\Th}\int_T\btens{\upsilon}:\PSigmaT{k-1}\utens{\tau}_T
    - (\ISigma[h]{k-1}\btens{\upsilon}, \utens{\tau}_h)_{\btens{\Sigma},h}
    \qquad\forall\utens{\tau}_h\in\utens{\Sigma}_h^{k-1}.
    \]
    Then, under the additional regularity $\btens{\upsilon}\in\btens{H}^k(\Th;\Symm)$, it holds
    \begin{equation}\label{eq:ErrS:est}
      \sup_{\utens{\tau}_h\in\utens{\Sigma}^{k-1}_h\setminus\{\utens{0}\}}\frac{\ErrSigma(\btens{\upsilon};\utens{\tau}_h)}{\norm[\btens{\Sigma},h]{\utens{\tau}_h}}
      \lesssim h^k\seminorm[\btens{H}^k(\Th;\Real^{2\times 2})]{\btens{\upsilon}}.
    \end{equation}
  \end{enumerate}
\end{lemma}

\begin{proof}
  \underline{(i) \emph{Proof of \eqref{eq:ErrV:est}.}}
  By the polynomial consistency \eqref{eq:L2prod.VT:polynomial.consistency} of the discrete $L^2$-product in $\uvec{V}_T^k$, we can write
  \begin{equation}\label{eq:ErrV:est:basic}
    \ErrV(\bvec{w};\uvec{v}_h)
    = \sum_{T\in\Th}\left(
    \term_1(T) + \term_2(T)
    \right),
  \end{equation}
  where, recalling that $\vlproj{k}{T}$ denotes the $L^2$-orthogonal projector on $\vPoly{k}(T)$,
  \[
  \term_1(T)\coloneq\int_T(\bvec{w} - \vlproj{k}{T}\bvec{w}_{|T})\cdot\PV{k}\uvec{v}_T,\qquad
  \term_2(T)\coloneq(\IV{k}(\bvec{w} - \vlproj{k}{T}\bvec{w}_{|T}),\uvec{v}_T)_{\bvec{V},T}.
  \]
  For the first term, a Cauchy--Schwarz inequality followed by the approximation properties of $\vlproj{k}{T}$ (see, e.g., \cite[Lemma 3.4]{Di-Pietro.Droniou:17} or \cite[Section 1.3.3]{Di-Pietro.Droniou:20}) and the definition \eqref{eq:norms.V.Sigma} of the $\norm[\bvec{V},T]{{\cdot}}$-norm give
  \begin{equation}\label{eq:ErrV:est:T1(T)}
    |\term_1(T)|
    \le\norm[\bvec{L}^2(T;\Real^2)]{\bvec{w} - \vlproj{k}{T}\bvec{w}_{|T}}
    \norm[\bvec{L}^2(T;\Real^2)]{\PV{k}\uvec{v}_T}
    \lesssim h_T^{k+1}\seminorm[\bvec{H}^{k+1}(T;\Real^2)]{\bvec{w}}
    \norm[\bvec{V},T]{\uvec{v}_T}.
  \end{equation}
  For the second term, a Cauchy--Schwarz inequality, the local norm equivalence expressed by \eqref{eq:norm.equivalence} with $\bullet = T$ along with the boundedness \eqref{eq:IV:boundedness} of $\IV{k}$, and again the approximation properties of $\vlproj{k}{T}$ give
  \begin{equation}\label{eq:ErrV:est:T2(T)}
    \begin{aligned}
      |\term_2(T)|
      &\le\norm[\bvec{V},T]{\IV{k}(\bvec{w} - \vlproj{k}{T}\bvec{w}_{|T})}
      \norm[\bvec{V},T]{\uvec{v}_T}
      \\
      &\lesssim\left(
      \sum_{i=0}^3h_T^i\seminorm[\bvec{H}^i(T;\Real^2)]{\bvec{w} - \vlproj{k}{T}\bvec{w}_{|T}}
      \right)\norm[\bvec{V},T]{\uvec{v}_T}
      \\
      &\lesssim
      h_T^{k+1}\seminorm[\bvec{H}^{k+1}(T;\Real^2)]{\bvec{w}}
      \norm[\bvec{V},T]{\uvec{v}_T}.
    \end{aligned}
  \end{equation}
  Using \eqref{eq:ErrV:est:T1(T)} and \eqref{eq:ErrV:est:T2(T)} to bound the terms in the right-hand side of \eqref{eq:ErrV:est:basic}, we obtain \eqref{eq:ErrV:est} after applying a discrete Cauchy--Schwarz inequality on the sum over $T\in\Th$.
  \medskip\\
  \underline{(ii) \emph{Proof of \eqref{eq:ErrS:est}.}}
  The proof coincides with the estimate the term $\term_1$ in the proof of \cite[Lemma 15]{Di-Pietro.Droniou:22} with $\ell = k+1$, to which we refer for further details.
\end{proof}

\subsection{Adjoint consistency of the discrete differential operators}

To state the following theorem, we denote by $\normal$ the normal vector field on $\partial\Omega$ pointing out of $\Omega$ and by $\tangent$ the tangent vector field oriented so that $(\tangent,\normal)$ forms a right-handed coordinate system.

\begin{lemma}[Adjoint consistency]\label{lem:adjoint.consistency}
  The discrete differential operators defined in Section \ref{sec:reconstructions} satisfy the following adjoint consistency properties:
  \begin{enumerate}
  \item Given $\btens{\upsilon}\in\btens{H}^2(\Omega;\Symm)$ such that $\btens{\upsilon}\tangent = \bvec{0}$ on $\partial\Omega$, we define the sym curl adjoint consistency error $ \Errsymcurl:\uvec{V}_h^k \to \Real $ by:
    For all $\uvec{v}_h\in\uvec{V}_h^k$, 
    \[
      \Errsymcurl(\btens{\upsilon}; \uvec{v}_h)
      \coloneq
      (\ISigma[h]{k-1}{\btens{\upsilon}},\uCsym[h]{k-1}\uvec{v}_h)_{\btens{\Sigma},h}
      + \sum_{T\in\Th}\int_{T} \VROT\btens{\upsilon} \cdot \PV{k}\uvec{v}_T.
      \]
      Then, further assuming $\btens{\upsilon}\in\btens{H}^k(\Th;\Symm)$, it holds:
      For all $ \uvec{v}_h \in \uvec{V}_h^k$,
      \begin{equation}\label{eq:adjoint_symcurl}
        |\Errsymcurl (\btens{\upsilon}; \uvec{v}_h)|
        \lesssim h^k \seminorm[\btens{H}^k(\Th;\Real^{2\times 2})]{\btens{\upsilon}} \norm[\btens{\Sigma},h]{\uCsym[h]{k-1}\uvec{v}_h}.
      \end{equation}
    \item Given $q\in H^2(\Omega)$ such that $q = \partial_{\normal} q = 0$ on $\partial\Omega$, we define the div-div adjoint consistency error $\Errdivdiv:\utens{\Sigma}_h^{k-1} \to \Real $ by:
      For all $\utens{\tau}_h \in  \utens{\Sigma}_h^{k-1}$,
      \begin{equation}\label{eq:Edivdivh}
        \Errdivdiv(q;\utens{\tau}_h)
        \coloneq \int_\Omega q~\DD[h]{k-2}\utens{\tau}_h
        - \sum_{T\in\Th}\int_T \HESS q : \PSigmaT{k-1}  \utens{\tau}_T.
      \end{equation}
      Then, further assuming $q \in H^{k+2}(\Omega)$, it holds:
      For all $\utens{\tau}_h \in  \utens{\Sigma}_h^{k-1}$,
      \begin{equation}\label{eq:adjoint_divdiv}
        |\Errdivdiv (q; \utens{\tau}_h)|
        \lesssim h^k \seminorm[H^{k+2}(\Th)]{q} \norm[\btens{\Sigma},h]{\utens{\tau}_h}.
      \end{equation}
  \end{enumerate}
\end{lemma}

\begin{proof}
  \underline{(i) \emph{Proof of \eqref{eq:adjoint_symcurl}.}}
  By definition \eqref{eq:L2 product in ET} of the local discrete $L^2$-product in $\utens{\Sigma}_h^{k-1}$ and the commutation property \eqref{eq:PSigmaT.circ.uCsymT=CsymT}, it holds that
  \begin{multline}\label{eq:error symcurlh}
    \Errsymcurl (\btens{\upsilon}; \uvec{v}_h)
    =
    \sum_{T\in \Th} \bigg[
      \int_{T} \PSigmaT{k-1}\ISigma{k-1}{\btens{\upsilon}}_{|T} : \Csym{k-1} \uvec{v}_T
      + s_{\btens{\Sigma},T}(\ISigma{k-1}{\btens{\upsilon}}_{|T}, \uCsym{k-1}\uvec{v}_T)
      \\
      + \int_{T} \VROT\btens{\upsilon}\cdot \PV{k}\uvec{v}_T
      \bigg].
  \end{multline}
  Accounting for Remark \ref{rem:validity.PVT}, it holds, for all $(\btens{\tau}_T)_{T\in\Th}\in\bigtimes_{T\in\Th}\tPoly{k-1}(T)$,
  \[
  \sum_{T\in\Th}\bigg[
  \int_T\PV{k}\uvec{v}_T\cdot\VROT\btens{\tau}_T
  + \int_T\Csym{k-1}\uvec{v}_T:\btens{\tau}_T
  - \sum_{E\in\ET}\omega_{TE}\int_E\bvec{v}_{\ET}\cdot(\btens{\tau}_T\,\tangent_E)
  \bigg]
  = 0.
  \]
  Subtracting this expression from \eqref{eq:error symcurlh}, we obtain
  \begin{equation*}
    \begin{aligned}
      \Errsymcurl (\btens{\upsilon}; \uvec{v}_h)
      &= 
      \sum_{T\in \Th}\bigg[
        \int_T \left(\PSigmaT{k-1}\ISigma{k-1}{\btens{\upsilon}}_{|T} - \btens{\tau}_T \right) : \Csym{k-1} \uvec{v}_T 
        +  s_{\btens{\Sigma},T}(\ISigma{k-1}{\btens{\upsilon}}_{|T}, \uCsym{k-1}\uvec{v}_T)
        \bigg]
      \\
      &\qquad
      +\sum_{T\in \Th} \bigg[
        \int_T \VROT(\btens{\upsilon}-\btens{\tau}_T)\cdot\PV{k}\uvec{v}_T +  \sum_{E\in \ET}  \omega_{TE} \int_E \bvec{v}_{\ET}\cdot(\btens{\tau}_T-\btens{\upsilon})\tangent_E
        \bigg],
    \end{aligned}
  \end{equation*}
  where we have additionally introduced $\btens{\upsilon}_{|E}\tangent_E$ into the boundary term using the fact that this quantity is single-valued if $E$ is an internal edge while it vanishes if $E\subset \partial \Omega$.
  Applying the integration by parts formula \eqref{eq:ibp.VT} to the third term leads to
  \begin{multline}\label{eq:error symcurl 2}
    \Errsymcurl (\btens{\upsilon}; \uvec{v}_h)
    =
    \sum_{T\in \Th}\bigg[
      \int_T \left(\PSigmaT{k-1}\ISigma{k-1}{\btens{\upsilon}}_{|T} - \btens{\tau}_T \right) \cdot \Csym{k-1} \uvec{v}_T
      +  s_{\btens{\Sigma},T}(\ISigma{k-1}{\btens{\upsilon}}_{|T}, \uCsym{k-1}\uvec{v}_T)
      \bigg]
    \\
    -\sum_{T\in \Th} \bigg[
      \int_T(\btens{\upsilon}-\btens{\tau}_T) \cdot \SYM\CURL  \PV{k}\uvec{v}_T
      - \sum_{E\in \ET}  \omega_{TE} \int_E (\bvec{v}_{\ET}-\PV{k}\uvec{v}_T)\cdot(\btens{\tau}_T-\btens{\upsilon}) \cdot \tangent_E
      \bigg].
  \end{multline}
  Take now $\btens{\tau}_T = \tlproj{k-1}{T}\btens{\upsilon}_{|T}$ for all $T\in\Th$.
  Using Cauchy-Schwarz inequalities in the right-hand side of \eqref{eq:error symcurl 2}
  followed by the approximation properties of $\PSigmaT{k-1}$ and $\tlproj{k-1}{T}$ (see, respectively, \cite[Proposition 14]{Di-Pietro.Droniou:22*1} and \cite[Theorem 1.45]{Di-Pietro.Droniou:20}) as well as the consistency property of the stabilisation term proved in \cite[Proposition 12]{Di-Pietro.Droniou:22*1}, we get
  \begin{multline*}
    |\Errsymcurl (\btens{\upsilon}; \uvec{v}_h)|
    \lesssim h^k \seminorm[\btens{H}^k(\Th;\Real^{2\times 2})]{\btens{\upsilon}}\bigg[
      \sum_{T\in\Th}\bigg(
      \norm[\btens{L}^2(T;\Real^{2\times 2})]{\Csym{k-1} \uvec{v}_T}^2
      + s_{\btens{\Sigma},T}(\uCsym{k-1}\uvec{v}_T, \uCsym{k-1}\uvec{v}_T)
      \\
      + \norm[\btens{L}^2(T;\Real^{2\times 2})]{\SYM\CURL  \PV{k}\uvec{v}_T}^2
      + h_T^{-1}\sum_{E\in\ET}\norm[\bvec{L}^2(E;\Real^2)]{\bvec{v}_{\ET}-\PV{k}\uvec{v}_T}^2
      \bigg)
      \bigg].
  \end{multline*}
  Let us consider the factor in square brackets.
  Using, respectively,
  \eqref{eq:PSigmaT.circ.uCsymT=CsymT} along with \eqref{eq:est.H1norm.PVT:term1} for the first term,
  the definition \eqref{eq:norms.V.Sigma} of the $L^2$-product norm on $\utens{\Sigma}_h^{k-1}$ for the second term,
  and \eqref{eq:est.H1norm.PVT} for the third and fourth terms,
  this factor is $\lesssim\norm[\btens{\Sigma},h]{\uCsym[h]{k-1}\uvec{v}_h}$,
  thus concluding the proof of \eqref{eq:adjoint_symcurl}.
  \medskip\\
\underline{(ii) \emph{Proof of \eqref{eq:adjoint_divdiv}.}}
Combining the definitions \eqref{eq:Edivdivh} of the adjoint consistency error and \eqref{eq:P.Sigma.T} of the tensor potential, it is inferred that, for all $(q_T)_{T\in\Th}\in\bigtimes_{T\in\Th}\Poly{k+1}(T)$,
\begin{equation*}
  \begin{aligned}
    \Errdivdiv (q,\utens{\tau}_h)
    &= 
    \sum_{T\in \Th}\bigg[
      \int_T (q - q_T)\, \DD{k-2}\utens{\tau}_T
      - \int_T\HESS(q - q_T) : \PSigmaT{k-1}\utens{\tau}_T
      \bigg]
    \\
    &\qquad
    + \sum_{T\in \Th}\omega_{TE}
    \sum_{E\in\ET}\left[
      \int_E \PSigmaE{k-1}\utens{\tau}_E\,\partial_{\normal_E} (q-q_T)
      - \int_E D_{\btens{\tau},E}\,(q-q_T)
    \right]
    \\
    &\qquad
    +
    \sum_{T\in\Th}\sum_{E\in\ET}\omega_{TE}\sum_{V\in\VE}\omega_{EV}(\btens{\tau}_V\normal_E\cdot\tangent_E)\,(q-q_T)(\bvec{x}_V),
  \end{aligned}
\end{equation*}
where the insertion of $q$ and $\partial_{\normal_E} q$ into the boundary integrals is possible since these quantities are continuous at internal edges and vanish on boundary edges.
  Taking $q_T = \lproj{k+1}{T} q$ for all $T\in\Th$ and using Cauchy--Schwarz inequalities followed by the approximation properties of $\lproj{k+1}{T}$, it is inferred that
  \begin{equation*}%\label{error.q.tau}
    \begin{aligned}
      |\Errdivdiv (q,\utens{\tau}_h)| 
      &\lesssim
      h^k\seminorm[H^{k+2}(\Th)]{q}
      \Bigg\{
      \sum_{T\in\Th}\bigg[
        h_T^4 \norm[L^2(T)]{\DD[T]{k-2}\utens{\tau}_T}^2
        + \norm[\bvec{L}^2(T;\Real^{2\times2})]{\PSigmaT{k-1}\utens{\tau}_T}^2
        \\
        &\quad
        +\sum_{E\in\ET}\Big(
        h_T\norm[L^2(E)]{\PSigmaE{k-1}\utens{\tau}_E}^2
        + h_T^3\norm[L^2(E)]{ D_{\btens{\tau},E}\,}^2
        + \sum_{V\in\VE} h_T^2|\btens{\tau}_V|^2
        \Big)
        \bigg]        
        \Bigg\rbrace^{\frac12}.
    \end{aligned}
  \end{equation*}
  Using \cite[Eq. (57)--(59)]{Di-Pietro.Droniou:22*1} for the first three terms and the definition \eqref{eq:tnorm.Sigma} for the last two, we infer that the quantity in braces is $\lesssim \tnorm[\btens{\Sigma},T]{\utens{\tau}_T}$, hence $\lesssim \norm[\btens{\Sigma},T]{\utens{\tau}_T}$ by the norm equivalence \eqref{eq:norm.equivalence} written for $\bullet = T$, thus concluding the proof of \eqref{eq:adjoint_divdiv}.
\end{proof}

%------------------------------------------------------------------------------%
%------------------------------------------------------------------------------%

\section{Serendipity DDR complex}\label{sec:serendipity.complex}

Denote, as in the previous section, by $k\ge 3$ the polynomial degree of the discrete complex.
Following \cite{Di-Pietro.Droniou:22}, we consider the construction illustrated in the following diagram:
\begin{equation}\label{eq:discrete.complex}
  \begin{tikzcd}[sep=huge]
    \arrow[<->]{d}{} \RT{1}(\Omega)
    \arrow{r}[above=2pt]{\IV[h]{k}} & \uvec{V}_h^k
    \arrow[bend left, dashed]{d}{\RV{h}}
    \arrow{r}[above=2pt]{\uCsym[h]{k-1}} & \utens{\Sigma}_h^{k-1}
    \arrow[bend left, dashed]{d}{\RS{h}}
    \arrow{r}[above=2pt]{\DD[h]{k-2}} &
    \arrow[<->]{d}{}
    \Poly{k-2}(\Th)
    \arrow{r}[above=2pt]{0} & 0
    \\
    \RT{1}(\Omega)
    \arrow{r}[above=2pt]{\sIV[h]{k}} & \suvec{V}_h^k
    \arrow[bend left]{u}{\EV{h}}
    \arrow{r}[above=2pt]{\suCsym[h]{k-1}} & \sutens{\Sigma}_h^{k-1}
    \arrow[bend left]{u}{\ES{h}}
    \arrow{r}[above=2pt]{\sDD[h]{k-2}} & \Poly{k-2}(\Th)
    \arrow{r}[above=2pt]{0} & 0,
  \end{tikzcd}
\end{equation}
where, according to \cite[Eqs. (2.2) and (2.4)]{Di-Pietro.Droniou:22}, we have set
\begin{equation}\label{eq:s.operators}
  \sIV[h]{k}\coloneq\RV{h}\IV[h]{k},\qquad
  \suCsym[h]{k-1}\coloneq\RS{h}\uCsym[h]{k-1}\EV{h},\qquad
  \sDD[h]{k-2}\coloneq\DD[h]{k-2}\ES{h}.
\end{equation}
The purpose of the rest of this section is to
\begin{itemize}
\item provide a precise definition of the extension and reduction operators $\EV{h}$, $\RV{h}$, $\ES{h}$, $\RS{h}$ as well as the spaces and operators that appear in the bottom (serendipity) complex;
\item prove that the properties of the top complex are inherited by the bottom complex.
\end{itemize}
This latter point makes the object of Theorems \ref{thm:homological.properties} and \ref{thm:analytical.properties} below, which are therefore the main results of this section.

As most of the developments are local, in what follows we denote by $T\in\Th$ a generic mesh element without necessarily specifying this fact at each occurrence.
As usual, a local version of diagram \eqref{eq:discrete.complex} on $T$ is obtained taking the restriction of the spaces and operators collecting the components attached to $T$ and, when present, to the edges and nodes that lie on its boundary.

\subsection{Estimate of symmetric tensor-valued polynomials}

Throughout the rest of this section, we work under the following assumption:

\begin{assumption}[Boundaries selection for serendipity spaces]\label{assum:choice.detaP}
  For each $T\in\Th$ element of the mesh, we select a set $\sET$ of $\eta_T\ge 2$ edges that are not pairwise aligned and such that, for all $E\in\sET$, $T$ lies entirely on one side of the hyperplane $H_E$ spanned by $E$.
  For all $E\in\sET$, denoting by $\bvec{x}_E$ its middle point and defining the scaled distance function to $H_E$ by $d_E(\bvec{x}) = h_E^{-1}\omega_{TE}(\bvec{x}-\bvec{x}_{E})\cdot\normal_E$, we assume the existence of a real number $\theta>0$ such that $d_E(\bvec{x}_{E'})\ge \theta$ for all $E\,,E'\in\sET$, $E\setminus\{E'\}$.
\end{assumption}

From this point on, the hidden constant in $a\lesssim b$ (see Section \ref{sec:setting:mesh}) will possibly depend also on the boundaries selection regularity parameter $\theta$.

\begin{lemma}[Estimate of symmetric tensor-valued polynomials]\label{lem:estimate.symmetric.poly}
  Let $m\ge 0$ and let Assumption \ref{assum:choice.detaP} hold.
  Let $T\in\Th$ be a mesh element. Then, for all $\btens{\tau}\in\tPoly{m}(T;\Symm)$, it holds
  \begin{multline}\label{eq:est.reconstruction}
    \norm[\btens{L}^2(T;\Real^{2\times 2})]{\btens{\tau}}
    \lesssim
    \norm[\btens{L}^2(T;\Real^{2\times 2})]{\Hproj{m-3}\btens{\tau}}
    +  \norm[\btens{L}^2(T;\Real^{2\times 2})]{\cHproj{m+2-\eta_T}\btens{\tau}}
    \\
    + \sum_{E\in\ET}\left(
    h_T^{\nicefrac12}\norm[L^2(E)]{\lproj{m-2}{E}(\btens{\tau}_{|E}\normal_E\cdot\normal_E)}
    + h_T^{\nicefrac32}\norm[L^2(E)]{\dE\btens{\tau}}
    \right)
    + h_T\sum_{V\in\VT}|\btens{\tau}(\bvec{x}_V)|.
  \end{multline}
\end{lemma}

\begin{remark}[Reduction by serendipity]
  Lemma \ref{lem:estimate.symmetric.poly} clearly shows which polynomial components $\utens{\Sigma}_T^{k-1}$ can be reduced by serendipity, namely the ones in $\cHoly{k-1}(T)$.
  As will become clear in what follows, in order to preserve the homological properties, a corresponding reduction of the components of $\uvec{V}_T^k$ in $\Poly{k-2}(T)$ is required; see Remark \ref{rem:dof.reduction} below.
\end{remark}

\begin{proof}[Proof of Lemma \ref{lem:estimate.symmetric.poly}]
  Let $\btens{\tau}\in\tPoly{m}(T;\Symm)$ and denote, for the sake of brevity, by $\mathcal{N}_T(\btens{\tau})$ the right-hand side of \eqref{eq:est.reconstruction}.
  
  We start by estimating $\DIV\VDIV\btens{\tau}$.
  Using the integration by parts formula \eqref{eq:ibp.SigmaT} with $q\in\Poly{m-2}(T)$, inserting $\Hproj{m-3}$ in front of $\btens{\tau}$ in the first term in the right-hand side (since $\HESS q\in\Holy{m-4}(T)\subset\Holy{m-3}(T)$) and $\lproj{m-2}{E}$ in front of $\btens{\tau}\normal_E\cdot\normal_E$ in the second term (since $\partial_{\normal_E}q\in\Poly{m-3}(E)\subset\Poly{m-2}(E)$), and using Cauchy--Schwarz along with discrete trace and inverse inequalities, we infer
  $
  \int_T\DIV\VDIV\btens{\tau}~q
  \lesssim h_T^{-2}\mathcal{N}_T(\btens{\tau})\norm[L^2(T)]{q}.
  $
  Taking $q = \DIV\VDIV\btens{\tau}$, simplifying, and multiplying both sides by $h_T^2$ yields
  \begin{equation}\label{eq:est.div.div.tau}
    h_T^2\norm[L^2(T)]{\DIV\VDIV\btens{\tau}}\lesssim\mathcal{N}_T(\btens{\tau}).
  \end{equation}

  By \eqref{eq:vPoly=Coly.oplys.cColy}, $\btens{\tau}$ can be decomposed as follows:
  \begin{equation}\label{eq:decomp.tau}
    \btens{\tau} = \SYM\CURL\bvec{v} + \btens{\upsilon},
  \end{equation}
  with $\bvec{v}\in\vPoly{m+1}(T;\Real^2)$ and $\btens{\upsilon}\in\cColy{m}(T)$.
  Since $\bvec{v}$ is defined up to a function in $\RT{1}(T)$, we can assume that
  \begin{equation}\label{eq:zero.mean}
    \text{
      $\int_{\partial T}\bvec{v} = \bvec{0}$\quad
      and\quad $\int_{\partial T}\partial_{\tangent,\partial T} v_{\tangent,\partial T} = 0$,
    }
  \end{equation}
  where we remind the reader that, as in Section \ref{sec:poincare:preliminary.results}, $\partial_{\tangent,\partial T}$ and $v_{\tangent,\partial T}$ are, respectively, the broken tangential derivative and tangential component of $\bvec{v}$ on $\partial T$.
  
  We next proceed to estimate the $L^2$-norms of the terms in the right-hand side of \eqref{eq:decomp.tau}.
To estimate $\norm[\btens{L}^2(T;\Real^{2\times 2})]{\btens{\upsilon}}$, we start with \eqref{eq:est.div.div.-1}, notice that $\DIV\VDIV\btens{\upsilon} = \DIV\VDIV\btens{\tau}$ (since $\DIV\VDIV\SYM\CURL = 0$), then invoke \eqref{eq:est.div.div.tau} to write
  \begin{equation}\label{eq:est.upsilon}
    \norm[\btens{L}^2(T;\Real^{2\times 2})]{\btens{\upsilon}}
    \lesssim h_T^2\norm[L^2(T)]{\DIV\VDIV\btens{\upsilon}}
    = h_T^2\norm[L^2(T)]{\DIV\VDIV\btens{\tau}}
    \lesssim\mathcal{N}_T(\btens{\tau}).
  \end{equation}

  To estimate $\norm[\btens{L}^2(T;\Real^{2\times 2})]{\SYM\CURL\bvec{v}}$, we start by using a discrete inverse inequality followed by \cite[Lemma 13]{Di-Pietro.Droniou:22} to write
  \begin{equation}\label{eq:est.norm.L2.sym.curl.v}
    \norm[\btens{L}^2(T;\Real^{2\times 2})]{\SYM\CURL\bvec{v}}
    \lesssim h_T^{-1} \norm[\bvec{L}^2(T;\Real^2)]{\bvec{v}}
    \lesssim h_T^{-1}\left(
    \norm[\bvec{L}^2(T;\Real^2)]{\vlproj{m+1-\eta_T}{T}\bvec{v}}
    + h_T^{\nicefrac12}\norm[\bvec{L}^2(\partial T;\Real^2)]{\bvec{v}}
    \right).
  \end{equation}
  We next proceed to estimate the terms in parentheses, starting with $\norm[\bvec{L}^2(\partial T;\Real^2)]{\bvec{v}}$.
  Since $\bvec{v}$  has zero average on $\partial T$, by a Poincaré--Wirtinger inequality we infer
  \begin{equation}\label{eq:est.v.pT:basic}
    \norm[\bvec{L}^2(\partial T;\Real^2)]{\bvec{v}}
    \lesssim h_T\norm[\bvec{L}^2(\partial T;\Real^2)]{\partial_{\tangent_{\partial T}}\bvec{v}}.
  \end{equation}
  Decomposing $\bvec{v}_{\partial T}$ into its normal and tangential components, and using triangle and H\"older inequalities along with the fact that $\normal_E$ and $\tangent_E$ are unit vectors, we get
  \begin{equation}\label{eq:est.ptE.v}
    \norm[\bvec{L}^2(\partial T;\Real^2)]{\partial_{\tangent_{\partial T}}\bvec{v}}
    \le\norm[L^2(\partial T)]{\partial_{\tangent_{\partial T}} v_{\normal,\partial T}}
    + \norm[L^2(\partial T)]{\partial_{\tangent_{\partial T}} v_{\tangent,\partial T}},
  \end{equation}
  where we remind the reader that $v_{\normal,\partial T}$ denotes the normal component of $\bvec{v}$ on $\partial T$.
  Since, for all $E\in\ET$, $\partial_{\tangent_E} v_{\normal,E} = (\SYM\CURL\bvec{v})\normal_E\cdot\normal_E = (\btens{\tau} - \btens{\upsilon})\normal_E\cdot\normal_E$ (cf., respectively, \cite[Eq. (3)]{Chen.Huang:20} and \eqref{eq:decomp.tau}), we can use a triangle inequality to write
  \begin{equation}\label{eq:est.ptE.vn}
    \begin{aligned}
      &\norm[L^2(\partial T)]{\partial_{\tangent_{\partial T}} v_{\normal,\partial T}}^2
      \\
      &\quad
      \lesssim
      \sum_{E\in\ET}\left(
      \norm[L^2(E)]{\btens{\tau}_{|E}\normal_E\cdot\normal_E}^2
      + \norm[L^2(E)]{\btens{\upsilon}_{|E}\normal_E\cdot\normal_E}^2
      \right)
      \\
      &\quad
      \lesssim
      \sum_{E\in\ET}\left(
      \norm[L^2(E)]{\lproj{m-2}{E}(\btens{\tau}_{|E}\normal_E\cdot\normal_E)}^2
      + h_E\sum_{V\in\VE}|\btens{\tau}(\bvec{x}_V)\normal_E\cdot\normal_E|^2
      + \norm[L^2(E)]{\btens{\upsilon}_{|E}\normal_E\cdot\normal_E}^2
      \right)
      \\
      &\quad
      \lesssim h_T^{-1}\mathcal{N}_T(\btens{\tau})^2,
    \end{aligned}
  \end{equation}
  where the second line follows from \eqref{eq:est.normL2.1d} applied to $\varphi = \btens{\tau}_{|E}\normal_E\cdot\normal_E$,
  while the conclusion follows using the definition of $\mathcal{N}_T(\btens{\tau})$ for the first two terms and a discrete trace inequality followed by \eqref{eq:est.upsilon} for the last term.
  To estimate $\norm[L^2(\partial T)]{\partial_{\tangent_{\partial T}} v_{\tangent,\partial T}}$, we proceed in a similar way as for the estimate of $\term_2$ in Proposition \ref{prop:est.boundary.fct} (using the fact that $\partial_{\tangent_{\partial T}} v_{\tangent,\partial T}$ has zero average on $T$) to infer
  \[
  \begin{aligned}
    \norm[L^2(\partial T)]{\partial_{\tangent_{\partial T}} v_{\tangent,\partial T}}
    &\lesssim
    h_T\norm[L^2(\partial T)]{\partial_{\tangent_{\partial T}}^2 v_{\tangent,\partial T}}
    + h_T^{\nicefrac12}\sum_{V\in\VT}|\SYM\CURL\bvec{v}(\bvec{x}_V)|
    \\
    &=
    h_T\left(
    \sum_{E\in\ET}\norm[L^2(E)]{\dE(\btens{\tau} - \btens{\upsilon})}^2
    \right)^{\nicefrac12}
    + h_T^{\nicefrac12}\sum_{V\in\VT}|\btens{\tau}(\bvec{x}_V) - \btens{\upsilon}(\bvec{x}_V)|
    \\
    &\lesssim
    h_T\left(
    \sum_{E\in\ET}\norm[L^2(E)]{\dE\btens{\tau}}^2
    \right)^{\nicefrac12}
    + h_T^{\nicefrac12}\sum_{V\in\VT}|\btens{\tau}(\bvec{x}_V)|
    + h_T^{-\nicefrac12}\norm[\btens{L}^2(T;\Real^{2\times 2})]{\btens{\upsilon}},
  \end{aligned}
  \]
  where, to pass to the second line, we have used \cite[Eq. (4)]{Chen.Huang:20} to write $\partial_{\tangent_E}^2 v_{\tangent,E} = \dE\SYM\CURL\bvec{v} = \dE(\btens{\tau} - \btens{\upsilon})$ for the first term and \eqref{eq:decomp.tau} for the second,
  while, to pass to the third line, we have used triangle inequalities followed by discrete inverse and trace inequalities along with $\card(\ET) = \card(\VT) \lesssim 1$ to treat the terms containing $\btens{\upsilon}$.
  Combining the definition of $\mathcal{N}_T(\btens{\tau})$ with \eqref{eq:est.upsilon}, we conclude that
  \begin{equation}\label{eq:est.ptE.vt}
  \norm[L^2(\partial T)]{\partial_{\tangent_{\partial T}} v_{\tangent,\partial T}}
  \lesssim h_T^{-\nicefrac12}\mathcal{N}_T(\btens{\tau}).
  \end{equation}
  Plugging \eqref{eq:est.ptE.vn} and \eqref{eq:est.ptE.vt} into \eqref{eq:est.ptE.v} and the resulting inequality into \eqref{eq:est.v.pT:basic}, we conclude that
  \begin{equation}\label{eq:est.v.pT}
    \norm[\bvec{L}^2(\partial T;\Real^2)]{\bvec{v}}
    \lesssim h_T^{\nicefrac12}\mathcal{N}_T(\btens{\tau}).
  \end{equation}

  It only remains to estimate $\norm[\bvec{L}^2(T;\Real^2)]{\vlproj{m+1-\eta_T}{T}\bvec{v}}$ in \eqref{eq:est.norm.L2.sym.curl.v}.
  To this end, we start using the integration by parts formula \eqref{eq:ibp.VT} to write, for all $\btens{\phi}\in\cHoly{m+2-\eta_T}(T)$,
  \[
  \begin{aligned}
    \int_T\bvec{v}\cdot\VROT\btens{\phi}
    &= -\int_T\SYM\CURL\bvec{v}:\btens{\phi}
    + \sum_{E\in\ET}\omega_{TE}\int_E\bvec{v}\cdot(\btens{\phi}\,\tangent_E)
    \\
    &= -\int_T\cHproj{m+2-\eta_T}\btens{\tau}:\btens{\phi}
    +\int_T\btens{\upsilon}:\btens{\phi}
    + \sum_{E\in\ET}\omega_{TE}\int_E\bvec{v}\cdot(\btens{\phi}\,\tangent_E),
  \end{aligned}
  \]
  where we have used \eqref{eq:decomp.tau} to pass to the second line and invoked its definition to insert $\cHproj{m+2-\eta_T}$ into the first term.
  We then apply Cauchy--Schwarz and discrete trace inequalites to get
  \[
  \begin{aligned}
    \left|
    \int_T\bvec{v}\cdot\VROT\btens{\phi}
    \right|
    &\lesssim\left(
    \norm[\btens{L}^2(T;\Real^{2\times 2})]{\cHproj{m+2-\eta_T}\btens{\tau}}
    + \norm[\btens{L}^2(T;\Real^{2\times 2})]{\btens{\upsilon}}
    + h_T^{-\nicefrac12}\norm[\bvec{L}^2(\partial T;\Real^2)]{\bvec{v}}
    \right)\norm[\btens{L}^2(T;\Real^{2\times 2})]{\btens{\phi}}
    \\
    &\lesssim h_T\mathcal{N}_T(\btens{\tau})\norm[\bvec{L}^2(T;\Real^2)]{\VROT\btens{\phi}},
  \end{aligned}
  \]
  where the conclusion follows using the definition of $\mathcal{N}_T(\btens{\tau})$ along with \eqref{eq:est.upsilon} and \eqref{eq:est.v.pT} for the first factor and \eqref{eq:est.vrot.-1} for the second.
  Taking the supremum over $\btens{\phi}\in\cHoly{m+2-\eta_T}(T)$ such that $\norm[\bvec{L}^2(T;\Real^2)]{\VROT\btens{\phi}} = 1$ finally yields
  \begin{equation*}%% \label{eq:est.proj.v}
    \norm[\bvec{L}^2(T;\Real^2)]{\vlproj{m+1-\eta_T}{T}\bvec{v}}\lesssim h_T\mathcal{N}_T(\btens{\tau}).
  \end{equation*}
  Plugging this result and \eqref{eq:est.v.pT} into \eqref{eq:est.norm.L2.sym.curl.v} gives $\norm[\btens{L}^2(T;\Real^{2\times 2})]{\SYM\CURL\bvec{v}}\lesssim\mathcal{N}_T(\btens{\tau})$ which, combined with \eqref{eq:est.upsilon}, gives \eqref{eq:est.reconstruction} after taking the $L^2$-norm of \eqref{eq:decomp.tau} and using a triangle inequality in the right-hand side.
\end{proof}

\subsection{Serendipity problem}

Recalling Assumption \ref{assum:choice.detaP}, we let
\begin{equation}\label{eq:ell.T}
  \ell_T\coloneq k - \eta_T\le k - 2;  
\end{equation}
Given a linear form $\mathcal{L}_T:\tPoly{k-1}(T;\Symm)\times\cHoly{\ell_T+1}(T)\to\Real$, we consider the following problem:
Find $(\btens{\sigma},\btens{\lambda})\in\tPoly{k-1}(T;\Symm)\times\cHoly{\ell_T+1}(T)$ such that
\begin{equation}\label{eq:serendipity.problem}
  \mathcal{A}_T((\btens{\sigma},\btens{\lambda}),(\btens{\tau},\btens{\mu}))
  = \mathcal{L}_T(\btens{\tau},\btens{\mu})
  \qquad\forall(\btens{\tau},\btens{\mu})\in\tPoly{k-1}(T;\Symm)\times\cHoly{\ell_T+1}(T),
\end{equation}
where the bilinear form $\mathcal{A}_T:\big[\tPoly{k-1}(T;\Symm)\times\cHoly{\ell_T+1}(T)\big]^2\to\Real$ is such that
\begin{equation}\label{eq:AT}
  \begin{aligned}
    \mathcal{A}_T((\btens{\upsilon},\btens{\nu}),(\btens{\tau},\btens{\mu}))
    &\coloneq
    h_T^4\int_T\DIV\VDIV\btens{\upsilon}~\DIV\VDIV\btens{\tau}
    \\
    &\quad
    + h_T\sum_{E\in\ET}\int_E\lproj{k-3}{E}(\btens{\upsilon}\normal_E\cdot\normal_E)~\lproj{k-3}{E}(\btens{\tau}\normal_E\cdot\normal_E)
    + h_T^3\sum_{E\in\ET}\int_E
    \dE\btens{\upsilon}~\dE\btens{\tau}
    \\
    &\quad
    + h_T^2\sum_{V\in\VT}\btens{\upsilon}(\bvec{x}_V):\btens{\tau}(\bvec{x}_V)
    + \int_T\btens{\upsilon}:\btens{\mu}
    - \int_T\btens{\tau}:\btens{\nu}.
  \end{aligned}
\end{equation}

\begin{lemma}[Inf-sup condition and well-posedness of the serendipity problem]
  The following inf-sup condition holds: For all $(\btens{\upsilon},\btens{\nu})\in\tPoly{k-1}(T;\Symm)\times\cHoly{\ell_T+1}(T)$,
  \begin{equation}\label{eq:inf-sup}
    \norm[T]{(\btens{\upsilon},\btens{\nu})}
    \lesssim\sup_{(\btens{\tau},\btens{\mu})\in\tPoly{k-1}(T;\Symm)\times\cHoly{\ell_T+1}(T)\setminus\{(\btens{0},\btens{0})\}}
    \frac{\mathcal{A}_T((\btens{\upsilon},\btens{\nu}),(\btens{\tau},\btens{\mu}))}{\norm[T]{(\btens{\tau},\btens{\mu})}}
    \eqcolon\$,
  \end{equation}
  where $\norm[T]{(\btens{\upsilon},\btens{\nu})}\coloneq\norm[\btens{L}^2(T;\Real^{2\times 2})]{\btens{\upsilon}} + \norm[\btens{L}^2(T;\Real^{2\times 2})]{\btens{\nu}}$.
  Hence, denoting by $\norm[T]{\mathcal{L}_T}$ the norm of $\mathcal{L}_T$ dual to $\norm[T]{{\cdot}}$, problem \eqref{eq:serendipity.problem} admits a unique solution that satisfies
  \begin{equation}\label{eq:a-priori}
    \norm[T]{(\btens{\sigma},\btens{\lambda})}
    \lesssim
    \norm[T]{\mathcal{L}_T}.
  \end{equation}
\end{lemma}

\begin{proof}
  The existence and uniqueness of a solution to \eqref{eq:serendipity.problem} as well as the a priori estimate \eqref{eq:a-priori} classically follow from \eqref{eq:inf-sup}.
  Let us establish the latter condition for a given $(\btens{\upsilon},\btens{\nu})\in\tPoly{k-1}(T;\Symm)\times\cHoly{\ell_T+1}(T)$.
  Taking $(\btens{\tau},\btens{\mu}) = (\btens{\upsilon},\btens{\nu})$ in the expression \eqref{eq:AT} of $\mathcal{A}_T$, we obtain, after using the uniform bound on the number of edges and vertices of $T$ that holds by mesh regularity,
  \begin{multline}\label{eq:inf-sup:1}
  h_T^4\norm[L^2(T)]{\DIV\VDIV\btens{\upsilon}}^2
  + \sum_{E\in\ET}\left(
  h_T\norm[L^2(E)]{\lproj{k-3}{E}(\btens{\upsilon}\normal_E\cdot\normal_E)}^2
  + h_T^3\norm[L^2(E)]{\dE\btens{\upsilon}}^2
  \right)
  \\
  + h_T^2\sum_{V\in\VT}|\btens{\upsilon}(\bvec{x}_V)|^2
  = \mathcal{A}_T((\btens{\upsilon},\btens{\nu}),(\btens{\upsilon},\btens{\nu}))
  \le\$\norm[T]{(\btens{\upsilon},\btens{\nu})}.
  \end{multline}
  
  We next observe that, for any $q\in\Poly{k-2\perp 1}(T)$, writing the integration by parts formula \eqref{eq:ibp.SigmaT} and inserting the appropriate $L^2$-orthogonal projectors according to their definition, it holds
  \[
  \begin{aligned}
    \int_T\Hproj{k-4}\btens{\upsilon}:\HESS q
    &=
    \int_T\DIV\VDIV\btens{\upsilon}~q
    + \sum_{E\in\ET}\omega_{TE}\left(
    \int_E\lproj{k-3}{T}(\btens{\upsilon}\normal_E\cdot\normal_E)~\partial_{\normal_E}q
    - \int_E\dE\btens{\upsilon}~q
    \right)
    \\
    &\quad
    + \sum_{E\in\ET}\omega_{TE}\sum_{V\in\VE}\omega_{EV}(\btens{\upsilon}(\bvec{x}_V)\normal_E\cdot\tangent_E)~q(\bvec{x}_V).
  \end{aligned}
  \]
  By Proposition \ref{prop:continuity.inverses}, we can select $q$ such that $\HESS q = \Hproj{k-4}\btens{\upsilon}$ and $\norm[L^2(T)]{q}\lesssim h_T^2\norm[\btens{L}^2(T;\Real^{2\times 2})]{\HESS q}$.
  Applying Cauchy--Schwarz and discrete trace and inverse inequalities to estimate the right hand side of the resulting expression, simplifying, and raising to the square, we obtain
  \begin{multline}\label{eq:inf-sup:2}
    \norm[\btens{L}^2(T;\Real^{2\times 2})]{\Hproj{k-4}\btens{\upsilon}}^2
    \lesssim
    h_T^4\norm[L^2(T)]{\DIV\VDIV\btens{\upsilon}}^2
    + \sum_{E\in\ET}\left(
    h_T\norm[L^2(E)]{\lproj{k-3}{E}(\btens{\upsilon}\normal_E\cdot\normal_E)}^2
    + h_T^3\norm[L^2(E)]{\dE\btens{\upsilon}}^2
    \right)
    \\
    + h_T^2\sum_{V\in\VT}|\btens{\upsilon}(\bvec{x}_V)|^2
    \lesssim\$\norm[T]{(\btens{\upsilon},\btens{\nu})},    
  \end{multline}
  where the conclusion follows from \eqref{eq:inf-sup:1}.

  Finally, writing the definition \eqref{eq:AT} of $\mathcal{A}_T$ with $(\btens{\tau},\btens{\lambda}) = (\btens{0},\cHproj{\ell_T+1}\btens{\upsilon})$, we get
  \begin{equation}\label{eq:inf-sup:3}
    \norm[\btens{L}^2(T;\Real^{2\times 2})]{\cHproj{\ell_T+1}\btens{\upsilon}}^2
    = \mathcal{A}_T((\btens{\upsilon},\btens{\nu}),(\btens{0},\cHproj{\ell_T+1}\btens{\upsilon}))
    \le\$\norm[T]{(\btens{0},\cHproj{\ell_T+1}\btens{\upsilon})}
    \le\$\norm[T]{(\btens{\upsilon},\btens{\nu})},
  \end{equation}
  where the conclusion follows from the uniform $L^2$-boundedness of $\cHproj{\ell_T+1}$.
  Summing \eqref{eq:inf-sup:1}, \eqref{eq:inf-sup:2}, and \eqref{eq:inf-sup:3} and recalling \eqref{eq:est.reconstruction}, we infer
  \begin{equation}\label{eq:inf-sup:norm.upsilon}
    \norm[\btens{L}^2(T;\Real^{2\times 2})]{\btens{\upsilon}}^2
    \lesssim\$\norm[T]{(\btens{\upsilon},\btens{\nu})}.
  \end{equation}

  To estimate the $L^2$-norm of $\btens{\nu}$, we take $(\btens{\tau},\btens{\mu}) = (-\btens{\nu},\btens{0})$ in the expression \eqref{eq:AT} of $\mathcal{A}_T$ (this is possible since $\btens{\nu}\in\cHoly{\ell_T+1}(T)\subset\tPoly{k-1}(T;\Symm)$ owing to \eqref{eq:ell.T}) and, after using Cauchy--Schwarz, discrete trace, and inverse inequalities, simplifying, and raising to the square, we obtain
  \begin{equation}\label{eq:inf-sup:norm.nu}
    \norm[\btens{L}^2(T;\Real^{2\times 2})]{\btens{\nu}}^2
    \lesssim\$ + \$\norm[T]{(\btens{\upsilon},\btens{\nu})}.
  \end{equation}
  Summing \eqref{eq:inf-sup:norm.upsilon} and \eqref{eq:inf-sup:norm.nu}, using Young's inequality for the rightmost term in \eqref{eq:inf-sup:norm.nu}, and taking the square root of the resulting expression gives \eqref{eq:inf-sup}.
\end{proof}

\subsection{Serendipity spaces}

Recalling the definition \eqref{eq:ell.T} of $\ell_T$, the local serendipity spaces are:

\begin{align*}%\label{eq:sVT}
  \suvec{V}_T^k&\coloneq\Big\{
  \begin{aligned}[t]
    &\suvec{v}_T=
    \big(
    \svec{v}_T,
    (\svec{v}_E)_{E\in\Eh},
    (\svec{v}_V, \btens{G}_{\bvec{v},V})_{V\in\Vh}
    \big)\st
    \\
    &\qquad\text{
      $\svec{v}_T\in\vPoly{\ell_T}(T;\Real^2)$,
    }
    \\
    &\qquad\text{
      $\svec{v}_E\in\vPoly{k-4}(E;\Real^2)$ for all $E\in\ET$,
    }
    \\
    &\qquad\text{      
      $\svec{v}_V\in\Real^2$
      and $\btens{G}_{\svec{v},V}\in\Real^{2\times 2}$
      for all $V\in\VT$
    }
    \Big\},
  \end{aligned}
  \\%\label{eq:sSigmaT}
  \sutens{\Sigma}_T^{k-1}&\coloneq\Big\{
  \begin{aligned}[t]
    &\sutens{\tau}_T
    =\big(
    (\stens{\tau}_{\cvec{H},T}, \stens{\tau}_{\cvec{H},T}^\compl,
    (\ser{\tau}_E,D_{\stens{\tau},E})_{E\in\Eh},
    (\stens{\tau}_V)_{V\in\Vh}
    \big)\st
    \\
    &\qquad\text{
      $\stens{\tau}_{\cvec{H},T}\in\Holy{k-4}(T)$ and $\stens{\tau}_{\cvec{H},T}^\compl\in\cHoly{\ell_T+1}(T)$,
    }
    \\
    &\qquad\text{
      $\ser{\tau}_E\in\Poly{k-3}(E)$
      and $D_{\stens{\tau},E}\in\Poly{k-2}(E)$ for all $E\in\ET$,
    }
    \\
    &\qquad\text{
      $\stens{\tau}_V\in\Symm$ for all $V\in\VT$
    }
    \Big\}.
  \end{aligned}
\end{align*}
Global spaces on $\Mh$ are obtained enforcing the single-valuedness of polynomial components located at internal edges and nodes.

\begin{remark}[Serendipity DOFs reduction]\label{rem:dof.reduction}
  Comparing the above expressions with those of the corresponding full spaces (i.e., the restrictions of \eqref{eq:VT} and \eqref{eq:SigmaT} to $T$) shows that the serendipity DOFs reduction acts on the components $\svec{v}_T$ and $\stens{\tau}_{\cvec{H},T}^\compl$, whose polynomial degrees are reduced from $(k-2,k-1)$ to $(\ell_T,\ell_T+1)$.
  Recalling \eqref{eq:ell.T}, the choice $\eta_T = 2$ therefore corresponds to no serendipity.
\end{remark}

In what follows, the component norms defined in Section \ref{sec:l2-product.norms} are applied to the elements of the serendipity spaces $\suvec{V}_h^k$ and $\sutens{\Sigma}_h^{k-1}$ after observing that the latter inject in the full spaces $\uvec{V}_h^k$ and $\utens{\Sigma}_h^{k-1}$ (notice that, by \eqref{eq:ell.T}, $\vPoly{\ell_T}(T;\Real^2)\subset\vPoly{k-2}(T;\Real^2)$ and $\cHoly{\ell_T+1}(T)\subset\cHoly{k-1}(T)$ for all $T\in\Th$).

\subsection{Serendipity operators}

The serendipity operators $\SV:\suvec{V}_T^k\to\tPoly{k-1}(T;\Symm)$ and $\SSigma:\sutens{\Sigma}_T^{k-1}\to\tPoly{k-1}(T;\Symm)$ are such that, for all $(\suvec{v}_T,\sutens{\tau}_T)\in\suvec{V}_T^k\times\sutens{\Sigma}_T^{k-1}$, $\SV\suvec{v}_T$ and $\SSigma\sutens{\tau}_T$ are the first components of the solutions of problem \eqref{eq:serendipity.problem} with right-hand side linear form $\mathcal{L}_T$ respectively equal to
\begin{equation}\label{eq:LT.V}
  \begin{aligned}
    \mathcal{L}_{\bvec{V},T}(\suvec{v}_T;\btens{\tau},\btens{\mu})
    &= h_T\sum_{E\in\ET}\int_E\lproj{k-3}{E}(\partial_{\tangent_E}\svec{v}_{\ET}\cdot\normal_E)~\lproj{k-3}{E}(\btens{\tau}\normal_E\cdot\normal_E)
    \\
    &\quad
    + h_T^3\sum_{E\in\ET}\int_E(\partial_{\tangent_E}^2\svec{v}_{\ET|E}\cdot\tangent_E)~\dE\btens{\tau}
    \\
    &\quad
    + h_T^2\sum_{V\in\VT}\mathbb{C}\btens{G}_{\svec{v},V}:\btens{\tau}(\bvec{x}_V)
    -\int_T\svec{v}_T\cdot\VROT\btens{\mu}
    + \sum_{E\in\ET}\omega_{TE}\int_E\svec{v}_{\ET}\cdot(\btens{\mu}\tangent_E)
  \end{aligned}
\end{equation}
and
\begin{equation}\label{eq:LT.Sigma}
  \begin{aligned}
    \mathcal{L}_{\btens{\Sigma},T}(\sutens{\upsilon}_T;\btens{\tau},\btens{\mu})
    &= \begin{aligned}[t]
      h_T^4\bigg[&
        \int_T\stens{\upsilon}_{\cvec{H},T}:\HESS\DIV\VDIV\btens{\tau}
        \\
        &
        -\sum_{E\in\ET}\omega_{TE}\left(
        \int_E\ser{\upsilon}_E~\partial_{\normal_E}(\DIV\VDIV\btens{\tau})
        -\int_E D_{\stens{\upsilon},E}~\DIV\VDIV\btens{\tau}
        \right)
        \\
        &
        -\sum_{E\in\ET}\omega_{TE}\sum_{V\in\VE}\omega_{EV}(\stens{\upsilon}_V\normal_E\cdot\tangent_E)
        \DIV\VDIV\btens{\tau}(\bvec{x}_V)
      \bigg]
    \end{aligned}
    \\
    &\quad
    + h_T\sum_{E\in\ET}\int_E\ser{\upsilon}_E~\lproj{k-3}{E}(\btens{\tau}\normal_E\cdot\normal_E)
    + h_T^3\sum_{E\in\ET}\int_E D_{\stens{\upsilon},E}~\dE\btens{\tau}
    \\
    &\quad
    + h_T^2\sum_{V\in\VT}\stens{\upsilon}_V:\btens{\tau}(\bvec{x}_V)
    + \int_T\stens{\upsilon}_{\cvec{H},T}^\compl:\btens{\mu}.
  \end{aligned}
\end{equation}

\begin{remark}[Alternative expression for $\mathcal{L}_{\btens{\Sigma},T}(\sutens{\upsilon}_T;\cdot)$]
  Using the injection $\sutens{\Sigma}_T^{k-1}\hookrightarrow\utens{\Sigma}_T^{k-1}$ to apply the operator $\DD{k-2}$ defined by \eqref{eq:DDT} to elements of $\suvec{\Sigma}_T^{k-1}$, we have the following equivalent reformulation of $\mathcal{L}_{\btens{\Sigma},T}(\sutens{\upsilon}_T;\cdot)$:
  For all $(\btens{\tau},\btens{\mu})\in\tPoly{k-1}(T;\Symm)\times\cHoly{\ell_T+1}(T)$,
  \begin{equation}\label{eq:LT.Sigma:bis}
    \begin{aligned}
      \mathcal{L}_{\btens{\Sigma},T}(\sutens{\upsilon}_T;\btens{\tau},\btens{\mu})
      &=
      h_T^4\int_T\DD{k-2}\sutens{\upsilon}_T~\DIV\VDIV\btens{\tau}
      \\
      &\quad
      + h_T\sum_{E\in\ET}\int_E\ser{\upsilon}_E~\lproj{k-3}{E}(\btens{\tau}\normal_E\cdot\normal_E)
      + h_T^3\sum_{E\in\ET}\int_E D_{\stens{\upsilon},E}~\dE\btens{\tau}
      \\
      &\quad
      + h_T^2\sum_{V\in\VT}\stens{\upsilon}_V:\btens{\tau}(\bvec{x}_V)
      + \int_T\stens{\upsilon}_{\cvec{H},T}^\compl:\btens{\mu}.
    \end{aligned}
  \end{equation}
\end{remark}

\subsection{Reduction and extension operators}

The restriction operators $\RV{T}:\uvec{V}_T^k\to\suvec{V}_T^k$ and $\RS{T}:\utens{\Sigma}_T^{k-1}\to\sutens{\Sigma}_T^{k-1}$ are defined taking $L^2$-orthogonal projections on the reduced component spaces:
For all $(\uvec{v}_T,\utens{\tau}_T)\in\uvec{V}_T^k\times\utens{\Sigma}_T^{k-1}$,
\begin{align}\label{eq:RV}
  \RV{T}\uvec{v}_T
  &\coloneq
  \Big(
  \vlproj{\ell_T}{T}\bvec{v}_T, (\bvec{v}_E)_{E\in\ET}, (\bvec{v}_V,\btens{G}_{\bvec{v},V})_{V\in\VT}
  \Big),
  \\ \label{eq:RS}
  \RS{T}\utens{\tau}_T
  &\coloneq
  \Big(
  \btens{\tau}_{\cvec{H},T}, \cHproj{\ell_T+1}\btens{\tau}_{\cvec{H},T}^\compl,
  (\tau_E, D_{\btens{\tau},E})_{E\in\ET},
  (\btens{\tau}_V)_{V\in\VT}
  \Big).
\end{align}
According to \cite[Eq. (2.4)]{Di-Pietro.Droniou:22}, the interpolators on the spaces $\suvec{V}_T^k$ and $\sutens{\Sigma}_T^{k-1}$ are respectively given by
\begin{equation}\label{eq:sinterpolators}
  \text{
    $\sIV{k}\coloneq\RV{T}\IV{k}$\quad and\quad
    $\sISigma{k-1}\coloneq\RS{T}\ISigma{k-1}$.
  }
\end{equation}

The extension operators $\EV{T}:\suvec{V}_T^k\to\uvec{V}_T^k$ and $\ES{T}:\sutens{\Sigma}_T^{k-1}\to\utens{\Sigma}_T^{k-1}$ are such that, for all $(\suvec{v}_T,\sutens{\tau}_T)\in\suvec{V}_T^k\times\sutens{\Sigma}_T^{k-1}$,
\begin{align}\label{eq:EV}
  \EV{T}\suvec{v}_T
  &\coloneq
  \Big(
  \EPT\suvec{v}_T, (\svec{v}_E)_{E\in\ET}, (\svec{v}_V,\btens{G}_{\svec{v},V})_{V\in\VT}
  \Big),
  \\ \label{eq:ES}
  \ES{T}\sutens{\tau}_T
  &\coloneq
  \Big(
  \stens{\tau}_{\cvec{H},T}, \cHproj{k-1}\SSigma\sutens{\tau}_T,
  (\ser{\tau}_E, D_{\stens{\tau},E})_{E\in\ET},
  (\stens{\tau}_V)_{V\in\VT}
  \Big),
\end{align}
where $\EPT:\sutens{\Sigma}_T^{k-1}\to\vPoly{k-2}(T)$ is such that, for all $\suvec{v}_T\in\suvec{V}_T^k$,
\begin{equation}\label{eq:EPT}
  \int_T\EPT\suvec{v}_T\cdot\VROT\btens{\tau}
  = -\int_T\SV\suvec{v}_T:\btens{\tau}
  + \sum_{E\in\ET}\omega_{TE}\int_E\svec{v}_{\ET}\cdot(\btens{\tau}\tangent_E)
  \qquad
  \forall\btens{\tau}\in\cHoly{k-1}(T).
\end{equation}
The fact that $\EPT\suvec{v}_T$ is uniquely defined by the above equation follows from the fact that $\VROT:\cHoly{k-1}(T)\to\vPoly{k-2}(T)$ is an isomorphism (see Proposition \ref{prop:continuity.inverses}).
Owing to the definition of the serendipity operator $\SV$ and bound \eqref{eq:a-priori}, it can be checked that, for all $T\in\Th$,
\begin{equation}\label{eq:EVT:continuity}
  \tnorm[\bvec{V},T]{\EV{T}\suvec{v}_T}
  \lesssim\tnorm[\svec{V},T]{\suvec{v}_T}.
\end{equation}

\subsection{Preliminary results}

\begin{lemma}[Polynomial consistency of the serendipity and extension operators]
  It holds:
  \begin{alignat}{4}\label{eq:SV:polynomial.consistency}
    \SV\sIV{k}\bvec{v} &= \SYM\CURL\bvec{v}
    &\qquad&
    \forall\bvec{v}\in\vPoly{k}(T;\Real^2),
    \\ \label{eq:EV:polynomial.consistency}
    \EV{T}\sIV{k}\bvec{v} &= \IV{k}\bvec{v}
    &\qquad&
    \forall\bvec{v}\in\vPoly{k}(T;\Real^2),
    \\ \label{eq:SS:polynomial.consistency}
    \SSigma\sISigma{k-1}\btens{\upsilon} &= \btens{\upsilon}
    &\qquad&
    \forall\btens{\upsilon}\in\tPoly{k-1}(T;\Symm),
    \\ \label{eq:ES:polynomial.consistency}
    \ES{T}\sISigma{k-1}\btens{\upsilon} &= \ISigma{k-1}\btens{\upsilon}
    &\qquad&
    \forall\btens{\upsilon}\in\tPoly{k-1}(T;\Symm).
  \end{alignat}
\end{lemma}

\begin{proof}
  \underline{(i) \emph{Proof of \eqref{eq:SV:polynomial.consistency}.}}
  Let $\suvec{v}_T\coloneq\sIV{k}\bvec{v}$.
  It suffices to show that $(\SYM\CURL\bvec{v},\btens{0})$ solves the problem defining $\SV\suvec{v}_T$, i.e., \eqref{eq:serendipity.problem} with linear form $\mathcal{L}_T(\cdot) = \mathcal{L}_{\bvec{V},T}(\suvec{v}_T;\cdot)$ given by \eqref{eq:LT.V}.
  Recalling the definition \eqref{eq:AT} of the bilinear form $\mathcal{A}_T$, we have, for all $(\btens{\tau},\btens{\mu})\in\tPoly{k-1}(T;\Symm)\times\cHoly{\ell_T+1}(T)$,
  \[
  \begin{aligned}
    \mathcal{A}_T((\SYM\CURL\bvec{v},\btens{0}),(\btens{\tau},\btens{\mu}))
    &= h_T^4\int_T\cancel{\DIV\VDIV\SYM\CURL\bvec{v}}~\DIV\VDIV\btens{\tau}
    \\
    &\quad
    + h_T\sum_{E\in\ET}\int_E\lproj{k-3}{E}\underbrace{(\SYM\CURL\bvec{v}\normal_E\cdot\normal_E}_{\partial_{\tangent_E}\bvec{v}_{|E}\cdot\normal_E})~\lproj{k-3}{E}(\btens{\tau}\normal_E\cdot\normal_E)
    \\
    &\quad
    + h_T^3\sum_{E\in\ET}\int_E\underbrace{\dE\SYM\CURL\bvec{v}}_{\partial_{\tangent_E}^2\bvec{v}_{|E}\cdot\tangent_E}~\dE\btens{\tau}
    + h_T^2\sum_{V\in\VT}\SYM\CURL\bvec{v}(\bvec{x}_V):\btens{\tau}(\bvec{x}_V),
    \\
    &\quad
      + \int_T\SYM\CURL\bvec{v}:\btens{\mu}.
  \end{aligned}
  \]
  where we have used \cite[Lemma 2.2]{Chen.Huang:20} for the second and third term.
    Using the integration by parts formula \eqref{eq:ibp.VT},
    observing that $\VROT\btens{\mu}\in\vPoly{\ell_T}(T)$ to insert $\vlproj{\ell_T}{T}$ into the first term and
    that $\bvec{v}_{|\partial T} = \svec{v}_{\ET}$ by polynomial consistency of this trace reconstruction,
    \[
    \int_T\SYM\CURL\bvec{v}:\btens{\mu}
    = -\int_T\underbrace{\vlproj{\ell_T}{T}\bvec{v}}_{=\,\svec{v}_T}\cdot\VROT\btens{\mu}
    + \sum_{E\in\ET}\omega_{TE}\int_{E}\svec{v}_{\ET}\cdot(\btens{\mu}\,\bvec{t}_E).
    \]
    Using the above relation, again $\bvec{v}_{|\partial T} = \svec{v}_{\ET}$, and further noticing that $\SYM\CURL\bvec{v}(\bvec{x}_V) = \mathbb{C}\btens{G}_{\svec{v},V}$ for all $V\in\VT$ by definition of the interpolator, we have, recalling the definition \eqref{eq:LT.V} of $\mathcal{L}_{\bvec{V},T}(\suvec{v}_T;\cdot)$,
  \[
  \mathcal{A}_T((\SYM\CURL\bvec{v},\btens{0}),(\btens{\tau},\btens{\mu}))
  = \mathcal{L}_{\bvec{V},T}(\suvec{v}_T;(\btens{\tau},\btens{\mu}))
  \qquad\forall(\btens{\tau},\btens{\mu})\in\tPoly{k-1}(T;\Symm)\times\cHoly{\ell_T+1}(T).
  \]
  Since problem \eqref{eq:serendipity.problem} is well-posed, this shows that $(\SYM\CURL\bvec{v},\btens{0})$ is its unique solution and, as a result, \eqref{eq:SV:polynomial.consistency} holds.
  \medskip\\
  \underline{(ii) \emph{Proof of \eqref{eq:EV:polynomial.consistency}.}}
  Set again $\suvec{v}_T\coloneq\sIV{k}\bvec{v}$.
  Starting from \eqref{eq:EPT}, using \eqref{eq:SV:polynomial.consistency} to write $\SV\suvec{v}_T = \SYM\CURL\bvec{v}$ along with the polynomial consistency of the trace $\svec{v}_{\ET} = \bvec{v}_{|\partial T}$, and concluding applying the integration by parts formula \eqref{eq:ibp.VT} to the right-hand side of the resulting expression, we have
  \[
  \int_T\EPT\suvec{v}_T\cdot\VROT\btens{\tau}
  = \int_T\bvec{v}\cdot\VROT\btens{\tau}
  \qquad\forall\btens{\tau}\in\cHoly{k-1}(T).
  \]
  Recalling that $\VROT:\cHoly{k-1}(T)\to\vPoly{k-2}(T)$ is an isomorphism, this shows that $\EPT\suvec{v}_T = \vlproj{k-2}{T}\bvec{v}$.
  Noticing that the other components of the local interpolator are not affected by the serendipity reduction process, \eqref{eq:EV:polynomial.consistency} follows.
  \\
  \underline{(iii) \emph{Proof of \eqref{eq:SS:polynomial.consistency}.}}
  It suffices to show that $(\btens{\upsilon},\btens{0})$ solves the problem defining $\SSigma\sISigma{k-1}\btens{\upsilon}$, i.e., \eqref{eq:serendipity.problem} with linear form $\mathcal{L}_T(\cdot) = \mathcal{L}_{\btens{\Sigma},T}(\sISigma{k-1}\btens{\upsilon};\cdot)$ given by \eqref{eq:LT.Sigma}.
  To this end, we use the alternative expression \eqref{eq:LT.Sigma:bis} of $\mathcal{L}_{\btens{\Sigma},T}(\sutens{\upsilon}_T;\cdot)$ based on the restriction of the operator $\DD{k-2}$ to $\sutens{\Sigma}_T^{k-1}$ resulting from the injection $\sutens{\Sigma}_T^{k-1}\hookrightarrow\utens{\Sigma}_T^{k-1}$.
  Since this operator only depends on the polynomial components of $\sutens{\Sigma}_T^{k-1}$ left unchanged by the serendipity reduction, by \cite[Eq.~(19)]{Di-Pietro.Droniou:22} it holds $\DD{k-2}\sISigma{k-1}\btens{\upsilon} = \DIV\VDIV\btens{\upsilon}$.
  Plugging this relation into \eqref{eq:LT.Sigma:bis} and recalling the definition \eqref{eq:sinterpolators} of $\sISigma{k-1}$, we obtain:
  For all $(\btens{\tau},\btens{\mu})\in\tPoly{k-1}(T;\Symm)\times\cHoly{\ell_T+1}(T)$,
  \[
  \begin{aligned}
    \mathcal{L}_{\btens{\Sigma},T}(\sISigma{k-1}\btens{\upsilon};\btens{\tau},\btens{\mu})
    &=
    h_T^4\int_T\DIV\VDIV\btens{\upsilon}~\DIV\VDIV\btens{\tau}
    \\
    &\quad
    + h_T\sum_{E\in\ET}\int_E\lproj{k-3}{E}(\btens{\upsilon}\normal_E\cdot\normal_E)~\lproj{k-3}{E}(\btens{\tau}\normal_E\cdot\normal_E)
    + h_T^3\sum_{E\in\ET}\int_E\cancel{\lproj{k-2}{E}}\dE\btens{\upsilon}~\dE\btens{\tau}
    \\
    &\quad
    + h_T^2\sum_{V\in\VT}\btens{\upsilon}(\bvec{x}_V):\btens{\tau}(\bvec{x}_V)
    + \int_T\cancel{\cHproj{\ell_T+1}}\btens{\upsilon}:\btens{\mu},
  \end{aligned}
  \]
  where the cancellation of the projectors is made possible by their definition.
  Comparing with the definition \eqref{eq:AT} of $\mathcal{A}_T$, we have thus proved that
  \[
  \mathcal{A}_T((\btens{\upsilon},\btens{0}),(\btens{\tau},\btens{\mu}))
  = \mathcal{L}_{\btens{\Sigma},T}(\sISigma{k-1}\btens{\upsilon};(\btens{\tau},\btens{\mu}))
  \qquad\forall(\btens{\tau},\btens{\mu})\in\tPoly{k-1}(T;\Symm)\times\cHoly{\ell_T+1}(T).
  \]
  By uniqueness of the solution to \eqref{eq:serendipity.problem}, this proves the assertion.
  \medskip\\
  \underline{(iv) \emph{Proof of \eqref{eq:ES:polynomial.consistency}.}}
  Immediate consequence of \eqref{eq:SS:polynomial.consistency} along with the definition \eqref{eq:ES} of $\ES{T}$.
\end{proof}

\begin{lemma}[Projections of extension and serendipity operators]
  It holds, for all $T\in\Th$,
  \begin{alignat}{4}\label{eq:vlproj.ell.T.EPT}
    \vlproj{\ell_T}{T}\EPT\suvec{v}_T &= \svec{v}_T
    &\qquad&
    \forall \suvec{v}_T\in\suvec{V}_T^k,
    \\ \label{eq:cHproj.k-1.Csym}
    \cHproj{k-1}\Csym{k-1}\EV{T}\suvec{v}_T &=
    \cHproj{k-1}\SV\suvec{v}_T,
    &\qquad&
    \forall \suvec{v}_T\in\suvec{V}_T^k,
    \\ \label{eq:cHproj.ellT+1.SSigma}
    \cHproj{\ell_T+1}\SSigma\sutens{\tau}_T
    &=\stens{\tau}_T^\compl
    &\qquad&
    \forall\sutens{\tau}_T\in\sutens{\Sigma}_T^{k-1}.
  \end{alignat}      
\end{lemma}
\begin{proof}
  \underline{(i) \emph{Proof of \eqref{eq:vlproj.ell.T.EPT}.}}
  For any $\btens{\mu}\in\cHoly{\ell_T+1}(T)$, taking tests functions of the form $(\btens{0},\btens{\mu})$ with $\btens{\mu}\in\cHoly{\ell_T+1}(T)$ in the problem defining $\SV$ (i.e., \eqref{eq:serendipity.problem} with $\mathcal{L}_T(\cdot) = \mathcal{L}_T(\suvec{v}_T;\cdot)$ given by \eqref{eq:LT.V}), it is inferred that
  \begin{equation}\label{eq:vlproj.ell.T.EPT:1}
    \int_T\SV\suvec{v}_T:\btens{\mu}
    = -\int_T\svec{v}_T\cdot\VROT\btens{\mu}
    + \sum_{E\in\ET}\omega_{TE}\int_E\svec{v}_{\ET}\cdot(\btens{\mu}\tangent_E).
  \end{equation}
  On the other hand, by definition \eqref{eq:EPT} of $\EPT$, and since $\btens{\mu}\in\cHoly{\ell_T + 1}(T) \subset \cHoly{k - 1}(T)$ (recall that $\ell_T + 1\le k - 1$ by \eqref{eq:ell.T}), we have
  \begin{equation}\label{eq:vlproj.ell.T.EPT:2}
    -\int_E\EPT\suvec{v}_T\cdot\VROT\btens{\mu}
    = \int_T\SV\suvec{v}_T:\btens{\mu}
    - \sum_{E\in\ET}\omega_{TE}\int_E\svec{v}_{\ET}\cdot(\btens{\mu}\tangent_E).
  \end{equation}
  Summing \eqref{eq:vlproj.ell.T.EPT:1} and \eqref{eq:vlproj.ell.T.EPT:2}, \eqref{eq:vlproj.ell.T.EPT} follows recalling that $\VROT:\cHoly{\ell_T+1}(T)\to\vPoly{\ell_T}(T)$ is an isomorphism.
  \medskip\\
  \underline{(ii) \emph{Proof of \eqref{eq:cHproj.k-1.Csym}.}}
  Using the definition \eqref{eq:CsymT} of $\Csym{k-1}$ for $\uvec{v}_T = \EV{T}\suvec{v}_T$ and recalling the definition \eqref{eq:EV} of $\EV{T}$, we can write, for any $\btens{\tau}\in\cHoly{k-1}(T)$, 
  \[
  \int_T\Csym{k-1}\EV{T}\suvec{v}_T:\btens{\tau}
  = -\int_E\EPT\suvec{v}_T\cdot\VROT\btens{\tau}
  + \sum_{E\in\ET}\omega_{TE}\int_E\svec{v}_{\ET}\cdot(\btens{\tau}\tangent_E)
  = \int_E\SV\suvec{v}_T:\btens{\tau},
  \]
 where the conclusion follows from the definition \eqref{eq:EPT} of $\EPT$.
 Then, \eqref{eq:cHproj.k-1.Csym} follows by definition of the $L^2$-orthogonal projector on $\cHoly{k-1}(T)$.
  \medskip\\
  \underline{(iii) \emph{Proof of \eqref{eq:cHproj.ellT+1.SSigma}.}}
  It suffices to take test functions of the form $(\btens{0},\btens{\mu})$ with $\btens{\mu}$ spanning $\cHoly{\ell_T+1}(T)$ in the problem defining $\SSigma$, that is \eqref{eq:serendipity.problem} with linear form $\mathcal{L}_T(\cdot) = \mathcal{L}_{\btens{\Sigma},T}(\sutens{\tau}_T;\cdot)$ given by \eqref{eq:LT.Sigma}.
\end{proof}

\subsection{Commutation property for the serendipity operators}

\begin{lemma}[Commutation property for the serendipity operators]
  Recalling that, according to \eqref{eq:s.operators}, $\suCsym{k-1} = \RS{T}\uCsym{k-1}\EV{T}$, it holds
  \begin{equation}\label{eq:commutation}
    \SSigma\suCsym{k-1}\suvec{v}_T = \SV\suvec{v}_T\qquad\forall\suvec{v}_T\in\suvec{V}_T^k,
  \end{equation}
  so that the following diagram commutes:
  \[
  \begin{tikzcd}
    \suvec{V}_T^k \arrow[r, "\SV"] \arrow[rd, swap, "\suCsym{k-1}"] & \tPoly{k-1}(T;\Symm) \\ 
    {} & \sutens{\Sigma}_T^{k-1} \arrow[u, swap, "\SSigma"] 
  \end{tikzcd}
  \]
\end{lemma}

\begin{proof}
  Let $\suvec{v}_T\in\suvec{V}_T^k$ and set $\uvec{v}_T\coloneq\EV{T}\suvec{v}_T$.
  Recalling \eqref{eq:EV}, we have $\bvec{v}_E = \svec{v}_E$ for all $E\in\ET$ and $(\bvec{v}_V,\btens{G}_{\bvec{v},V}) = (\svec{v}_V,\btens{G}_{\svec{v},V})$ for all $V\in\VT$.
  We next analyse the expression \eqref{eq:LT.Sigma} of $\mathcal{L}_{\btens{\Sigma},T}(\sutens{\upsilon}_T;\cdot)$ when $\sutens{\upsilon}_T \coloneq \suCsym{k-1}\suvec{v}_T =\RS{T}\uCsym{k-1}\uvec{v}_T$ with the aim of showing that
  \begin{equation}\label{eq:commutation:0}
    \mathcal{L}_{\btens{\Sigma},T}(\sutens{\upsilon}_T;(\btens{\tau},\btens{\mu}))
    = \mathcal{L}_{\bvec{V},T}(\suvec{v}_T;(\btens{\tau},\btens{\mu}))
    \qquad\forall(\btens{\tau},\btens{\mu})\in\tPoly{k-1}(T;\Symm)\times\cHoly{\ell_T+1}(T).
  \end{equation}
  The conclusion follows from this relation proceeding as in \cite[Lemma 20]{Di-Pietro.Droniou:22}.
  
  We start by observing that, for all $q\in\Poly{k-2}(T)$,
  \begin{equation}\label{eq:commutation:1}
    \begin{aligned}
      &\int_T\stens{\upsilon}_{\cvec{H},T}:\HESS q
      - \sum_{E\in\ET}\omega_{TE}\left(
      \int_E\ser{\upsilon}_E~\partial_{\normal_E}q - \int_E D_{\stens{\upsilon},E}~q   
      %\\
      %&\qquad- \sum_{E\in\ET}\omega_{TE}
      \sum_{V\in\VE}\omega_{EV}(\stens{\upsilon}_V\normal_E\cdot\tangent_E)~q(\bvec{x}_V)\right)  
      \\
      &\quad
      = \int_T\Hproj{k-4}\Csym{k-1}\uvec{v}_T:\HESS q
      \\
      &\qquad
      -\sum_{E\in\ET}\omega_{TE}\left(
      \int_E\lproj{k-3}{E}(\partial_{\tangent_E}\bvec{v}_{\ET}\cdot\normal_E)~\partial_{\normal_E} q
      - \int_E(\partial_{\tangent_E}^2\bvec{v}_{\ET}\cdot\tangent_E)~q
      \right)
      \\
      &\qquad
      - \sum_{E\in\ET}\omega_{TE}\sum_{V\in\VE}\omega_{EV}(\mathbb{C}\btens{G}_{\bvec{v},V}\normal_E\cdot\tangent_E)~q(\bvec{x}_V)
      \\
      &\quad
      = \int_T\DD{k-2}\uCsym{k-1}\uvec{v}_T~q = 0,
    \end{aligned}
  \end{equation}
  where the second equality follows from the definitions \eqref{eq:uCsymT} of $\uCsym{k-1}$ and \eqref{eq:DDT} of $\DD{k-2}$,
  while the conclusion is a consequence of the fact that \eqref{eq:discrete.complex} defines a complex.
  This implies that the terms in the first three lines of \eqref{eq:LT.Sigma} vanish since $\DIV\VDIV\btens{\tau}\in\Poly{k-3}(T)\subset\Poly{k-2}(T)$.
  Additionally, from property \eqref{eq:cHproj.k-1.Csym} it follows that $\stens{\upsilon}_{\ctens{H},T}^\compl = \cHproj{\ell_T+1}\uCsym{k-1}\EV{T}\suvec{v}_T = \cHproj{\ell_T+1}\SV\suvec{v}_T$.
  Hence, for all $\btens{\mu}\in\cHoly{\ell_T+1}(T)$,
  \begin{equation}\label{eq:commutation:2}
    \begin{aligned}
      \int_T\stens{\upsilon}_{\ctens{H},T}^\compl:\btens{\mu}
      = \int_T\SV\suvec{v}_T:\btens{\mu}
      = -\int_T\svec{v}_T\cdot\VROT\btens{\mu}
      + \sum_{E\in\ET}\omega_{TE}\int_E\svec{v}_E\cdot(\btens{\mu}\tangent_E),
    \end{aligned}
  \end{equation}
  where the conclusion follows from the definition of $\SV$.
  Plugging \eqref{eq:commutation:1}--\eqref{eq:commutation:2} into \eqref{eq:LT.Sigma} and comparing with \eqref{eq:LT.V} proves \eqref{eq:commutation:0}.
\end{proof}

\subsection{Homological properties of the serendipity DDR sequence}

\begin{theorem}[Homological properties of the serendipity DDR sequence]\label{thm:homological.properties}
  The following properties hold:
  \begin{compactenum}
  \item Complex properties:
    \begin{alignat}{4}\label{eq:complex.V}
      \EV{h}\RV{h}\uvec{v}_h &= \uvec{v}_h %% &\in \Image(\IV[h]{k})
      &\qquad&
      \forall\uvec{v}_h\in\Ker(\uCsym[h]{k-1}),
      \\ \label{eq:complex.S}
      \ES{h}\RS{h}\utens{\tau}_h - \utens{\tau}_h &\in \Image(\uCsym[h]{k-1})
      &\qquad&
      \forall \utens{\tau}_h\in\utens{\Sigma}_h^{k-1};
    \end{alignat}
  \item Cochain map properties for the reduction and extension maps:
    \begin{alignat}{4}
      \label{eq:cochain.EV}
      \EV{h}\sIV[h]{k}\bvec{v} &=
      \IV[h]{k}\bvec{v}
      &\qquad&
      \forall \bvec{v}\in\RT{1}(\Omega),
      \\ \label{eq:cochain.RS}
      \suCsym[h]{k-1}\RV{h}\uvec{v}_h &= \RS{h}\uCsym[h]{k-1}\uvec{v}_h 
      &\qquad&
      \forall \uvec{v}_h\in\uvec{V}_h^k,
      \\ \label{eq:cochain.ES}
      \ES{h}\suCsym[h]{k-1}\suvec{v}_h &= \uCsym[h]{k-1}\EV{h}\suvec{v}_h 
      &\qquad&
      \forall \suvec{v}_h\in\suvec{V}_h^k;
    \end{alignat}    
  \item Isomorphism properties for the cohomology groups:
    \begin{alignat}{4}\label{eq:isomorphism.V}
      \RV{h}\EV{h}\suvec{v}_h &= \suvec{v}_h
      &\qquad&
      \forall \suvec{v}_h\in\suvec{V}_h^k,%%\Ker(\suCsym[h]{k-1}),
      \\ \label{eq:isomorphism.S}
      \RS{h}\ES{h}\sutens{\tau}_h &= \sutens{\tau}_h
      &\qquad&
      \forall\sutens{\tau}_h\in\sutens{\Sigma}_h^{k-1}.
    \end{alignat}
  \end{compactenum}
  Hence, the cohomologies of the top and bottom complexes in \eqref{eq:discrete.complex} are isomorphic.
\end{theorem}

\begin{remark}[Homological properties]
  The respective role of the above properties is the following:
  the \emph{complex properties} ensure that the serendipity DDR sequence is a cochain complex;
  thanks to the \emph{cochain map properties}, the reduction and extension maps are cochain maps;
  finally, the \emph{isomorphism properties} guarantee that the cohomology groups of the DDR and serendipity DDR complexes are isomorphic.  
  We additionally notice, in passing, that:
  \begin{compactitem}
  \item It would suffice for property \eqref{eq:complex.S} to hold for all $\utens{\tau}_h\in\Ker(\DD[h]{k-2})$ to ensure that the serendipity DDR sequence is a cochain complex;
  \item The cochain property for $\RV{h}$ (i.e., $\RV{h}\IV[h]{k}\bvec{w} = \sIV[h]{k}\bvec{w}$ for all $\bvec{w}\in\RT{1}(\Omega)$), holds by definition \eqref{eq:s.operators} of $\sIV[h]{k}$, and is therefore not listed in point 2.;
  \item Property \eqref{eq:isomorphism.V} (resp., \eqref{eq:isomorphism.S}) could be restricted to $\suvec{v}_h\in\Ker(\suCsym[h]{k-1})$ (resp., $\sutens{\tau}_h\in\Ker(\sDD[h]{k-2})$) for the isomorphism in cohomology to hold.%
  \end{compactitem}
\end{remark}

\begin{proof}[Proof of Theorem \ref{thm:homological.properties}]
  The isomorphism between the cohomologies of the top and bottom complexes in \eqref{eq:discrete.complex} is a straightforward consequence of \cite[Proposition 2]{Di-Pietro.Droniou:21*1} once we prove properties \eqref{eq:complex.V}--\eqref{eq:isomorphism.S}, which we do next.
  \medskip\\
  \underline{(i) \emph{Proof of \eqref{eq:complex.V}.}}
  We notice that $\uCsym[h]{k-1}\uvec{v}_h = \utens{0}$ implies $\uCsym{k-1}\uvec{v}_T = \utens{0}$ for all $T\in\Th$.
  The exactness of the local DDR complex proved in \cite[Theorem 3]{Di-Pietro.Droniou:22*1} then implies, for any $T\in\Th$, the existence of $\bvec{w}_T\in\RT{1}(T)$ such that $\uvec{v}_T = \IV{k}\bvec{w}_T$.
  We can then write $\EV{T}\RV{T}\uvec{v}_T = \EV{T}\RV{T}\IV{k}\bvec{w}_T = \EV{T}\sIV{k}\bvec{w}_T = \IV{k}\bvec{w}_T$, where we have used the definition \eqref{eq:s.operators} of $\sIV{k}$ in the second step and the polynomial consistency property \eqref{eq:EV:polynomial.consistency} (after observing that $\bvec{w}_T\in\vPoly{k}(T;\Real^2)$) to conclude.
  \medskip\\ 
  \underline{(ii) \emph{Proof of \eqref{eq:complex.S}.}}
  Let $\utens{\tau}_h\in\utens{\Sigma}_h^{k-1}$ and set $\sutens{\tau}_h \coloneq \RS{h}\utens{\tau}_h$.
  The components of $\utens{\tau}_h$ and $\ES{h}\sutens{\tau}_h$ on the mesh vertices and edges, as well as on $\Holy{k-4}(T)$, $T\in\Th$, coincide by definition of the restriction and extension operators (see \eqref{eq:ES} and \eqref{eq:RS}).
  Since $\DD{k-2}$ only depends on these components (see \eqref{eq:DDT}), this implies
  %\begin{equation}\label{eq:DD.E.R.tau=DD.tau}
   $\DD[h]{k-2}\ES{h}\sutens{\tau}_h = \DD[h]{k-2}\utens{\tau}_h$,
  %\end{equation}
  i.e.
  \begin{equation}\label{eq:E.stau-tau}
    \ES{h}\sutens{\tau}_h - \utens{\tau}_h
    = \big(
    (\btens{0}, \cHproj{k-1}\SSigma\sutens{\tau}_T - \btens{\tau}_{\cvec{H},T}^\compl)_{T\in\Th},
    (0,0)_{E\in\Eh},
    (\btens{0})_{V\in\Vh}
    \big)\in\Ker(\DD[h]{k-2}).
  \end{equation}
  By exactness of the local DDR complex (see \cite[Theorem 3]{Di-Pietro.Droniou:22*1}), for all $T\in\Th$ there exists $\uvec{v}_T\in\uvec{V}_T^k$, defined up to the interpolate on $\uvec{V}_T^k$ of an element of $\RT{1}(T)$, such that $\ES{T}\sutens{\tau}_T - \utens{\tau}_T = \uCsym{k-1}\uvec{v}_T$ which additionally satisfies, by \eqref{eq:E.stau-tau},
  \[
  \begin{gathered}
    \text{   
      $\bvec{\pi}_{\cvec{P},E}^{k-3}\partial_{\tangent_E}(\bvec{v}_{\ET}\cdot\normal_E) = 0$ 
      and $\partial_{\tangent_E}^2(\bvec{v}_{\ET}\cdot\tangent_E) = 0$
      for all $E\in\ET$
    }
    \\
    \text{
      and $\mathbb{C}\btens{G}_{\bvec{v},V} = \btens{0}$ for all $V\in\VT$.
    }
  \end{gathered}
  \]
  Under these conditions, \cite[Point 1. of Theorem 3]{Di-Pietro.Droniou:22} yields the existence of $\bvec{w}_T\in\RT{1}(T)$ such that $\bvec{v}_{\ET} = \bvec{w}_{T|\partial T}$.
  Possibly up to the substitution $\uvec{v}_T\gets\uvec{v}_T - \IV{k}\bvec{w}_T$, we can therefore assume that $\bvec{v}_{\ET} = \bvec{0}$.
  Hence, the $\uvec{v}_T$, $T\in\Th$, can be patched together on internal edges to form an element of $\uvec{V}_h^k$.
  This concludes the proof of \eqref{eq:complex.S}.
  \medskip\\
  \underline{(iii) \emph{Proof of \eqref{eq:cochain.EV}.}}
  The cochain map property \eqref{eq:cochain.EV} for $\EV{h}$ immediately follows from \eqref{eq:EV:polynomial.consistency} applied to polynomials in $\RT{1}(T)\subset\vPoly{k}(T;\Real^2)$ for all $T\in\Th$.
  \medskip\\
  \underline{(iv) \emph{Proof of \eqref{eq:cochain.RS}.}}
  Let $\uvec{v}_h\in\uvec{V}_h^k$ and set, for the sake of brevity $\uvec{w}_h\coloneq\EV{h}\RV{h}\uvec{v}_h$.
  By \eqref{eq:s.operators}, $\suCsym[h]{k-1}\RV{h}\uvec{v}_h =  \RS{h}\uCsym[h]{k-1}\uvec{w}_h$.
  The components of $\uvec{w}_h$ and $\uvec{v}_h$ on the mesh vertices and edges coincide by definitions \eqref{eq:EV} of $\EV{h}$ and \eqref{eq:RV} of $\RV{h}$, hence so do the components of their discrete symmetric curls on the edges and vertices, as well as those on $\Holy{k-4}(T)$, $T\in\Th$ (notice that the first term in the right-hand side of \eqref{eq:CsymT} vanishes for $\btens{\tau}\in\Holy{k-4}(T)$ since $\VROT\HESS = \bvec{0}$).
  It only remains to prove the equality of the components on $\cHoly{\ell_T+1}(T)$, $T\in\Th$, which follows if we prove that:
  \begin{equation}\label{eq:cHproj.ellT+1.Csym}
    \cHproj{\ell_T+1}\Csym{k-1}\uvec{w}_T
    = \cHproj{\ell_T+1}\Csym{k-1}\uvec{v}_T   
    \quad \text{ for all $T\in\Th$}.
  \end{equation}
  Set $\suvec{v}_T \coloneq \RV{T}\uvec{v}_T$.
  By virtue of \eqref{eq:cHproj.k-1.Csym}, it suffices to prove that $\cHproj{\ell_T+1}\SV\suvec{v}_T = \cHproj{\ell_T+1}\Csym{k-1}\uvec{v}_T$.
  This relation can be established taking test functions of the form $(\btens{0}, \btens{\mu})$ with $\btens{\mu}\in\cHoly{\ell_T+1}(T)$ in the problem defining $\SV\uvec{v}_T$ (i.e., \eqref{eq:serendipity.problem} with linear form $\mathcal{L}_T(\cdot) = \mathcal{L}_{\bvec{V},T}(\suvec{v}_T;\cdot)$) to write
  \[
  \int_T\SV\suvec{v}_T:\btens{\mu}
  = -\int_T\cancel{\vlproj{\ell_T}{T}}\bvec{v}_T\cdot\VROT\btens{\mu}
  + \sum_{E\in\ET}\omega_{TE}\int_E\bvec{v}_T\cdot(\btens{\mu}\tangent_E)
  = \int_T\Csym{k-1}\uvec{v}_T:\btens{\mu},
  \]
  where we have used the fact that, by \eqref{eq:RV}, $\svec{v}_T = \vlproj{\ell_T}{T}\bvec{v}_T$ and $\svec{v}_{\ET} = \bvec{v}_{\ET}$ in the first step (and also cancelled the projector since $\VROT\btens{\mu}\in\vPoly{\ell_T}(T)$), while the conclusion follows from the definition \eqref{eq:CsymT} of $\Csym{k-1}$.
  This concludes the proof of \eqref{eq:cHproj.ellT+1.Csym} and, therefore, of \eqref{eq:cochain.RS}.
  \medskip\\
  \underline{(v) \emph{Proof of \eqref{eq:cochain.ES}.}}
  By \eqref{eq:s.operators}, \eqref{eq:cochain.ES} amounts to proving that $\ES{h}\RS{h}\uCsym[h]{k-1}\EV{h}\suvec{v}_h = \uCsym[h]{k-1}\EV{h}\suvec{v}_h$.
  Since $\ES{h}$ and $\RS{h}$ leave the components on mesh vertices, edges, as well as those on $\Holy{k-4}(T)$, $T\in\Th$, unaltered, the equality of this components in \eqref{eq:cochain.ES} is immediate.
  It only remains to prove the equality of the components on $\cHoly{k-1}(T)$, $T\in\Th$.
  To this purpose, it suffices to invoke \eqref{eq:commutation} and \eqref{eq:cHproj.k-1.Csym} to write:
  For all $T\in\Th$,
  \[
  \cHproj{k-1}\SSigma\suCsym{k-1}\suvec{v}_T
  = \cHproj{k-1}\SV\suvec{v}_T
  = \cHproj{k-1}\Csym{k-1}\EV{T}\suvec{v}_T.
  \]
  \\
  \underline{(vi) \emph{Proof of \eqref{eq:isomorphism.V} and \eqref{eq:isomorphism.S}.}}
  These relations are immediate consequences of, respectively, \eqref{eq:vlproj.ell.T.EPT} and \eqref{eq:cHproj.ellT+1.SSigma} along with the definitions \eqref{eq:RV} and \eqref{eq:RS} of the restrictions.%
\end{proof}

\subsection{Analytical properties of the serendipity complex}

Following \cite[Eq.~(2.3)]{Di-Pietro.Droniou:22}, for $\bullet\in\Th\cup\{h\}$, the discrete $L^2$-products and norms on $\suvec{V}_\bullet^k$ and $\sutens{\Sigma}_\bullet^{k-1}$ are defined setting, for all $\suvec{w}_\bullet,\,\suvec{v}_\bullet\in\suvec{V}_\bullet^k$ and $\sutens{\upsilon}_\bullet,\,\sutens{\tau}_\bullet\in\sutens{\Sigma}_\bullet^{k-1}$,
\begin{alignat}{4} \label{eq:l2prod.norm.sV}
  (\suvec{w}_\bullet,\suvec{v}_\bullet)_{\svec{V},\bullet}
  &\coloneq (\EV{\bullet}\suvec{w}_\bullet,\EV{\bullet}\suvec{v}_\bullet)_{\bvec{V},\bullet}
  &\quad\text{and}\quad
  \norm[\svec{V},\bullet]{\suvec{v}_\bullet}
  &\coloneq\norm[\bvec{V},\bullet]{\EV{\bullet}\suvec{v}_\bullet}, %% (\suvec{v}_\bullet,\suvec{v}_\bullet)_{\svec{V},\bullet}^{\nicefrac12},
  \\ \label{eq:l2prod.norm.sSigma}
  (\sutens{\upsilon}_\bullet, \sutens{\tau}_\bullet)_{\stens{\Sigma},\bullet}
  &\coloneq (\ES{\bullet}\sutens{\upsilon}_\bullet, \ES{\bullet}\sutens{\tau}_\bullet)_{\btens{\Sigma},\bullet}
  &\quad\text{and}\quad
  \norm[\stens{\Sigma},\bullet]{\sutens{\tau}_\bullet}
  &\coloneq\norm[\btens{\Sigma},\bullet]{\ES{\bullet}\sutens{\upsilon}_\bullet}.
  %%(\sutens{\tau}_\bullet,\sutens{\tau}_\bullet)_{\stens{\Sigma},\bullet}^{\nicefrac12}.
\end{alignat}

\begin{lemma}[Equivalence of norms on  $\suvec{V}_T^k$]\label{lem:norm.equiv.sVT}
  It holds $\norm[\svec{V},T]{{\cdot}} \simeq \tnorm[\bvec{V},T]{{\cdot}}$ on $\suvec{V}_T^k$.
\end{lemma}

\begin{proof}
  For all $\suvec{v}_T \in \suvec{V}_T^k$, we have
  \[
  \norm[\svec{V},T]{\suvec{v}_T}
  = \norm[\bvec{V},T]{\EV{T}{\suvec{v}_T}}
  \lesssim \tnorm[\bvec{V},T]{\EV{T}{\suvec{v}_T}}
  \lesssim\tnorm[\bvec{V},T]{\suvec{v}_T},
  \]
  where the first inequality comes from the norm equivalence \eqref{eq:norm.equivalence}, while the conclusion is \eqref{eq:EVT:continuity}.
  
  To prove the converse inequality, we use \eqref{eq:vlproj.ell.T.EPT} to write:
  \begin{equation*}
    \begin{aligned}
      \tnorm[\bvec{V},T]{\suvec{v}_T}^2 &=
      \norm[\bvec{L}^2(T;\Real^2)]{\vlproj{\ell_T}{T}\EPT\suvec{v}_T}^2 + \sum_{E\in\ET} h_T\norm[\bvec{L}^2(E;\Real^2)]{\svec{v}_E}^2
      + \sum_{V\in\VT} \left(
      h_T^2 |\svec{v}_V|^2 + h_T^4 |\btens{G}_{\svec{v},V}|^2
      \right)\\
      &\le
      \norm[\bvec{L}^2(T;\Real^2)]{\EPT\suvec{v}_T}^2+ \sum_{E\in\ET} h_T\norm[\bvec{L}^2(E;\Real^2)]{\svec{v}_E}^2
      + \sum_{V\in\VT} \left(
      h_T^2 |\svec{v}_V|^2 + h_T^4 |\btens{G}_{\svec{v},V}|^2
      \right)
      \\
      &=
      \tnorm[\bvec{V},T]{\EV{T}{\suvec{v}_T}}^2,
    \end{aligned}
  \end{equation*}
  where the inequality follows from the $L^2$-boundedness of $\vlproj{\ell_T}{T}$, while the conclusion is an immediate consequence of the definitions \eqref{eq:tnorm.V} of $\tnorm[\bvec{V},T]{{\cdot}}$ and \eqref{eq:EV} of $\EV{T}$.
  We then continue with the equivalence of norms \eqref{eq:norm.equivalence} and with \eqref{eq:l2prod.norm.sV} to write $\tnorm[\bvec{V},T]{\EV{T}{\suvec{v}_T}}^ 2 \lesssim \norm[\bvec{V},T]{\EV{T}{\suvec{v}_T}}^ 2= \norm[\svec{V},T]{\suvec{v}_T}^2$.
\end{proof}
\begin{remark}[Equivalence of norms on $\sutens{\Sigma}_T^{k-1}$]
  The uniform equivalence of $\norm[\stens{\Sigma},T]{{\cdot}}$ defined in \eqref{eq:l2prod.norm.sSigma} and $\tnorm[\btens{\Sigma},T]{{\cdot}}$ can be established in a similar way.
  Since this result is not needed in what follows, the details are left to the reader.
\end{remark}

\begin{theorem}[Analytical properties of the serendipity DDR complex]\label{thm:analytical.properties}
  The following properties hold:
  \begin{compactenum}
  \item \emph{Continuity of the reductions:}
    \begin{alignat}{4}\label{eq:RV:continuity}
      \norm[\svec{V},h]{\RV{h}\uvec{v}_h}
      &\lesssim\norm[\bvec{V},h]{\uvec{v}_h}
      &\qquad&\forall\uvec{v}_h\in\uvec{V}_h^k,
      \\ \label{eq:RS:continuity}
      \norm[\stens{\Sigma},h]{\RS{h}\utens{\tau}_h}
      &\lesssim\norm[\btens{\Sigma},h]{\utens{\tau}_h}
      &\qquad&\forall\utens{\tau}_h\in\utens{\Sigma}_h^{k-1};
    \end{alignat}
  \item \emph{Polynomial consistency:} For all $T\in\Th$,
    \begin{alignat}{4}\label{eq:R.E.V:polynomial.consistency}
      \EV{T}\RV{T}\IV{k}\bvec{v} &= \bvec{v}
      &\qquad&\forall\bvec{v}\in\vPoly{k}(T;\Real^2),
      \\ \label{eq:R.E.S:polynomial.consistency}
      \ES{T}\RS{T}\ISigma{k-1}\btens{\tau} &= \btens{\tau}
      &\qquad&\forall\btens{\tau}\in\tPoly{k-1}(T;\Symm).
    \end{alignat}
  \end{compactenum}
  Hence, Lemmas \ref{lem:poincare}, \ref{lem:consistency}, and \ref{lem:adjoint.consistency} hold with $(\uvec{V}_T^k,\utens{\Sigma}_T^{k-1})$ replaced by $(\suvec{V}_T^k,\sutens{\Sigma}_T^{k-1})$.
\end{theorem}

\begin{proof}
  The fact that Lemmas \ref{lem:poincare}, \ref{lem:consistency}, and \ref{lem:adjoint.consistency} hold with $(\uvec{V}_T^k,\utens{\Sigma}_T^{k-1})$ replaced by $(\suvec{V}_T^k,\sutens{\Sigma}_T^{k-1})$ is a consequence of Theorem \ref{thm:homological.properties} along with the continuity of the interpolators \eqref{eq:I:boundedness} and \cite[Propositions 4--9]{Di-Pietro.Droniou:22} once properties \eqref{eq:RV:continuity}--\eqref{eq:R.E.S:polynomial.consistency} have been proved.
  We therefore turn to the latter.
  \medskip\\
  \underline{(i) \emph{Proof of \eqref{eq:RV:continuity} and \eqref{eq:RS:continuity}.}}
  Using the norm equivalence in Lemma \ref{lem:norm.equiv.sVT} and the definitions \eqref{eq:tnorm.V} of the component norm $\tnorm[\bvec{V},T]{{\cdot}}$ and \eqref{eq:RV} of $\RV{T}\uvec{v}_T$, we infer
  \begin{equation*}%\label{eq:RV:continuity'}
    \begin{aligned}
      \norm[\svec{V},T]{\RV{T}\uvec{v}_T}^2 
      &\lesssim
      \tnorm[\bvec{V},T]{\RV{T}\uvec{v}_T}^2 \\
      &=
      \norm[\bvec{L}^2(T;\Real^2)] {\vlproj{\ell_T}{T}\bvec{v}_T}^2
      + \sum_{E\in\ET}h_T\norm[\bvec{L}^2(E;\Real^2)]{\bvec{v}_E}^2
      +\sum_{V\in\VT} \left( h_T^2  |\bvec{v}_V|^2+h_T^4|\btens{G}_{\bvec{v},V})|^2\right)
      \\
      &\lesssim
      \norm[\bvec{L}^2(T;\Real^2)] {\bvec{v}_T}^2
      + \sum_{E\in\ET}h_T\norm[\bvec{L}^2(E;\Real^2)]{\bvec{v}_E}^2
      + \sum_{V\in\VT} \left( h_T^2  |\bvec{v}_V|^2+h_T^4|\btens{G}_{\bvec{v},V})|^2\right)
    \end{aligned}
  \end{equation*}
  where the second line results from the $L^2$-boundedness of ${\vlproj{\ell_T}{T}}$.
  Noticing that the expression in the last line is precisely $\tnorm[\bvec{V},T]{\uvec{v}_T}^2$ and invoking the uniform norm equivalence \eqref{eq:norm.equivalence} with $\bullet=T$ concludes the proof of \eqref{eq:RV:continuity}.
  The proof of  \eqref{eq:RV:continuity} is similar and we omit the details for the sake of conciseness.
  \medskip\\
  \underline{(ii) \emph{Proof of \eqref{eq:R.E.V:polynomial.consistency} and \eqref{eq:R.E.S:polynomial.consistency}.}}
  Recalling the definition \eqref{eq:sinterpolators} of the interpolators on the serendipity spaces, properties \eqref{eq:R.E.V:polynomial.consistency} and \eqref{eq:R.E.S:polynomial.consistency} are nothing but \eqref{eq:EV:polynomial.consistency} and \eqref{eq:ES:polynomial.consistency}, respectively.
\end{proof}

\appendix
\section{Poincaré--Korn type inequalities in hybrid spaces}\label{appendix}

The proof of the functional inequality for hybrid vector fields that is used to establish point (i) of Lemma \ref{lem:poincare} (see Section \ref{sec:poincare:proof}) is presented below. First, we introduce some additional notations concerning tensor calculus in three dimensions. For a bounded, Lipschitz domain $D\subset\Real^3$ and for a sufficiently regular tensor field $\btens{P}: D\to\Real^{3\times3}$, we define
$$
 \TCURL\btens{P}\coloneq\begin{pmatrix}
    \partial_2 P_{13} - \partial_3 P_{12} &  \partial_3 P_{11} - \partial_1 P_{13} &  \partial_1 P_{12} - \partial_2 P_{11} \\
    \partial_2 P_{23} - \partial_3 P_{22} &  \partial_3 P_{21} - \partial_1 P_{23} &  \partial_1 P_{22} - \partial_2 P_{21} \\
    \partial_2 P_{33} - \partial_3 P_{32} &  \partial_3 P_{31} - \partial_1 P_{33} &  \partial_1 P_{32} - \partial_2 P_{31} 
    \end{pmatrix}.
$$
For later use, we also introduce the space $\boldsymbol{\mathcal{RM}}_3 \coloneq \{ \bvec{a}\times\bvec{x} + \bvec{b} \st \bvec{a},\bvec{b}\in\Real^3 \}$ of three-dimensional rigid-body motions and the operator $\ANTI:\Real^3\to\Real^{3\times3}$ given by
$$
\ANTI\bvec{a} \coloneq\begin{pmatrix}
    0 & - a_3 &  a_2\\
    a_3&  0 &  -a_1 \\
    -a_2&  a_1 &  0
    \end{pmatrix}\quad
    \forall \bvec{a}\in\Real^3.
$$

The discrete functional inequalities below hinge on \cite[Theorem 3.3]{Lewintan.Muller.ea:21}, which the authors refer to as \emph{incompatible Korn type inequality for $L^p$-regular tensor fields}. For the sake of clarity, we recall the statement of this key result. 
\begin{lemma}[Incompatible Korn type inequality]
  \label{lem:incomp_korn}
  Let $D\subset\Real^3$ be a bounded, Lipschitz domain and let $p\in (1,\infty)$. Then, there exists $C_{\rm IK}>0$ depending only on $D$ and $p$ such that, for all $\btens{P}\in L^p(D;\Real^{3\times3})$,
  \begin{equation}\label{eq:incomp_korn}
    \inf_{\bvec{w}\in \boldsymbol{\mathcal{RM}}_3}\norm[\btens{L}^p(D;\Real^{3\times3})]{\btens{P} - \ANTI\bvec{w}} \le 
    C_{\rm IK}\left(
    \norm[\btens{L}^p(D;\Real^{3\times3})]{\SYM\btens{P}}
    + \norm[\btens{W}^{-1,p}(D;\Real^{3\times3})]{\SYM\TCURL\btens{P}}
    \right).
  \end{equation}
\end{lemma}
It has been observed in \cite{Neff.Pauly.ea:15,Lewintan.Muller.ea:21} that the previous result can be seen as a generalisation of both the Poincar\'e--Wirtinger and Korn's second inequalities. In the following Proposition, we apply Lemma \ref{lem:incomp_korn} to some particular cases in which the tensor field $\btens{P}$ is skew-symmetric and assuming $p=2$.
\begin{proposition}[Poincar\'e--Korn inequalities for $L^2$-regular vector fields]
  \label{prop:poin-korn_L2}
  Let $D\subset\Real^n$, with $n\in\{2,3\}$ be a bounded, Lipschitz domain. Then, the following inequalities hold:
  \begin{alignat}{2}\label{eq:poincare}
    \inf_{\overline{\bvec{u}}\in\cvec{P}^0(D;\Real^d)}\norm[\bvec{L}^2(D;\Real^d)]{\bvec{u} - \overline{\bvec{u}}}
    &\lesssim 
    C_{\rm IK}\norm[\btens{H}^{-1}(D;\Real^{n\times d})]{\GRAD\bvec{u}}
    &\quad& \forall \bvec{u}\in \bvec{L}^2(D;\Real^d) \text{ with } 1\le d\le n;
    \\ \label{eq:korn:sym.curl}
    \inf_{\overline{\bvec{v}}\in\cvec{RT}^1(D)} \norm[\bvec{L}^2(D;\Real^2)]{\bvec{v} - \overline{\bvec{v}}}
    &\lesssim 
    C_{\rm IK}\norm[\btens{H}^{-1}(D;\Real^{2\times2})]{\SYM\CURL\bvec{v}}
    &\quad& \forall \bvec{v}\in \bvec{L}^2(D;\Real^2)  \text{ with } n=2;
    \\\label{eq:korn.sym.grad}
    \displaystyle\inf_{\overline{\bvec{w}}\in\cvec{RM}_d}\norm[\btens{L}^2(D;\Real^d)]{\bvec{w} - \overline{\bvec{w}}}
    &\lesssim C_{\rm IK}\norm[\btens{H}^{-1}(D;\Real^{d\times d})]{\SYM\GRAD\bvec{w}}
    &\quad& \forall \bvec{w}\in \btens{L}^2(D;\Real^d) \text{ with } d=n.
  \end{alignat}
\end{proposition}
\begin{proof}
  In order to establish \eqref{eq:poincare} for $1\le d\le n \le 3$, it suffices to consider the case $n=3$ and $d=1$. Hence, we let $u:D\to\Real$ and apply Lemma \ref{lem:incomp_korn} with $\btens{P}$ such that $P_{3,2}= - P_{2,3}= u$ and all the other components set to zero. Therefore, we clearly have $\SYM\btens{P} = \btens{0}$ and
  $$
  \inf_{\bvec{w}\in \boldsymbol{\mathcal{RM}}_3}\norm[\btens{L}^p(D;\Real^{3\times3})]{\btens{P} - \ANTI\bvec{w}}
  = \inf_{\overline{u}\in\Real}\norm[\btens{L}^p(D;\Real^{3\times3})]{\btens{P} - \ANTI(\overline{u},0,0)}
  = \sqrt{2} \inf_{\overline{u}\in\Real}\norm[L^2(D)]{u - \overline{u}}.
  $$
  Moreover, it is observed that
  $$
  \SYM\TCURL\btens{P} =\frac12 \begin{pmatrix}
    0 & -\partial_2 u &  -\partial_3 u\\
    -\partial_2 u&  2\partial_1 u &  0 \\
    -\partial_3 u&  0 &  2\partial_1 u
  \end{pmatrix} \quad\Longrightarrow\;
  \norm[\btens{H}^{-1}(D;\Real^{3\times3})]{\SYM\TCURL\btens{P}}
  \le 2 \norm[\btens{H}^{-1}(D;\Real^{3})]{\GRAD u}.
  $$
  As a result, we get the conclusion.%
  \smallskip
  
  We now proceed with the proof of \eqref{eq:korn:sym.curl}.
  We let $n=2$, $\bvec{v}\in \bvec{L}^2(D;\Real^2)$, and define a skew-symmetric tensor field $\btens{P}$ such that
  $$
  \btens{P} = \begin{pmatrix}
    0 & 0& v_1\\
    0 & 0 & v_2\\
    -v_1 & v_2 & 0
  \end{pmatrix}\quad\Longrightarrow\;
  \SYM\TCURL\btens{P} =
  %\frac12 \begin{pmatrix}
  %2 \partial_2 v_1 &\partial_2 v_2 - \partial_1 v_1 & 0\\
  %\partial_2 v_2 - \partial_1 v_1 & -2 \partial_1 v_2 & 0\\
  %0 & 0 & 2(\partial_2 v_1- \partial_1 v_2)
  %\end{pmatrix}.
  \begin{pmatrix}
    \SYM\CURL \bvec{v} & \bvec{0}\\
    \bvec{0} & \ROT \bvec{v}
  \end{pmatrix} = 
  \begin{pmatrix}
    \SYM\CURL \bvec{v} & \bvec{0}\\
    \bvec{0} & \TR(\SYM\CURL \bvec{v})
  \end{pmatrix}.
  $$
  Therefore, it is inferred that $\norm[\btens{H}^{-1}(D;\Real^{3\times3})]{\SYM\TCURL\btens{P}}\le 2 \norm[\btens{H}^{-1}(D;\Real^{3})]{\SYM\CURL \bvec{v}}$.
  Additionally, since $\bvec{v}$ does not depend on $x_3$ and due to the position of the non-zero entries in $\btens{P}$, it is readily inferred that
  $$
  \inf_{\bvec{w}\in \boldsymbol{\mathcal{RM}}_3}\norm[\btens{L}^p(D;\Real^{3\times3})]{\btens{P} - \ANTI\bvec{w}}
  =  \inf_{\overline{\bvec{v}}\in\cvec{RT}^1(D)}\norm[\btens{L}^p(D;\Real^{3\times3})]{\btens{P} - \ANTI (\overline{\bvec{v}},0)}
  = \sqrt{2} \inf_{\overline{\bvec{v}}\in\cvec{RT}^1(D)} \norm[\bvec{L}^2(D;\Real^2)]{\bvec{v} - \overline{\bvec{v}}}.
  $$
  Hence, the conclusion follows again by using \eqref{eq:incomp_korn}.
  \smallskip
  
  The proof of \eqref{eq:korn:sym.curl} is obtained with similar arguments by using Lemma \ref{lem:incomp_korn} with
  $$
  \btens{P} = \begin{pmatrix}
    0 & 0& w_2\\
    0 & 0 & -w_1\\
    -w_2 & w_1 & 0
  \end{pmatrix}\quad\text{and }\; 
  \btens{P} = \begin{pmatrix}
    0 & -w_3& w_2\\
    w_3 & 0 & -w_1\\
    -w_2 & w_1 & 0
  \end{pmatrix},
  $$
  for the case $n=d=2$ and $n=d=3$, respectively.
\end{proof}
We are now ready to establish the main result of this Section. For the sake of simplicity, we detail the result only for the two dimensional case, but we refer to Remark \ref{rem:generalise} for some possible generalisations.
\begin{proposition}[Poincar\'e--Korn inequalities for hybrid vector fields]
\label{prop:Poin-Korn_hybrid}
Let 
$$
 \uvec{U}_h^k\coloneq\Big\{
  \begin{aligned}[t]
    &\uvec{u}_h=
    \big(
    (\bvec{u}_T)_{T\in\Th},
    (\bvec{u}_E)_{E\in\Eh},
    \big)\st
    \text{$\bvec{u}_T\in\vPoly{k}(T;\Real^2)\ \forall T\in\Th$}, \;
    \text{$\bvec{u}_E\in\vPoly{k}(E;\Real^2)\  \forall E\in\Eh$}
    \Big\}
  \end{aligned}
  $$
  and, for all $\uvec{u}_h\in\uvec{U}_h^k$, denote by $\bvec{u}_h$ the piecewise polynomial field on $\Th$ such that $(\bvec{u}_h)_{|T} \coloneq \bvec{u}_T$ for all $T\in\Th$.
  Then, there is a constant $C_{\rm PK}>0$, only depending on $\Omega$ and the mesh regularity parameter, such that
  \begin{enumerate}
    \item For all $\uvec{u}_h\in \uvec{U}_h^k$ satisfying $\int_\Omega\bvec{u}_h= \bvec{0}$,
  \begin{equation}
    \label{eq:hho-Poin}
    \norm[\bvec{L}^2(\Omega;\Real^2)]{\bvec{u}_h}^2 \le C_{\rm PK} \sum_{T\in\Th}\left(\norm[\bvec{L}^2(T;\Real^{2\times2})]{\GRAD\bvec{u}_T}^2 
                                                                                    +  \sum_{E\in\ET}h_T^{-1}\norm[\bvec{L}^2(E;\Real^2)]{\bvec{u}_T - \bvec{u}_E}^2\right);
  \end{equation}
  \item For all $\uvec{u}_h\in \uvec{U}_h^k$ satisfying $\int_\Omega\bvec{u}_h\cdot\bvec{w} = 0$ for all $\bvec{w}\in\RT{1}(\Omega)$,
  \begin{equation}
    \label{eq:hho-SymCurl}
    \norm[\bvec{L}^2(\Omega;\Real^2)]{\bvec{u}_h}^2 \le C_{\rm PK} \sum_{T\in\Th}\left(\norm[\bvec{L}^2(T;\Real^{2\times2})]{\SYM\CURL\bvec{u}_T}^2 
                                                                                    +  \sum_{E\in\ET}h_T^{-1}\norm[\bvec{L}^2(E;\Real^2)]{\bvec{u}_T - \bvec{u}_E}^2\right);
  \end{equation}
    \item For all $\uvec{u}_h\in \uvec{U}_h^k$ satisfying $\int_\Omega\bvec{u}_h\cdot\bvec{w} = 0$ for all $\bvec{w}\in\ctens{RM}_2$,
  \begin{equation}
    \label{eq:hho-Korn}
    \norm[\bvec{L}^2(\Omega;\Real^2)]{\bvec{u}_h}^2 \le C_{\rm PK} \sum_{T\in\Th}\left(\norm[\bvec{L}^2(T;\Real^{2\times2})]{\SYM\GRAD\bvec{u}_T}^2 
                                                                                    +  \sum_{E\in\ET}h_T^{-1}\norm[\bvec{L}^2(E;\Real^2)]{\bvec{u}_T - \bvec{u}_E}^2\right).
  \end{equation}
  \end{enumerate}
\end{proposition}
\begin{proof}  
  We only detail the proof of \eqref{eq:hho-SymCurl}, which is used in the proof of Lemma \ref{lem:poincare}, since \eqref{eq:hho-Poin} and \eqref{eq:hho-Korn} can be obtained by reasoning in a similar way.
  Let $\uvec{u}_h\in \uvec{U}_h^k$ and observe  that the condition $\int_\Omega\bvec{u}_h\cdot\bvec{w} = 0$ for all $\bvec{w}\in\RT{1}(\Omega)$ implies
  $$
  \inf_{\overline{\bvec{v}}\in\cvec{RT}^1(\Omega)} \norm[\bvec{L}^2(\Omega;\Real^2)]{\bvec{u}_h - \overline{\bvec{v}}} =  \norm[\bvec{L}^2(\Omega;\Real^2)]{\bvec{u}_h}.
  $$
  Therefore, applying the second inequality in Proposition \ref{prop:poin-korn_L2}, it follows that
  $$
  \begin{aligned}
    \norm[\bvec{L}^2(\Omega;\Real^2)]{\bvec{u}_h} &\lesssim
    \norm[\btens{H}^{-1}(\Omega;\Real^{2\times 2})]{\SYM\CURL\bvec{u}_h}
    =\sup_{\btens{\eta}\in \btens{H}^1_0(\Omega;\Symm),\,\norm[\btens{H}^1(\Omega;\Real^{2\times2})]{\btens{\eta}}=1} \int_\Omega \bvec{u}_h \cdot \left(\VROT\btens{\eta}\right)
    \\
    &=
    \sup_{\btens{\eta}\in \btens{H}^1_0(\Omega;\Symm),\,\norm[\btens{H}^1(\Omega;\Real^{2\times2})]{\btens{\eta}}=1}
    \sum_{T\in\Th}\left(-\int_T\CURL\bvec{u}_T:\btens{\eta}  
    + \sum_{E\in\ET}\omega_{TE}\int_{E}\bvec{u}_T\cdot(\btens{\eta} \,\bvec{t}_E)\right)
    \\
    &=
    \sup_{\btens{\eta}\in \btens{H}^1_0(\Omega;\Symm),\,\norm[\btens{H}^1(\Omega;\Real^{2\times2})]{\btens{\eta}}=1}
    \sum_{T\in\Th}\left(-\int_T\SYM\CURL\bvec{u}_T:\btens{\eta}
    + \sum_{E\in\ET}\omega_{TE}\int_{E}(\bvec{u}_T - \bvec{u}_E)\cdot\btens{\eta} \,\bvec{t}_E \right),
  \end{aligned}
  $$
  where we have integrated by parts element by element and used the fact that $\btens{\eta}$ has continuous tangential traces across interedges and vanishing tangential traces on the boundary in order to insert $\bvec{u}_E$ into the boundary term. 
  Applying a Cauchy--Schwarz inequality on the integrals and invoking a discrete Cauchy--Schwarz inequality on the sum over 
  $T\in\Th$, we infer that
  $$
  \begin{aligned}
    \norm[\bvec{L}^2(\Omega;\Real^2)]{\bvec{u}_h}
    %     &\le\sup_{ \norm[\btens{H}^1_0(\Omega;\Symm)]{\btens{\eta}}=1} \sum_{T\in\Th}\left(
    %     \norm[\btens{L}^2(T;\Real^{2\times2})]{\SYM\CURL\bvec{u}_T} \norm[\btens{L}^2(T;\Real^{2\times2})]{\btens{\eta}} \hspace{-0.5mm}
    %    + \sum_{E\in\ET} \norm[\btens{L}^2(E;\Real^2)]{\bvec{u}_T - \bvec{u}_{E}} \norm[\btens{L}^2(E;\Real^2)]{\btens{\eta}\, \bvec{t}_E} \right)
    %    \\
    &\lesssim
    \sup_{\btens{\eta}\in \btens{H}^1_0(\Omega;\Symm),\,\norm[\btens{H}^1(\Omega;\Real^{2\times2})]{\btens{\eta}}=1}
    \left(\sum_{T\in\Th}\norm[\btens{L}^2(T;\Real^{2\times2})]{\SYM\CURL\bvec{u}_T}^2 \right)^{\frac12}
    \norm[\btens{L}^2(\Omega;\Real^{2\times2})]{\btens{\eta}} 
    \\
    &\qquad
    +\left(\sum_{T\in\Th}\sum_{E\in\ET}h_T^{-1}\norm[\btens{L}^2(E;\Real^2)]{\bvec{u}_T - \bvec{u}_E}^2\right)^{\frac12}
    \left(\sum_{T\in\Th}\sum_{E\in\ET}h_T \norm[\btens{L}^2(E;\Real^2)]{\btens{\eta} \,\bvec{t}_E} \right)^{\frac12}
    \\
    &\lesssim 
    \left[
      \sum_{T\in\Th}\left(
      \norm[\btens{L}^2(T;\Real^{2\times2})]{\SYM\CURL\bvec{u}_T}^2
      + \sum_{E\in\ET} \hspace{-1mm} h_T^{-1}\norm[\bvec{L}^2(E;\Real^2)]{\bvec{u}_T - \bvec{u}_{E}}^2
      \right)
      \right]^{\frac12}
    \\
    &\qquad\times\cancelto{1}{%
      \sup_{\btens{\eta}\in \btens{H}^1_0(\Omega;\Symm),\,\norm[\btens{H}^1(\Omega;\Real^{2\times2})]{\btens{\eta}}=1}
      \norm[\btens{H}^1(\Omega;\Symm)]{\btens{\eta}}
    },
  \end{aligned}
  $$
  where, in the second inequality, we have used the continuous trace inequality \cite[Lemma 1.31]{Di-Pietro.Droniou:20}. 
\end{proof}
\begin{remark}[Generalisations]
  \label{rem:generalise}
  The results of Proposition \ref{prop:Poin-Korn_hybrid} admit several extensions that we have decided not to include for the sake of brevity.
  First, \eqref{eq:hho-Poin} and \eqref{eq:hho-Korn} can also be established in the three-dimensional case simply by replacing the interedges with interfaces.
  Second, since the starting argument given by Lemma \ref{lem:incomp_korn} holds for all Lebesgue indices $p\in(1,\infty)$, we can generalise the discrete Poincaré--Korn inequalities to the Banach setting.
  The main modification required in the proof consists in replacing Cauchy--Schwarz inequalities with suitable versions of H\"older inequalities.
  Finally, we notice that in the proof of Proposition \ref{prop:Poin-Korn_hybrid} we are not using any inverse inequality requiring the hybrid vector fields to be polynomials. Thus, the previous Poincar\'e--Korn inequalities can be extended to vector fields with piecewise Sobolev regularity.
\end{remark}
%------------------------------------------------------------------------------%
% Acknowledgements
%------------------------------------------------------------------------------%

\section*{Acknowledgments}
Michele Botti acknowledges funding from the European Union’s Horizon 2020 research and innovation programme under the Marie Sk\l odowska-Curie grant agreement No. 896616 ``PDGeoFF''.
Daniele Di Pietro acknowledges the partial support of \emph{Agence Nationale de la Recherche} through grants ANR-20-MRS2-0004 ``NEMESIS'' and ANR-16-IDEX-0006 ``RHAMNUS''.

%------------------------------------------------------------------------------%
% Bibliography
%------------------------------------------------------------------------------%

\printbibliography

\end{document}